\documentclass[a4paper, 11pt]{article}
\usepackage[utf8]{inputenc}     
\usepackage[T1]{fontenc}
\usepackage{lmodern}
\usepackage{graphicx}
\usepackage[english]{babel}
\usepackage{amsmath, amsfonts, amssymb,amsthm}  
\usepackage{textcomp}
\usepackage{mathrsfs}
\usepackage{float}
\usepackage{subfigure}
\usepackage{url}

\usepackage{tikz} 
\usetikzlibrary{arrows}
\usetikzlibrary{decorations}

\theoremstyle{plain}
\newtheorem{theorem}{Theorem}[section]
\newtheorem{proposition}[theorem]{Proposition}
\newtheorem{lemma}[theorem]{Lemma}

\newtheorem{remark}[theorem]{Remark}

\theoremstyle{definition}
\newtheorem{definition}[]{Definition}
\newtheorem{example}[theorem]{Example}

\newcommand{\im}{\mathrm{i}}
\newcommand{\bydef}{\,\stackrel{\mbox{\tiny\textnormal{\raisebox{0ex}[0ex][0ex]{def}}}}{=}\,}

\newcommand{\ellnu}{\ell_{\nu}^1}
\newcommand{\R}{\mathbb{R}}

\begin{document}
 \title{Computer assisted proof of homoclinic chaos \\ 
 in the spatial equilateral restricted four body problem
}
\author{
Maxime Murray  \footnote{Florida Atlantic University, 777 Glades Rd.,
Boca Raton, FL 33431, USA. mmurray2016@fau.edu} \and
J.D. Mireles James \footnote{Florida Atlantic University, 777 Glades Rd.,
Boca Raton, FL 33431, USA. jmirelesjames@fau.edu}
}
\maketitle
\begin{abstract}
We develop computer assisted arguments for proving the 
existence of transverse homoclinic connecting orbits, 
and apply these arguments for a number of non-perturbative
parameter and energy values in the 
spatial equilateral circular restricted four body problem.  
The idea is to formulate the desired connecting orbits as solutions of 
certain two point boundary value problems for orbit segments 
which originate and terminate on the local 
stable/unstable manifolds attached to a periodic orbit. 
These boundary value problems are studied via a Newton-Kantorovich 
argument in an appropriate Cartesian product of Banach algebras of 
rapidly decaying sequences of Chebyshev coefficients.  Perhaps the most 
delicate part of the problem is controlling the boundary conditions, 
which must lie on the local stable/unstable manifolds of the periodic orbit.
For this portion of the problem we use a parameterization method to 
develop Fourier-Taylor approximations equipped with a-posteriori error bounds.
This requires validated computation of a finite number of Fourier-Taylor coefficients
via Newton-Kantorovich arguments in appropriate Cartesian  product of
rapidly decaying sequences of Fourier coefficients, followed by a fixed point
argument to bound the tail terms of the Taylor expansion.  Transversality 
follows as a consequence of the Newton-Kantorovich argument.  
\end{abstract}


\section{Introduction}
The goal of the present work is to prove the existence of chaotic dynamics
in a particular restricted gravitational four body problem. The restriction 
is that the orbits of three of the bodies, called \textit{the primaries}, are
constrained to the equilateral triangle configuration of Lagrange.  That is, 
at each instant the primaries are located at the vertices of an equilateral
triangle, rigidly rotating with constant angular velocity about the center of mass.  
Considered singly, the orbit of each primary is a Keplerian circle about their 
center of mass.  The three circles need not be the same: they coincide 
if and only if the masses of the three bodies have equal mass.
   
One now introduces a fourth, massless particle 
 moving under the 
gravitational influence of the three massive primaries. 
Imagine a man-made space craft/satellite or small 
natural astroid/comet.
Changing to a co-rotating coordinate frame fixes the locations of the primary 
bodies, and the equations of motion for the massless particle
are now autonomous, albeit in a non-inertial frame.
The problem is referred to as the equilateral restricted four body problem, or simply the 
circular restricted four body problem (CRFBP), and it was derived and first studied 
by Pedersen \cite{pedersen1,pedersen2} in the mid 1940's and 50's.
Interest in the problem was revived in the late 1970's, after 
the study of Sim\'{o} \cite{MR510556}.
The equations of motion are given explicitly in Section \ref{sec:eqOfMotion}.  

The plane of the triangle is an invariant subsystem, and
chaotic motions in the planar problem have been established in 
a number of different contexts.  For example the authors of 
\cite{MR3626383,MR3038224,MR3158025}
prove the existence of planar chaotic motions by taking the 
mass of the second and third primary equal and very small. 
More precisely, they obtain the existence of planar chaos in CRFBP 
using Melnikov analysis and perturbing out of the planar Kepler problem. 
Chaos via the mechanism of Devaney (see \cite{MR442990}) is studied in 
 \cite{MR3105958}, again in the case that the second and third masses are
 small and equal.  For non-equal, and non-perturbative masses 
 $m_1 = 0.5, m_2 = 0.3, m_3 = 0.2$ (the same non-symmetric mass 
 values studied in  \cite{MR510556}), the authors of \cite{MR3906230}
 show the existence of planar chaotic motions by directly verifying  
 the hypotheses of Devaney's theorem.  The proof is computer 
 assisted.

In the present work we consider the \textit{spatial}
CRFBP, and prove the existence of out of plane chaotic dynamics.   
The idea of our proof is to directly verify the hypotheses of Smale's homoclinic tangle 
theorem \cite{MR228014}, using constructive computer assisted methods. 
 Recall that -- in brief -- the Smale theorem says the existence
of a transverse homoclinic to a periodic solution of an ODE implies chaos. 
In an earlier work \cite{MR4128684}, the present authors 
provide numerical evidence suggesting the existence spatial chaos
in the CRFBP.  The idea was to 
study the vertical Lyapunov families of periodic orbits attached to the 
the saddle-focus libration points of the CRFBP.  
 These saddle-focus libration points were 
shown to admit planar homoclinic orbits in \cite{MR3906230}, and
in  \cite{MR4128684} we used these planar homoclinics to locate approximate homoclinic connections for vertical Lyapunov 
family. These approximations were projected onto parameterizations
of the stable/unstable manifolds of the periodic orbits, and refined via
a differential corrections/Newton scheme.

The present work begins where  \cite{MR4128684} left off.  That is,
we develop an a-posteriori argument which allows us to pass 
from the numerical evidence in  \cite{MR4128684} to a
mathematically rigorous  (computer assisted)  proof of the existence 
of spatial Smale horseshoes in the CRFBP.  
We remark that computer assisted proofs of chaos in the 
planar restricted three body problem (as opposed to the four body case 
considered here) were established 
in \cite{MR1947690} for the case of equal masses, and in 
\cite{WilczakZgli,MR2174417,MR3032848}
for the Sun-Jupiter mass values.
These works exploit Poincare sections (restricted to an energy manifold), 
a strategy which reduces
the problem to a question about intersections of one dimensional arcs in the plane.  
Exploiting this reduction in the spatial case leads to questions about intersections of 
two dimensional manifolds for four dimensional Poincare maps, and visualizing the problem is much 
more difficult.  

While the geometric methods used by the authors just cited can certainly be extended 
to higher (even infinite dimensions - see for example \cite{MR4113209}), we choose 
an alternative  approach, which is to 
develop a-posteriori analysis for the  BVP set up
utilized in  \cite{MR4128684}, working directly with the periodic orbits, their attached local 
invariant manifolds, and connecting orbit
segments between them in the six dimensional phase space of the ODE.  Actually
--since our approach requires extensive manipulation of Taylor, Fourier, Chebyshev, 
and Fourier-Taylor series --   we find it convenient to append three additional differential equations, 
effectively embedding the CRFBP into a nine dimensional \textit{polynomial system} of ODEs.  
(This polynomial embedding is reviewed in \ref{sec:eqOfMotion}).

In this sense, the present work builds on the earlier work of 
\cite{Ransford,MR3896998}
on computer assisted Fourier analysis of periodic orbits, 
the work of \cite{MR3871613}
on validated Fourier-Taylor computation of local 
stable/unstable manifolds attached to periodic orbits, 
and the works of \cite{MR3148084,MR3353132,paperBridge}
on computer assisted proofs for two point BVPs 
projected into Chebyshev space.  
One important technical point is that, since we work with higher dimensional 
manifolds and in a higher dimensional phase than in these prior works, 
we find it convenient  to formulate the a-posteriori error analysis 
using a ``matrix free'' approach,
similar to that developed for Taylor series in \cite{MR3281845,MR3792792}.
That is, we do not employ a Newton-Like operator in the error analysis of the 
stable/unstable manifolds of the periodic orbits, and hence avoid working with 
large interval matrices.  Rather, we formulate  a fixed point problem on an appropriate 
space of ``Fourier-Taylor tails'' which -- for a given choice of the 
truncation order -- contracts for appropriate choices of 
the scalings of the stable/unstable bundle parameterizations.
It is also important to note that one pair of parameterized stable/unstable 
manifolds (with validated error bounds) can be used in many different computer
assisted proofs of distinct connecting orbits.  This provides one justification of 
the computational effort which goes into studying them.

The remainder of the paper is organized as follows.
In the next subsection we review the Equations of motion for the CRFBP
\ref{sec:eqOfMotion}.
Section \ref{sec:background} reviews background material
on the parameterization method, Banach algebras of rapidly decaying 
coefficients, and a-posteriori analysis needed 
in the remainder of the paper.  
In Section \ref{sec:lowOrderTerms} we discuss the computer assisted techniques for 
proving the existence of the periodic orbits, performing validated computations 
of their stable/unstable normal bundles, solving the homological 
equations describing the jets of the stable/unstable manifolds, and
establishing the existence of connecting orbits via the solution 
of two point boundary value problems.  
Mathematically rigorous bounds on the tail of the Fourier-Taylor expansion 
of the stable/unstable manifold are developed in 
\ref{sec:TailArgument}.
In Section \ref{Sec:Results} we discuss our main results, and 
several appendices describe some more technical/tedious bounds.  
The computer codes which execute the computer assisted portions 
of our arguments are implemented in MatLab using the IntLab library
for managing round off errors \cite{Ru99a}, 
and are freely available on the homepage of the second author:

\smallskip

\noindent \url{https://cosweb1.fau.edu/~jmirelesjames/spatialCRFBP_CAP_chaos.html}

\subsection{The (Spatial) Equilateral Restricted Four Body Problem} \label{sec:eqOfMotion}

\begin{figure}[!t]
\centering
\includegraphics[width=4.0in]{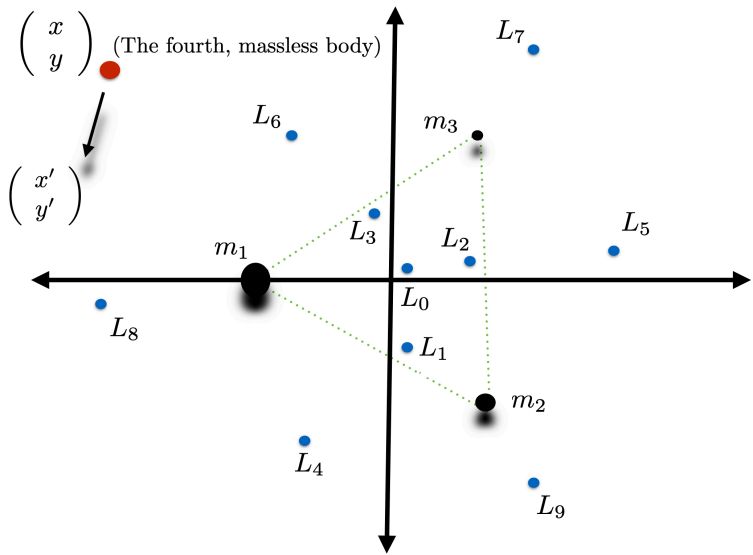}
\caption{\textbf{Configuration space for the CRFBP:} The three 
primary bodies with masses $m_1,m_2,$ and $m_3$ are 
arranged in the equilateral triangle configuration of 
Lagrange --  a relative equilibrium solution of the 
three body problem.    Transforming to a
co-rotating frame one considers the motion of a fourth
massless body.  The equations of motion have 
$8$, $9$, or $10$ equilibrium solutions
denoted by $\mathcal{L}_j$
for $0 \leq j \leq 9$.  The number of libration points,
and their stability, varies depending on $m_1$, $m_2$, and $m_3$. 
In this work we study the points $\mathcal{L}_{0, 4,5,6}$ which 
have saddle focus stability when $m_1 \approx m_2 \approx m_3$.
}\label{fig:rotatingframe}
\end{figure}

Let the three primary bodies have masses $m_1$, $m_2$ and $m_3$, 
normalized so that $0 < m_3 \leq m_2 \leq m_1$,
and 
\[
m_1 + m_2 + m_3 = 1.
\] 
We sometimes refer to the primary bodies simply as $m_1, m_2,$ and $m_3$.
Further normalizations allow taking the center of mass at the origin,
the position of $m_1$ on the negative $x$-axis, and $m_3$ 
is in the first quadrant.  

Under these constraints, the locations of the primaries are functions of the
masses $m_1, m_2, m_3$ only.  That is, letting
\[
p_j = (x_j, y_j,z_j), \quad \quad \quad \mbox{for } j = 1,2,3,
\]
denote the positions of the primary bodies in the rotating coordinate system, 
and defining 
\[
K = m_2(m_3 - m_2) + m_1(m_2 + 2 m_3),
\]
it can be shown that 
\begin{eqnarray*}
x_1 &=&   \frac{-|K| \sqrt{m_2^2 + m_2 m_3 + m_3^2}}{K} \\
y_1 &=&   0 
\end{eqnarray*}
\begin{eqnarray*}
x_2 &=&  \frac{|K|\left[(m_2 - m_3) m_3 + m_1 (2 m_2 + m_3)  \right]}{
2 K \sqrt{m_2^2 + m_2 m_3 + m_3^2}
}  \\
y_2 &=&  \frac{-\sqrt{3} m_3}{2 m_2^{3/2}} \sqrt{\frac{m_2^3}{m_2^2 + m_2 m_3 + m_3^2}},
\end{eqnarray*}
\begin{eqnarray*}
x_3 &=&  \frac{|K|}{2 \sqrt{m_2^2 + m_2 m_3 + m_3^2}}  \\
y_3 &=&  \frac{\sqrt{3}}{2 \sqrt{m_2}} \sqrt{\frac{m_2^3}{m_2^2 + m_2 m_3 + m_3^2}}, 
\end{eqnarray*}
and 
\[
z_1 = z_2 = z_3 = 0. 
\]
We refer to  \cite{MR510556,MR3105958} for detailed derivations of these
expressions, and note that  explicit formulas for the coordinates of the primaries 
are essential, as the distance from the massless particle to the primaries 
is essential in the equations of motion.  

Indeed, letting 
\[
r_j(x,y,z) := \sqrt{(x-x_j)^2 + (y-y_j)^2 +(z-z_j)^2}, \quad \quad \quad j = 1,2,3,
\]
denote the distances between the primaries and the massless particle, one
defines the potential function 
\[
\Omega(x,y,z) :=
\frac{1}{2} (x^2 + y^2) + \sum_{j=1}^3 \frac{m_j}{r_j(x,y,z)}.
\]
The equations of motion describing the infinitesimal particle 
in a co-rotating frame are 
\begin{equation}\begin{split}\label{ecuacionesfinales}
\ddot{x}-2\dot{y}&=\Omega_{x},\\
\ddot{y}+2\dot{x}&=\Omega_{y},\\
\ddot{z} &= \Omega_z.
\end{split}
\end{equation}
We remark that the system preserves the first integral 
\begin{equation}\label{eq:Energy}
H(x, \dot x, y, \dot y, z, \dot z) = x^2 +y^2 +2\left(\frac{m_1}{r_1(x,y,z)} 
+\frac{m_2}{r_2(x,y,z)} +\frac{m_3}{r_3(x,y,z)}\right)  -(\dot x^2 +\dot y^2 +\dot z^2),
\end{equation}
which is traditionally referred to as the Jacobi integral.

It was conjectured in \cite{MR510556} (based on a careful numerical analysis) 
that the CRFBP admits $8, 9$ or $10$ equilibrium solutions, depending on the 
values of the mass parameters.  All of the equilibrium solutions, 
which are traditionally referred to as ``libration points,'' lie in the $xy-$ plane -- 
that is, the invariant plane of the equilateral triangle. 
For a detailed discussion of stability of the equilibrium solutions see 
\cite{MR4307390,MR4384822}.
Mathematically rigorous (computer assisted) proofs 
confirming the correctness of the conjectured 
libration point count for all 
values of the mass parameters, are found in 
\cite{MR2232439,MR2784870,MR3176322} and also in 
\cite{tuckerLibrations}.
No closed formulas for the locations of the equilibrium solutions
exist, so that in practice they are computed numerically
via Newton's method.  Many out of plane periodic solutions were proven 
to exist (again with computer assistance) in 
\cite{MR3896998}.

A schematic describing the locations of the 10 equilibrium solutions, 
along with our naming conventions,  is given in Figure \ref{fig:rotatingframe}.
The reader interested in a more thorough discussion of the qualitative features of 
the dynamics of the CRFBP may consult the works of 
\cite{BaltagiannisPapadakis,MR2845212,MR3500916,MR3571218,MR3715396,MR2027748,MR2596303,MR3919451,MR4128684},
and the references contained therein.

\subsection{Polynomial embedding} \label{sec:polyField}
We now describe the polynomial embedding used to simplify formal series calculation
for the CRFBP.
Write  
\begin{equation}\label{eq:AutomaticDiffeqn1}
 u_1 = x, \quad u_2 = \dot x, \quad u_3= y, \quad u_4 = \dot y, \quad u_5=z,\quad u_6= \dot z,
\end{equation}
and consider the first order ODE $\dot u = f(u)$ given by
\begin{equation} \label{eq:FirstOrderOriginal}
\begin{split}
\dot{u_1} &= u_2, \\
\dot{u_2} &= 2 u_4 +\Omega_{u_1},\\
\dot{u_3} &= u_4, \\
\dot{u_4} &= -2u_2 +\Omega_{u_3},\\
\dot{u_5} &= u_6, \\
\dot{u_6} &= \Omega_{u_5}.
\end{split}
\end{equation}
Now define the auxiliary variables 
\begin{align}
u_7 = \frac{1}{\sqrt{(x-x_1)^2 +(y-y_1)^2 +(z-z_1)^2}} =  \frac{1}{\sqrt{(u_1-x_1)^2 +(u_3-y_1)^2 +(u_5-z_1)^2}}, \label{eq:AutomaticDiffeqn2} \\
u_8 = \frac{1}{\sqrt{(x-x_2)^2 +(y-y_2)^2 +(z-z_2)^2}} =  \frac{1}{\sqrt{(u_1-x_2)^2 +(u_3-y_2)^2 +(u_5-z_2)^2}}, \label{eq:AutomaticDiffeqn3} \\
u_9 = \frac{1}{\sqrt{(x-x_3)^2 +(y-y_3)^2 +(z-z_3)^2}} =  \frac{1}{\sqrt{(u_1-x_3)^2 +(u_3-y_3)^2 +(u_5-z_3)^2}}, \label{eq:AutomaticDiffeqn4}
\end{align}
and let $U \subset \mathbb{R}^6$ denote the open set which excludes the locations of the primary bodies.
Consider the function $R:U \to \mathbb{R}^9$ given by 
\begin{equation}\label{eq:R}
R(u_1,u_2,u_3,u_4,u_5,u_6) = \begin{pmatrix}
u_1 \\ u_2 \\ u_3 \\ u_4 \\ u_5 \\ u_6 \\                
\frac{1}{\sqrt{(u_1-x_1)^2 +(u_3-y_1)^2 +(u_5-z_1)^2}} \\
\frac{1}{\sqrt{(u_1-x_2)^2 +(u_3-y_2)^2 +(u_5-z_2)^2}} \\
\frac{1}{\sqrt{(u_1-x_3)^2 +(u_3-y_3)^2 +(u_5-z_3)^2}}
\end{pmatrix}.
\end{equation}

Note that by differentiating Equation \eqref{eq:AutomaticDiffeqn2} and 
applying the chain rule, we have that 
\[
u_7' = (x_1-u_1)u_2u_7u_7u_7 +(y_1-u_3)u_4u_7u_7u_7 +(z_1-u_5)u_6u_7u_7u_7,
\]
and that similar equations calculations hold for $u_{8,9}$.  
Motivated by these observations we define the polynomial 
vector field 
$F \colon \mathbb{R}^9 \to \mathbb{R}^9$ 
by 
\begin{align}\label{eq:Fwithoutalpha}
F(u)= 
\begin{pmatrix}
u_2 \\
2u_4 +u_1 +m_1(x_1-u_1)u_7u_7u_7 +m_2(x_2-u_1)u_8u_8u_8 +m_2(x_3-u_1)u_9u_9u_9 \\
u_4 \\
-2u_2 +u_3 +m_1(y_1-u_3)u_7u_7u_7 +m_2(y_2-u_3)u_8u_8u_8 +m_2(x_3-u_1)u_9u_9u_9 \\
u_6 \\
m_1(z_1-u_5)u_7u_7u_7 +m_2(z_2-u_5)u_8u_8u_8 +m_2(z_3-u_5)u_9u_9u_9 \\
(x_1-u_1)u_2u_7u_7u_7 +(y_1-u_3)u_4u_7u_7u_7 +(z_1-u_5)u_6u_7u_7u_7 \\
(x_2-u_1)u_2u_8u_8u_8 +(y_2-u_3)u_4u_8u_8u_8 +(z_2-u_5)u_6u_8u_8u_8 \\
(x_3-u_1)u_2u_9u_9u_9 +(y_3-u_3)u_4u_9u_9u_9 +(z_3-u_5)u_6u_9u_9u_9
\end{pmatrix}.
\end{align}
The Jacobi integral becomes 
\begin{equation}\label{eq:EnergyPolyEmbedding}
H(u) = u_1^2 +u_3^2 +2\left(m_1u_7 +m_2u_8 +m_3u_9\right)  -(u_2^2 +u_4^2 +u_6^2).
\end{equation}

The polynomial field $F$ and the CRFBP field $f$
enjoy the infinitesimal conjugacy relation 
\begin{equation}\label{eq:autoDiff}
DR(u)f(u) = F(R(u)), \quad \forall u \in U.
\end{equation}
It follows that orbits of $u' = F(u)$ with initial conditions on $\mbox{graph}(R)$
correspond, after projection onto the first six coordinates, to orbits of $x' = f(x)$.
Note that the polynomial vector field does not have any singularity. Nevertheless,
the dynamics of the two are related only on the graph of $R$, so that
$R$ caries the singularities of $f$.  That is, the 
polynomial embedding does not regularize collisions.  
Instead the virtue of $F$ is that we have a problem involving only 
differentiation and multiplication, two operations with excellent formal
and numerical properties when working with Fourier-Taylor series.  
For more general discussion, see  \cite{MR3906230}.
In the present work we study periodic solutions, their attached local stable/unstable manifolds,
and connecting orbits for the polynomial 
vector field $F$.  Results for $f$, the CRFBP, are obtained by projection.  
 

\section{Background} \label{sec:background}

\subsection{Parameterization method: case of complex conjugate Floquet multipliers} \label{sec:parmMethod}
The Smale homoclinc tangles studied in the present work are
built on periodic orbits whose stable/unstable Floquet multipliers come in 
complex conjugate pairs.  It follows that  
the attached local stable/unstable manifolds are three dimensional,
and we compute these manifolds using the parameterization method.  
The main idea of the parameterization method in this context
is to study a complex infinitesimal invariance equation
which conjugates the dynamics on the manifold to a simple linear 
flow generated by the complex conjugate multipliers.  
However, since we are ultimately interested in the real dynamics of the system,  
we have to discuss how to obtain the real image of the complex parameterization.  

We remark that a parameterization method for stable/unstable manifolds 
associated with a singe real stable or unstable multiplier for periodic orbits
of vector fields  was introduced in  
\cite{MR2177465}.  Generalizations to higher dimensional manifolds,
efficient algorithms, and techniques for a-posteriori error analysis are developed in
 \cite{MR2551254,MR3118249,MR3304254,robertoII,MR3927444,MR3754682}.
In the next section, we review the parameterization method for the case of a
complex conjugate pair of multipliers.

\subsubsection{Linear stability of a periodic solution} \label{sec:floquet}
Let $U$ be an open subset of 
$\mathbb{R}^n$ and $f: U \to \R^n$  
be a real analytic vector field.
Suppose that $\gamma: \mathbb{R} \to \R^n$ 
has that 
\begin{equation}\label{eq:ode}
\frac{d}{dt}\gamma(t) = f(\gamma(t)),
\end{equation}
with 
\[
\gamma(t + T) = \gamma(t), 
\]
for all $t \in \mathbb{R}$.  Then $\gamma$
 is a $T$ periodic solution of the ODE.

We say that $\lambda \in \mathbb{C}$ is a Floquet multiplier
of the periodic orbit $\gamma$, with attached invariant vector
bundle $\xi \colon \mathbb{R} \to \mathbb{R}^n$, if $\xi(t)$ is 
a periodic function and 
the pair $(\lambda, \xi(t))$ satisfy the eigenvalue problem 
\begin{equation}\label{eq:eigenvalueProblem}
-\frac{d}{d t} \xi(t) + Df(\gamma(t)) \xi(t) = \lambda \xi(t). 
\end{equation}
It is a standard result from Floquet theory that the period of 
$\xi(t)$ can only be $T$ or $2T$.
In the former case we say that the vector 
bundle is orientable, and say that it is 
non-orientable in the later.

A periodic solution $\gamma(t)$ always has one trivial 
Floquet multiplier, associated with the tangent bundle of the orbit
(to see this, simply take $\xi(t) = \gamma'(t)$ and 
differentiate Equation \eqref{eq:ode} with respect to $t$). 
In the remainder of the present work, we are especially 
interested in the case where $\gamma(t)$ has a 
complex conjugate pair of Floquet multipliers 
\[
\lambda = \alpha + i \beta, \quad \quad \mbox{and} \quad \quad 
\overline{\lambda} = \alpha - i \beta,
\]
with $\alpha, \beta \in \mathbb{R}$ and $\beta > 0$.

\begin{remark}[Systems with a first integral] \label{rem:HamFloquet}
If $f$ has a conserved first integral
(as is the case for the spatial CRTBP), then 
$\gamma$ has a second zero Floquet multiplier in the direction 
normal to the level set of the conserved quantity.  The 
remaining $d - 2$ multipliers are (generically) either stable, unstable, 
or purely imaginary. (That is, additional multipliers of zero imply that 
$\gamma$ is undergoing a local bifurcation).
In the Hamiltonian case, even more is true, and we have 
that if $\lambda \in \mathbb{C}$ is a Floquet multiplier
then so are $-\lambda, \overline{\lambda}$, and $-\overline{\lambda}$, though
these are not distinct if $\lambda$ is real. 
In this paper we are especially interested in the case when $d = 6$, and where
the four non-zero multipliers are of the form 
\[
\lambda_{1,2,3,4} = \pm \alpha \pm i \beta, 
\]
for $\alpha, \beta > 0$.  In this case the stable/unstable manifolds
attached to $\gamma$ are three dimensional.  
We remark that the CRFBP is Hamiltonian in position/momentum 
variables.  While we employ position velocity variables in the present 
work, the change between these two systems is affine and more importantly
does not effect the stability of periodic orbits. 
\end{remark}

In the present work, we are especially interested in the case where 
 $\lambda_1,\lambda_2 \in \mathbb{C}$ are a pair of complex
conjugate stable Floquet exponents
 for the periodic orbit $\gamma$.  We write $\lambda_1 = -\alpha + i \beta$ and 
 $\lambda_2 = \overline{\lambda_1}$.

\begin{figure}[!t]
	\includegraphics[width=1.0\textwidth]{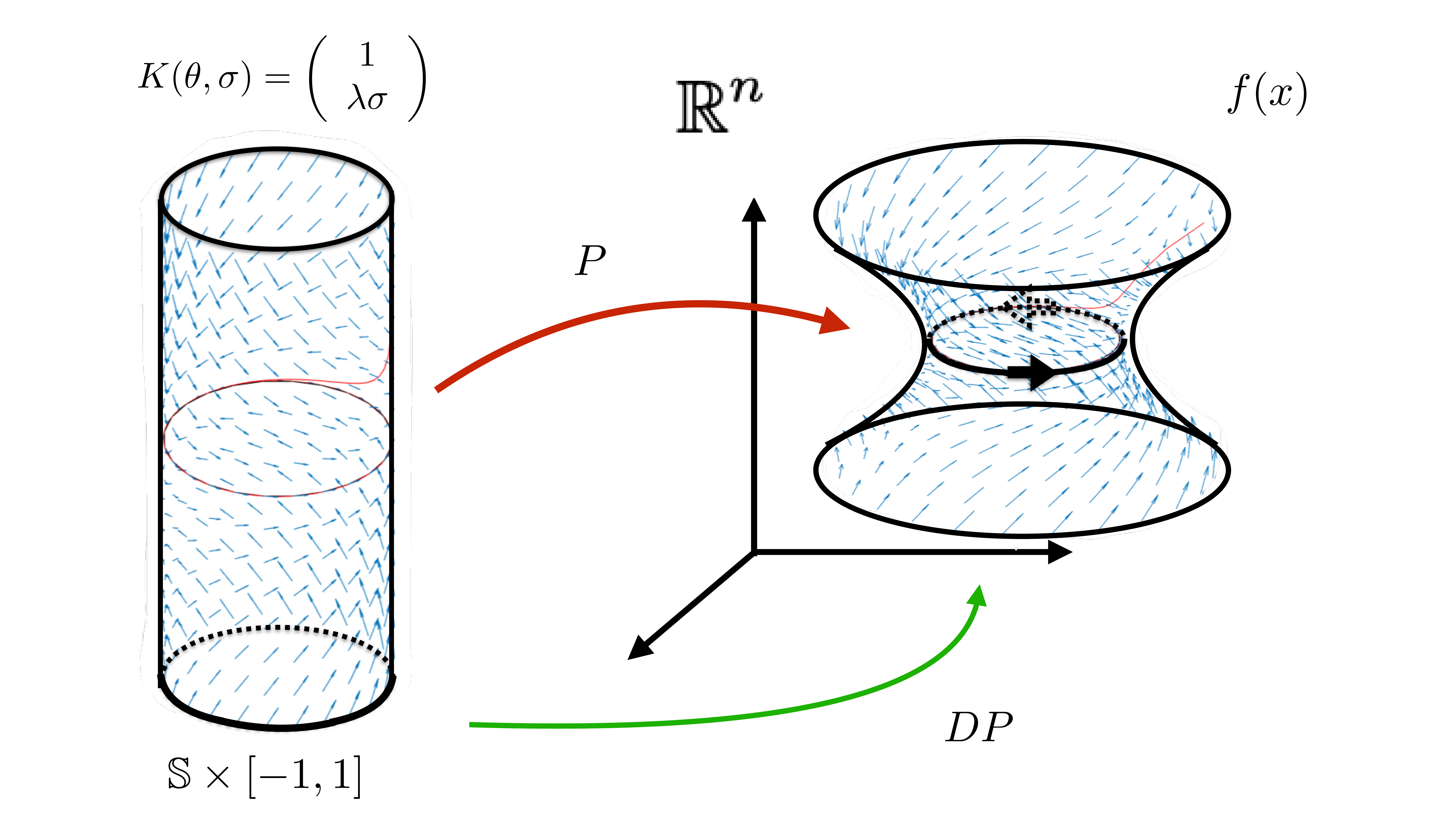}
\caption{\textbf{Geometric meaning of Equation \eqref{eq:conjugacy}:} the figure illustrates
the idea behind the parameterization method, which is that we model the dynamics on the 
stable manifold of the periodic using the vector field $K$ on the cylinder.  If the push forward 
of $K$ under $DP$ agrees with the vector field $f$ restricted to the image of $P$, then 
$P$ maps orbits to orbits, and the dynamics of the two systems --the toy system and the 
stable manifold -- are the same. This figure pictures a lower dimensional example, but this idea persists in higher dimensional cases.
}\label{fig:parmSchematic}
\end{figure}

 \subsubsection{The conjugacy equation for the parameterization method}
 Define the vector field 
 \[
 K(\theta, z_1, z_2) = \left(
 \begin{array}{c}
 1 \\
 \lambda_1 z_1 \\
 \lambda_2 z_2
 \end{array}
 \right),
 \]
on the cylinder $\mathcal{C} = \mathbb{S} \times \mathbb{D}^2$. Here 
$\mathbb{D}$ is the unit polydisk in $\mathbb{C}^2$ given by
\[
\mathbb{D}= \left\{ (z_1, z_2) \in \mathbb{C}^2 \, | \, |z_1|, |z_2| < 1 \right\}.
\]
Note that the ODE
\[
\left(
\begin{array}{c}
\theta' \\
z_1' \\
z_2' \\
\end{array}
\right) = 
\left(
\begin{array}{c}
1 \\
\lambda_1 z_1 \\
\lambda_2 z_2 \\
\end{array}
\right)
\]
generates the flow 
\[
\phi(\theta, z_1, z_2, t) = \left(
\begin{array}{c}
\theta + t  \quad (\mbox{mod } 1)\\
z_1 e^{\lambda_1 t} \\
z_2 e^{\lambda_2 t} \\
\end{array}
\right), 
\]
where 
\[
e^{\lambda_1 t} = e^{-\alpha t}\left( \cos(\beta t) + i \sin(\beta t) \right),
\quad \quad \mbox{and} \quad \quad 
e^{\lambda_2 t} = e^{-\alpha t}\left( \cos(\beta t) - i \sin(\beta t) \right),
\]
are complex exponentials.  
We have that 
\begin{itemize}
\item \textbf{Forward invariance:} If $(\theta, z_1, z_2) \in \mathbb{S} \times \mathbb{D}^2$ then 
$\phi(\theta, z_1, z_2, t) \in \mathbb{S} \times \mathbb{D}^2$ for all $t \geq 0$.
\item \textbf{Real image:}
If $z_1 = \sigma_1 + i \sigma_2$ and $z_2 = \sigma_1 - i \sigma_2$ with 
$\sigma_1, \sigma_2 \in \mathbb{R}$ and $\sigma_1^2 + \sigma_2^2 < 1$, then 
\[
\phi(\theta, \sigma_1 + i \sigma_2, \sigma_1 - i \sigma_2, t) \in \mathbb{S} \times B^2, 
\]
for all $t \geq 0$, where $B^2 \subset \mathbb{R}^2$ is the real unit disk in the plane.  
\item \textbf{Real bundles:} 
If $\lambda_1, \lambda_2$ are complex conjugate Floquet multipliers, then we
can choose complex conjugate Floquet bundles, in the sense that $(\lambda_1, \xi_1(\theta))$
and $(\lambda_2, \xi_2(\theta))$ are solutions of Equation \eqref{eq:eigenvalueProblem}, and 
there are $T$ periodic $\eta_1, \eta_2 \colon \mathbb{R} \to \mathbb{R}^d$ so that 
\[
\xi_1(\theta) = \eta_1(\theta) + i \eta_2(\theta), 
\]
while 
\[
\xi_2(t)  = \eta_1(\theta) - i \eta_2(\theta).
\]
\item \textbf{Real linear approximation: }
The linear approximation of the true dynamics near $\gamma$
is given by 
\[
L(\theta, \sigma_1, \sigma_2, t) = 
\gamma(\theta + t)  +
 \xi_1(\theta+t)(\sigma_1 + i \sigma_2) e^{\lambda_1 t} + 
 \xi_2(\theta+t) (\sigma_1 - i \sigma_2) e^{\lambda_2 t},
\]
and we note that $L(\theta, \sigma_1, \sigma_2, t)$ is real valued.
\end{itemize}

The following lemma provides a method for obtaining higher order 
corrections to the linear approximation in a natural way.  

\begin{lemma}[Parameterization lemma - case of complex conjugate Floquet exponents] \label{lem:Parm} 
Suppose that $P \colon \mathbb{S} \times \mathbb{D}^2 \to \mathbb{C}^d$ satisfies the first order 
constraints 
\[
P(\theta, 0, 0) = \gamma(\theta), 
\]
with 
\[
\frac{\partial}{\partial z_1} P(\theta, 0, 0) = \xi_1(\theta), 
\quad \quad \mbox{and} \quad \quad 
\frac{\partial}{\partial z_2} P(\theta, 0, 0) = \xi_2(\theta),
\]
and that $P$ solves the partial differential equation 
\begin{equation}\label{eq:conjugacy}
\frac{\partial}{\partial \theta}P(\theta, z_1, z_2) +
\lambda_1 z_1 \frac{\partial}{\partial z_1} P(\theta,z_1, z_2)  +
\lambda_2 z_2 \frac{\partial}{\partial z_2} P(\theta,z_1, z_2)
= f\big(P(\theta, z_1, z_2)\big).
\end{equation}
for $\theta \in \mathbb{S}$ and $z_1, z_2 \in \mathbb{D}^2$.

Then, for all $(\theta, z_1, z_2) \in \mathbb{S} \times \mathbb{D}^2$ we have that 
\begin{equation}\label{eq:conjugacyFlow}
P(\theta + t, z_1 e^{\lambda_1 t}, z_2 e^{\lambda_2 t}) = \Phi(P(\theta, z_1, z_2), t),
\end{equation}
It follows that
$\mbox{image}(P) = P(\mathbb{S}, \mathbb{D}, \mathbb{D})$ is a subset of the local 
stable manifold for $\gamma$.  
\end{lemma}

\begin{proof}
Assume that $P$ satisfies the first order constraints and is a solution of 
Equation \eqref{eq:conjugacy}.
Choose $(\hat \theta, \hat z_1, \hat z_2) \in \mathbb{S} \times \mathbb{D}^2$
and define the function $u \colon [0, \infty) \to \mathbb{C}^d$ by 
\[
u(t) = P(\hat \theta + t, \hat z_1 e^{\lambda_1 t}, \hat z_2 e^{\lambda_2 t}).
\]
Note that $u(0) = P(\hat \theta, \hat z_1, \hat x_2)$.  We claim that $u(t)$ is a solution 
of the differential equation $u' = f(u)$.  To see this,  
let $ t \geq 0$ and define 
\[
K(t) =  
\left(\hat \theta +t, \hat z_1 e^{\lambda_1 t}, \hat z_2 e^{\lambda_2 t}\right).
\]
Note that $K(t) \in \mathbb{S} \times \mathbb{D}^2$ for all $t \geq 0$, that 
\[
u(t) = P(K(t)), 
\]
and that $u(0) = P(K(0))$.
Differentiating $u(t)$ with respect to time
leads to 
\begin{align*}
\frac{d}{dt} u(t) &= 
\frac{d}{dt} P(\hat \theta +t, \hat z_1 e^{\lambda_1 t}, \hat z_2 e^{\lambda_2 t}) \\
& = DP(\hat \theta +t, \hat z_1 e^{\lambda_1 t}, \hat z_2 e^{\lambda_2 t}) \left(
\begin{array}{c}
\frac{d}{dt} (\hat \theta + t) \\
\frac{d}{dt} (\hat z_1 e^{\lambda_1} t) \\
\frac{d}{dt}(\hat z_2 e^{\lambda_2 t} )
\end{array}
\right) \\
 &= \frac{\partial}{\partial \theta} P(K(t))
 +  \lambda_1 K_2(t)\frac{\partial}{\partial z_1} P(K(t))
 + \lambda_2 K_3(t)  \frac{\partial}{\partial z_2}P(K(t))  \\
                  &= f\big(P(K(t)\big) \\
              &   = f(u(t)).
\end{align*}
Then  
\begin{align*}
\Phi(P(\hat \theta, \hat z_1, \hat z_2, t))&= u(t)  \\
&= P(\hat \theta +t, \hat z_1 e^{\lambda_1 t}, \hat z_2 e^{\lambda_2 t}),
\end{align*}
as desired.

Moreover, from the flow conjugacy relationship and the fact that 
\[
e^{\lambda_{1,2} t} \to 0,
\]
as $t \to \infty$ we see that
\[
P(\hat \theta +t, \hat z_1 e^{\lambda_1 t}, \hat z_2 e^{\lambda_2 t}))
\to \gamma, 
\]
as $t \to \infty$.  This shows that the orbit of a point on 
the image of $P$  accumulates to $\gamma$.
\end{proof}

\begin{remark}[Dimension count] \label{rem:dimensionCount}
If it is known that $\lambda_1, \lambda_2 \in \mathbb{C}$ are
the only two stable Floquet multipliers of $\gamma$, then 
we know by the stable manifold theorem for periodic orbits
that the stable manifold of $\gamma$ is two dimensional.
It then follows that $P$ parameterizes a local stable manifold 
for $\gamma$.
\end{remark}

\begin{remark}[Unstable manifold] \label{rem:unstableManifold}
If $\lambda_1, \lambda_2 \in \mathbb{C}$ are a complex
conjugate pair of unstable Floquet multipliers for $\gamma$
and $P$ solves Equation \eqref{eq:conjugacy}, then $P$
parameterizes a subset of the unstable manifold of $\gamma$.
The argument follows exactly as above, except with time reversed.
If it is known that $\lambda_{1,2}$ are the only unstable multipliers of 
$\gamma$, then $P$ parameterizes a local unstable manifold for 
$\gamma$. 
\end{remark}

\subsubsection{Parameterization by Fourier-Taylor series}
 We seek a representation of 
 $P$ as a power series
\begin{equation}\label{eq:Ptaylorexpansion}
P(\theta, z_1, z_2)= \sum_{m = 0}^\infty \sum_{n = 0}^\infty A_{m,n}(\theta) z_1^m z_2^n, 
\end{equation}
with coefficients $A_{m,n}(\theta)$,  $T-$periodic complex functions
as from here on we assume that the bundles are orientable. (However
in the non-orientable case we simply replace $T$ by $2T$
throughout the discussion).   
We use the multi-index notation $\alpha = (m,n) \in \mathbb{Z}^2$
and $|\alpha| = m+ n$ whenever convenient.

Note that the first order coefficients are fully determined by the 
first order constraints on the parameterization method. 
That is, 
\begin{align*}
A_0(\theta) &= \gamma(\theta), ~\forall \theta \in [0,T] , ~\mbox{and} \\
A_{1,0} (\theta) &= \xi_1(\theta), \quad \quad A_{0,1}(\theta) = \xi_2(\theta),
 ~\forall \theta \in [0,T]. 
\end{align*}
Note that (at least formally) 
\[
\frac{\partial}{\partial \theta} P(\theta, z_1, z_2) = 
\sum_{m = 0}^\infty \sum_{n = 0}^\infty \frac{\partial}{\partial \theta}
A_{m,n}(\theta) z_1^m z_2^n,
\]
and that 
\[
\lambda_1 z_1 \frac{\partial}{\partial z_1} P(\theta, z_1, z_2) + 
 \lambda_2 z_2  \frac{\partial}{\partial z_2} P(\theta, z_1, z_2) = 
\sum_{m = 0}^\infty \sum_{n = 0}^\infty 
(\lambda_1 m + \lambda_2 n)A_{m,n}(\theta) z_1^m z_2^n.
\]
Let
\[
f(P(\theta, z_1, z_2)) = 
\sum_{m = 0}^\infty \sum_{n = 0}^\infty 
B_{m,n}(\theta) z_1^m z_2^n,
\]
where the $B_{m,n}$ are in fact functions of the $A_{j,k}$, for 
$0 \leq j + k \leq m+n$.  Indeed, we have that 
\begin{align*}
B_{m,n} & = [f(P(\theta, z_1, z_2))]_{m,n} \\
& = \left. \frac{1}{(m,n)!} \frac{\partial^{|m+n|}}{\partial_m \partial_n} f(\theta, z_1, z_2) \right|_{z_1 = z_2 = 0},
\end{align*}
is the $(m,n)$-th Taylor coefficient of the composition $f(P(\theta, z_1,z_2)$.

Matching like powers of $(m,n)$ leads to the ordinary differential equations 
\begin{equation}\label{eq:conjugacyalpha}
\frac{d}{d \theta}A_{m,n}(\theta) + 
(m \lambda_1 + n \lambda_2) A_{m,n}(\theta) = B_{m,n},
\end{equation}
determining the $(m,n)-th$ coefficient of $P$, for 
for $m + n \geq 2$.  Equation \eqref{eq:conjugacyalpha}
is referred to as \textit{the homological equation} for $P$.

The term $B_{m,n}$ on the right hand side of Equation \eqref{eq:conjugacyalpha}
is the $(m,n) -$th Taylor coefficient
of $f(P(\theta, z_1, z_2))$.  It is important to note that $B_{m,n}$ 
depends on $A_{m,n}$.  That is, the term $A_{m,n}$ is not isolated 
on the left hand side of Equation \eqref{eq:conjugacyalpha}.
It can be shown that this dependence is linear, and in fact that 
\[
B_{m,n} = Df(A_{0,0}) A_{m,n} + R_{m,n},  
\]
where $A_{0,0}$ is the periodic orbit
and $R_{m,n}$ depends only on the lower order terms
$A_{i,j}$ with $0 < i + j < m + n$.
Then 
\begin{equation} \label{eq:homEq_Amn_isolated}
-\frac{d}{d \theta}A_{mn}(\theta) + 
\left[  Df(A_{0,0})  - (m \lambda_1 + n \lambda_2) \mbox{Id} \right] A_{m,n}(\theta) = -R_{m,n},
\end{equation}
Note that when $m + n = 1$ this reduces Equation \eqref{eq:eigenvalueProblem}
 for the linear bundles (eigenfunctions).

An explicit formulas for $R_{m,n}$ is derived in one of two 
ways.  The first approach is to 
repeatedly evaluate derivatives of $f(P(\theta, z_1, z_2))$
with respect to $(z_1, z_2)$ and evaluate these at $(z_1, z_2) = (0,0)$.  
Expressions for partial derivatives of all orders
are worked out in a combinatorial fashion using the Fa\`{a} di Bruno formula.
This approach has the advantage of being completely general, as it is
just an application of Taylor's theorem.  However the resulting formulas are 
quite complicated, and not optimal for numerical computations.

More convenient formulas are obtained in practice by expanding
the compositions of Fourier-Taylor series directly using Cauchy products.  
This approach takes advantage of 
the structure of the system, and is especially clear in the case of 
polynomial vector fields $f$.  Non-polynomial fields, like the ones considered 
in this paper, are embedded into polynomial systems using ``automatic differentiation 
for power series''.   
Recall that in the present work we exploit the polynomial 
embedding of the CRFBP given in Equation \eqref{eq:Fwithoutalpha}.
 For a general introduction and an overview of ``polynomial 
 embeddings'' of nonlinear vector fields we refer to 
 \cite{MR1431038, MR2204531}, and also to  
 Chapter $2$ of \cite{MR3467671}, Chapter $4.7$ of \cite{MR3077153}, 
 and to \cite{MR2146523}.

\begin{remark}[Non resonance criteria]
Equation \eqref{eq:homEq_Amn_isolated} makes it clear that 
the homological equations have unique, periodic solutions as 
long as  
 \[
 m \lambda_1 + n \lambda_2  \neq \lambda_{1,2},
 \]
 with $m + n \geq 2$.   (The existence and uniqueness
 follows from Floquet theory, see \cite{MR3304254}).
 Such an equality is called an inner resonance, 
 or simply a resonance.
 Since, in the present work, the multipliers $\lambda_1, \lambda_2$ are complex
 conjugates, there are no possible resonances, and the homological equators
 are solvable to all orders.
 That is, in the case of 
 a single pair of stable or unstable complex conjugate
 multipliers, the formal series solution of Equation \eqref{eq:conjugacy} always exists.  
 \end{remark}

\subsection{Banach Algebras of Fourier Sequences} \label{sec:BanachAlgebra}
Recall that the Fourier coefficients of a real analytic periodic function decay exponentially 
(Paley-Wiener Theorem).  More precisely, let $g:[0,T]\to R$ be a $T-$periodic, real analytic function, 
and $b = \{ b_k : k\in \mathbb{Z} \}$ represent the sequence of Fourier coefficients, 
so that
\[
g(t) = \sum_{k\in \mathbb{Z}} b_{k}e^{\im \frac{2\pi}{T}t}.
\]
Then there exists a constant $\nu \geq 1$ such that
\[
\sum_{k\in \mathbb{Z}} |b_k|\nu^{|k|} < \infty.
\]
This motivates our interest in the following Banach 
space of bi-infinite sequences.

\begin{definition}[Weighted Space of Fourier Coefficients]

Let $b=\{b_k\}_{k\in\mathbb{Z}}$.  We say that $b \in \ellnu$ if
\[
\|b\|_{1,\nu} \bydef \sum_{k\in \mathbb{Z}} |b_k|\nu^{|k|} < \infty,
\]
and note that $\ellnu$ is a Banach algebra under the convolution 
product $\star:\ellnu \times \ellnu \to \ellnu$ defined by  
\[
(a\star b)_k = \sum_{\substack{k_1+k_2=k \\ k_1,k_2 \in \mathbb{Z}}} a_{k_1}b_{k_2}.
\]
That is, we have that 
\[
\| a \star b \|_{1, \nu} \leq \| a \|_{1, \nu} \, \| b \|_{1, \nu}. 
\]
Endow the Cartesian product 
\[
x= (x^1, \hdots, x^n) \in \ellnu \times \hdots \times \ellnu
\]
 with the norm
\[
\| x \|_{(\ellnu)^n}= \max_{1\leq i \leq n} \| x^i \|_{1,\nu}.
\]
\end{definition}

The following space also plays an important role.  
\begin{definition}
Let $b=\{b_k\}_{k\in\mathbb{Z}}$. We say that $b \in \ell_{\nu^{-1}}^\infty$ if
\[
\|b\|_{\infty,\nu^{-1}} \bydef \sup_{k\in \mathbb{Z}} \frac{|b_k|}{\nu^{|k|}} < \infty.
\]
A straightforward computation shows that 
$ \ell_{\nu^{-1}}^\infty$ is isomorphic to the dual space of $\ellnu$,  denoted by $(\ellnu)^*$.
\end{definition}
The space $\ell_{\nu^{-1}}^\infty$ is useful for obtaining bounds on certain linear 
operators, including the convolution product. 
Note that for any $a,b \in \ellnu$, we have that 
 $ a \in \ell_{\nu^{-1}}^\infty$ and
\[
\left| (a \star b)_k\right| \leq \|b\|_{1,\nu} \cdot \sup_{i \in \mathbb{Z}} \frac{|a_i|}{\nu^{|k-i|}} < \infty.
\]

Another important remark is that the $\ellnu$ norm of the 
Fourier coefficients provides an upper bound on
 the $C^0([0,T])$ norm. To see this, let
\[
g(t) = \sum_{k\in \mathbb{Z}} a_k e^{\im \frac{2\pi}{T}t} \quad \mbox{and} \quad h(t)= \sum_{k\in \mathbb{Z}} b_k e^{\im \frac{2\pi}{T}t},
\]
have $a, b \in \ell_\nu^1$
with $\|a-b\|_{1,\nu}<r$.  We have 
\begin{align*}
\|g-h\|_\infty = \sup_{t\in [0,T]} |g(t)-h(t)| &\leq \sum_{k\in\mathbb{Z}} \sup_{t\in [0,T]} \left| (a_k-b_k)e^{\im \frac{2\pi}{T}t} \right| \\
               &\leq \sum_{k\in\mathbb{Z}} \left| (a_k-b_k) \right| \\
               &\leq \sum_{k\in\mathbb{Z}} \left| (a_k-b_k) \right|\nu^{|k|} \\
               &= \|a-b\|_{1,\nu} < r
\end{align*}

\begin{remark}[Symmetry in sequence space]
If $f$ is real valued and $\{a_k\}_{k \in \mathbb{Z}}$ are the Fourier coefficients, then 
\[
a_{k} = \bar{a}_{-k},
\]
for all $k \in \mathbb{Z}$. In particular, the coefficient is real when $k=0$. This symmetry can be used to reduce computation times in Theorem \ref{thm:ContractionMapping}, 
and to increase the accuracy of a given finite dimensional approximation.
\end{remark}

\begin{remark}[Products of higher dimensions]
Let $f,g,h:\mathbb{T} \to \R$ be analytic periodic functions with the same period and Fourier expansion given respectively by the sequences $a,b,c \in \ellnu$. We note that their product can be computed using the Cauchy product again and that it is associative. It satisfies
\[
f(t)g(t)h(t) = \sum_{k\in \mathbb{Z}} \left[ (a \star b)\star c \right]_k e^{\im \frac{2\pi}{T}t}= \sum_{k\in \mathbb{Z}} \left[ a \star (b\star c) \right]_k e^{\im \frac{2\pi}{T}t} = \sum_{k\in \mathbb{Z}} \left( a \star b\star c \right)_k e^{\im \frac{2\pi}{T}t}.
\]
This remark can be extended to any higher degree product with a recursive application of the Cauchy product. So that it is possible to express any polynomial combination of Fourier expansions. 
\end{remark}

The following projections and needed for numerical applications.  

\begin{definition}
For $n,M \in \mathbb{N}$, define $\pi_M^n : (\ellnu)^n \to \mathbb{C}^{n(2M-1)}$ to be
the truncation of the components of an $n$-vector of bi-infinite sequences,
each to $M$ modes. 
For example, with $b \in \ellnu$, we have
\[
\pi_M^1(b) = (b_{-M+1},b_{-M+2},\hdots, b_{M-2},b_{M-1}). 
\]
Define also the inclusion maps $\iota_M^n: \mathbb{C}^{n(2M-1)} \to (\ellnu)^n$
so that $\pi_M^n \left(\iota_M^N(c)\right)= c$ for all $c \in \mathbb{C}^{n(2M-1)}$. 
\end{definition}

%
%
%
%

\subsection{Banach Algebras of Fourier-Taylor Sequences} 
We will approximate stable/unstable manifolds attached to periodic 
solutions using functions described by Fourier series in the periodic 
variable and by Taylor series in the stable or unstable variables.  
As above, we work in coefficient space.

\begin{definition}\label{def:FourierTaylorSpace}
Let $\alpha = (\alpha_1,\alpha_2) \in \mathbb{N}^2$, and $\sigma \in \mathbb{D}^2$. We set 
\[
|\alpha|= \alpha_1 +\alpha_2, \quad\mbox{and} \quad \sigma^\alpha = \sigma_1^{\alpha_1}\sigma_2^{\alpha_2}.
\]
We consider the Banach spaces
$X_2 =\{ x_\alpha \in \ellnu : \alpha \in \mathbb{N}^{2},  \mbox{ and } \|x\|_{X_2} < \infty \}$, 
where 
\begin{equation}\label{def:normX}
\|x\|_{X_2} \bydef \sum_{|\alpha|=0}^\infty \|x_\alpha\|_{1,\nu} < \infty.
\end{equation}
and the dual space $X^*$ with the norm
\[
\|x\|_{X_2^*} \bydef \sup_{\alpha \in \mathbb{N}^{2}} \|x_\alpha\|_{\infty,\frac{1}{\nu}}.
\]
Again, we can take Cartesian products of such spaces and endow 
$\left(X_2\right)^n$  with the norm 
\[
\|x\|_{(X_2)^n} = \max_{i=1,\hdots, n} \| x^i\|_{X_2}.
\]
\end{definition}

An element $a \in X_2$ is identified with the 
Fourier-Taylor coefficients of an analytic 
function  $g:[0,T]\times \mathbb{D}^2 \to \mathbb{C}$
periodic in its first variable 
 given by
\begin{equation}\label{eq:gDoubleSum}
g(t,\sigma)= \sum_{|\alpha|=0}^\infty \sum_{k \in \mathbb{Z}} a_{\alpha,k}e^{\im \frac{2\pi}{T}kt}\sigma^\alpha.
\end{equation}
Note that finiteness of the $X_2$ norm of $a$ implies that 
$g$ is analytic (and periodic) in $t$, and analytic in $\sigma$ on the 
unit disk $\mathbb{D}^2$.

Again, we have that $X_2$ is a Banach algebra with the Cauchy-Convolution 
product $\ast:X_2\times X_2 \to X_2$ defined below.
\begin{definition}[Cauchy-Convolution product]\label{def:product}
Let $b,c \in X_2$.  The Cauchy-Convolution, denoted by $\ast$, is given by 
\[
(b\ast c)_{\alpha,k} = \sum_{\substack{ \alpha_1, \alpha_2 \in \mathbb{N}^2  \\ \alpha_1 +\alpha_2 = \alpha}} (b_{\alpha_1} \star c_{\alpha_2})_k.
\]
\end{definition}
Again, the norm $\|\cdot\|_{X_2}$ bounds the  $C^0([0,T]\times \mathbb{D}^2)$ norm
%
following the same steps as in the $\ellnu$ case. 
Moreover, we have the following useful estimate.
\begin{proposition}[Norm estimates]\label{EstimatesBound}
Let $a,b \in X_2$. Then 
\begin{enumerate}
\item For all $k \in \mathbb{Z}$ and $\alpha \in \mathbb{N}^2$,
\[
\left| (a\ast b)_{\alpha,k} \right| \leq \|b\|_{X_2} \cdot \max_{0\leq |\beta|\leq |\alpha |}\sup_{i \in \mathbb{Z}} \frac{|a_{\beta,k-i}|}{\nu^{|i|}}.
\]
\item $\| a \ast b \|_{X_2} \leq  \| a \|_{X_2} \cdot  \| b \|_{X_2}$.
\end{enumerate}
\end{proposition}

\begin{remark}\label{remark:rescale}
The choice of $\sigma$ in the unit disk is to keep numerical stability in the computation of the norms. 
This is done without loss of generality as it is always possible to change the scale of the Taylor coefficients 
to obtain a radius of convergence of $1$.  When computing stable/unstable manifolds
attached to periodic orbits using 
the parameterization method, this is equivalent to choosing the scale of the 
stable/unstable eigenfunctions.
\end{remark}

\begin{definition}[The CRFBP in Fourier-Taylor space]\label{def:FinSpace}
Computing the coefficients of a parameterized manifold of the CRFBP requires to rewrite $F:\mathbb{R}^9 \to \mathbb{R}^9$, as defined in equation \eqref{eq:Fwithoutalpha}, in the appropriate space. For all $\alpha \in \mathbb{N}^2$ and $a \in \left(X_2\right)^9$, we set
\begin{equation*}
\small
\mathcal{F}_\alpha(a) =\begin{pmatrix}
  a_\alpha^{2}  \\
2a_\alpha^{4} +a_\alpha^{1} -
\displaystyle\sum_{i=1}^3 m_i\left(a^{1i}\star a^{6+i}\star a^{6+i} \star a^{6+i}\right)_\alpha \\ 
a_\alpha^{4} \\
-2a_\alpha^{2} +a_\alpha^{3} -
\displaystyle\sum_{i=1}^3 m_i\left(a^{3i}\star a^{6+i}\star a^{6+i} \star a^{6+i}\right)_\alpha \\ 
a_\alpha^{6} \\
 -\displaystyle \sum_{i=1}^3 m_i\left(a^{5i}\star a^{6+i}\star a^{6+i} \star a^{6+i}\right)_\alpha \\  \\
-\left(a^{11}\star a^{2}\star a^{7}\ast a^{7}\star a^{7}\right)_\alpha
-\left(a^{31}\star a^{4}\star a^{7}\star a^{7}\star a^{7}\right)_\alpha 
-\left(a^{51}\star a^{6}\star a^{7}\star a^{7}\star a^{7}\right)_\alpha \\
-\left(a^{12}\star a^{2}\star a^{8}\ast a^{8}\star a^{8}\right)_\alpha 
-\left(a^{32}\star a^{4}\star a^{8}\star a^{8}\star a^{8}\right)_\alpha
 -\left(a^{52}\star a^{6}\star a^{8}\star a^{8}\star a^{8}\right)_\alpha \\
-\left(a^{13}\star a^{2}\star a^{9}\ast a^{9}\star a^{9}\right)_\alpha
 -\left(a^{33}\star a^{4}\star a^{9}\star a^{9}\star a^{9}\right)_\alpha
  -\left(a^{53}\star a^{6}\star a^{9}\star a^{9}\star a^{9}\right)_\alpha \\
 \end{pmatrix}.
\end{equation*}
Each $a^{ij} \in X_2$ denotes the Fourier-Taylor expansion of the difference between $a^i$, the $i$-th coordinate of $a$, and the respective coordinate of the $j-$th primary. This notation does not define any additional variable, but it helps reducing the amount of convolution products in the presentation and in every computation.
\end{definition}

\subsection{A-posteriori validation: Newton-like operators} \label{sec:a-posteriori}

Suppose $X$ is a Banach space, $c\in X$ and $r>0$.  Then 
\[
B_r(c) \bydef \{ x \in X : \|x- c \|_X < r \}
\] 
denotes the ball of radius $r$ centered at the element $c$. The case of the unit disk, $c=0$ and $r=1$, will be denoted by $\mathbb{D}_X$. 
Let $\mathcal{B}(X,Y)$ denote the Banach algebra of bounded linear operators
from $X$ into $Y$, endowed with usual operator norm
\[
\|T\|_{\mathcal{B}(X,Y)} = \sup_{\substack{x \in X \\ \|x\|_X =1}} \| T(x) \|_Y ,
\]
for $T \in \mathcal{B}(X,Y)$. In the case $X=Y$, we write $\mathcal{B}(X)$.
The following is the basis of all our a-posteriori analysis.

\begin{theorem}\label{thm:ContractionMappingFixedPt}
Suppose that $X$ is a Banach space, and that $T:X \to X$ is a Fr\'{e}chet differentiable mapping. 
Assume that $Y$ is a positive constant and $Z:(0,r_*) \to [0,\infty)$  is a non-negative function satisfying
\begin{align*}
 \| T(\bar x) - \bar x \|_X &\leq Y \\
\sup_{x \in \overline{B_r(\bar x)}} \left\| DT(x) \right\|_{\mathcal{B}(X)} &\leq Z(r), \quad \mbox{for all} ~r \in (0,r_*) \\
\end{align*}
Define the polynomial
\[
P(r)= Z(r) -r +Y
\]
If there exists $r_0\geq 0$ such that $P(r_0)<0$ and $r_0 < r_*$, then there exists a unique $\tilde x \in B_r(\bar x) $ such that $T(\tilde x)= \tilde x$.  
\end{theorem}

The following is the basis of our a-posteriori analysis for periodic orbits and 
two point BVPs.

%
%
%
%
\begin{theorem}\label{thm:ContractionMapping}
Suppose that $X,Y$ are Banach spaces. Suppose moreover, that $F:X \to Y$ is a Fr\'{e}chet differentiable mapping. Suppose that $\bar x \in X$, $A^\dagger \in \mathcal{B}(X,Y)$, and $A \in \mathcal{B}(Y,X)$ with $A$ injective. Assume that $Y,~ Z_0,~Z_1$ are positive constants and $Z_2:(0,\infty) \to [0,\infty)$  is a non-negative function satisfying
\begin{align*}
 \| AF(\bar x) \|_X &\leq Y \\
\left\| \mbox{Id}_X -AA^\dagger \right\|_{\mathcal{B}(X)} &\leq Z_0 \\
\left\| A(DF(\bar x) - A^\dagger )  \right\|_{\mathcal{B}(X)} &\leq Z_1 \\
\left\| A(DF(c) -DF(\bar x)) \right\|_{\mathcal{B}(X)} &\leq Z_2(r)r, \quad \mbox{for all}~ c \in \overline{B_r(\bar x)}~\mbox{and} ~r>0
\end{align*}
Define the polynomial
\[
P(r)= Z_2(r)r^2 +(Z_0 +Z_1 -1)r +Y
\]
If there exists $r_0\geq 0$ such that $P(r_0)<0$, then there exists a unique $\tilde x \in B_r(\bar x) $ such that $F(\tilde x)= 0$.  Moreover, it follows that $DF(\tilde x)$ is invertible.  
\end{theorem}

The reader interested in the proof is referred to \cite{MR3612178}. 
In practice, the point $\bar x$ is a finite dimensional approximate zero of $F$
obtained numerically, while $A^\dagger$ is an eventually diagonal operator 
approximating $DF(\bar x)$, and  $A$ is an approximate inverse.

\section{Functional analytic set up for homological equations} \label{sec:lowOrderTerms}
Computer assisted existence proofs for the periodic orbit, the eigenfunctions and 
Floquet multipliers, solutions of the homological equations, and the solution of the 
two point boundary value problem for the connecting orbit segment exploit 
computer assisted Fourier and Chebyshev analysis which is 
explained in great detail in a number of references.  The works of 
\cite{HLM,Ransford,MR3612178,MR3896998,MR3353132}
are especially relevant to the present discussion.  

Since the technicalities have been discussed in many places, our main goal is
to describe the set up of the 
appropriate zero finding problems, with all necessary phase conditions,
unfolding parameters, and optimizations.  In particular, some technical 
discussion of the numerical solutions of the homological equation is needed.  
After this, the computer assisted proofs go through using the techniques of the 
references just mentioned.  
On the other hand, the tail bounds for the Fourier-Taylor expansions of 
the stable/unstable manifold parameterizations are 
more novel and are discussed in the next Section.

\subsection{Zeroth order: nonlinear equation for the periodic orbit} \label{sec:PO}
We now define a zero finding problem which isolates a unique periodic 
solution of the CRFBP.  This requires four phase conditions. First, a Poincar\'{e} condition to eliminate 
non-uniqueness due to the fact that the shift of a periodic orbit is again a periodic orbit. 
Second, three additional phase conditions which insure that the 
automatic differentiation leads to the right nonlinearity.  That is, these conditions 
impose the initial conditions associated with the appended ODEs describing the nonlinearities.   
Both conditions are enclosed in the 
function $\eta : (\ellnu)^9 \to \R^4$ given by 
\begin{equation}\label{def:Eta}
\eta(a_0) = \begin{pmatrix}
\langle u_1, u_0 -\gamma(0) \rangle \\
\left( \left(\gamma^1(0) -x_1\right)^2 + \left(\gamma^3(0)-y_1\right)^2 + \left(\gamma^5(0)-z_1\right)^2 \right)\left(\gamma^7(0)\right)^2 -1 \\
\left( \left(\gamma^1(0) -x_2\right)^2 + \left(\gamma^3(0)-y_2\right)^2 + \left(\gamma^5(0)-z_2\right)^2 \right)\left(\gamma^8(0)\right)^2 -1 \\
\left( \left(\gamma^1(0) -x_3\right)^2 + \left(\gamma^3(0)-y_3\right)^2 + \left(\gamma^5(0)-z_3\right)^2 \right)\left(\gamma^9(0)\right)^2 -1 \\
\end{pmatrix}.
\end{equation}
Note that the first entry is a Poincar\'{e} condition depending on 
$u_0 \in \mathbb{R}^9$, and $u_1= F(u_0)$.  Projecting onto Fourier coefficients,  we have 
\[
\gamma(0) = \displaystyle \sum_{k \in \mathbb{Z}} a_{0,k}.
\]
Each of these conditions is balanced by the addition of an unfolding 
parameter
$y= (y_1,y_2,y_3,y_4)$, incorporated into the vector field 
by the function $G: \mathbb{C}^4 \times (\ellnu)^9 \mapsto (\ellnu)^9$
given by
\[
(G(y,a_0))_k = \begin{pmatrix}
0 \\
y_1 a_{0,k}^2 \\
0 \\
0 \\ 
0 \\
0 \\
y_2 \cdot (a_0^7 \star a_0^7 \star a_0^7)_k \\
y_3 \cdot (a_0^8 \star a_0^8 \star a_0^8)_k \\
y_4 \cdot (a_0^9 \star a_0^9 \star a_0^9)_k
\end{pmatrix},~\mbox{for all}~ k \in \mathbb{Z}.
\]
We use $g$ to define the family of vector field $f_y(u)= f(u) +g(y,u)$.
One can prove that $f_y$ admits a solution if and only if $y= (0,0,0,0)$.
That is, if we prove that $f_y$ has a zero, then we obtain a-posteriori that 
the unfolding parameters are zero.  A proof of this fact is found in 
\cite{jpHalo}.
The appropriate zero finding operator is then 
\[
F_0^i(y,a_0) \bydef \begin{cases}
 \eta^i(a_0), &~\mbox{for}~ 1\leq i \leq 4,\\
\left\{ -\im \omega k a_{0,k}^{i-4} + \mathcal{F}_{0,k}^{i-4}(a_0)  + G_k^{i-4}(y,a_0) : k \in \mathbb{Z} \right\}, & ~\mbox{for}~ 5\leq i \leq 13,
\end{cases}
\]
where the first four coordinates are scalars 
and the last nine are elements of $\ell_{\nu'}^1$. Moreover, $\mathcal{F}$ denotes the evaluation of the vector field at a nine dimensional Fourier or Fourier-Taylor expansion, as given in definition \ref{def:FinSpace}.

We now define finite dimensional projection.  Let $K$ denote the truncation size of 
the Fourier expansion. This value is chosen so that all the Fourier coefficients 
ignored by the approximation are expected to be below machine precision. Then 
we compute, using a Newton scheme, 
$(\bar y, \bar a_0) \in \mathbb{C}^4 \times \mathbb{C}^{18K-9}$
an approximate zero of the truncated operator. 
This requires the projection $\pi_K^1$ of each of the infinite dimensional 
coordinate of $F_0$.  We denote by $\bar F_0$ such truncation. 
So
$\overline{F}_0 : \mathbb{C}^4 \times \mathbb{C}^{18K-9} \to \mathbb{C}^4 \times 
\mathbb{C}^{18K-9}$ and $ \overline{\mathcal{F}}_0(\bar y, \bar a_0) \approx 0$. 

Next we define $A^\dagger$ and $A$, two eventually diagonal operator approximating the derivative
$DF(\bar y, \bar a_0)$ and its inverse respectively, as in \cite{jpHalo}. 
Their exact definition is omitted in this case 
but is similar to the approach discussed in Section \ref{Sec:Chebyshev}.

We then apply Theorem \ref{thm:ContractionMapping} to obtain $r_0>0$ such that
\[
\|A_0 - \bar A_0 \|_\infty \leq \| a_0 - \bar a_0 \|_{(\ellnu)^9} < r_0.
\]
The definition of the operator is specific to the case $\alpha=0$, since it is only 
at this stage of the proof where the phase condition $\eta$ is necessary. 
Nevertheless, the other cases are treated using the same approach. 

\begin{remark}[Energy level of the periodic orbit $\gamma$]
Note that, because we are working with a system which preserves energy, 
we fix $\omega$ and then we solve for a periodic orbit with this frequency.
This exploits the fact that periodic solutions appear in one parameter families
parameterized by energy/frequency when we have conserved quantities.  
\end{remark}

\subsection{Eigenfunctions and multipliers: an almost linear system}

We now expand on the case $|\alpha|=1$, and solve for the 
Floquet multipliers and their associated eigenfunctions.  
In this case, solutions are unique after fixing the magnitude 
of the eigenfunction.  This introduces a phase condition fixing the 
value of the eigenfunction at $t = 0$.  
Choose $\xi_0 \in \mathbb{R}$ the desired initial value,
and $k_0$ the number of Fourier coefficients desired to approximate the initial value.
Set 
\begin{align}\label{eq:PhaseBundle}
\xi(a_\alpha) = \sum_{j=1}^9 \left( \sum_{|k|<k_0} a_{\alpha,k}^j\right)^2 -\xi_0.
\end{align} 
Note that $\xi : (\ellnu)^9 \to \mathbb{C}$ does not enforce that the magnitude of the 
eigenfunction is exactly $\xi_0$, but that this relaxed condition suffices to obtain uniqueness. 
We now define 
$F_\alpha : \mathbb{C} \times (\ellnu)^9 \mapsto  \mathbb{C} \times (\ell_{\nu'}^1)^9$, 
for $|\alpha|=1$, coordinate-wise by 
\[
F_\alpha^i(\lambda,a_{\alpha}) = \begin{cases}
\xi(a_\alpha), & \mbox{for}~ i=1, \\
\left\{ (-\im \omega k  -\lambda)a_{\alpha,k}^{i-1} + \mathcal{F}_{\alpha,k}^{i-1}(a)  : k \in \mathbb{Z} \right\} , &\mbox{for}~ 2\leq i \leq 10.
\end{cases}
\]

Note that, by definition of the Cauchy-Convolution product, the computation of $F_\alpha$ requires the coefficients of the periodic orbit $a_0$.
Such elements of smaller order are assumed to be fully known at this stage of the computation thanks to the zero-th order
validation.  In the case of the CRFBP, since we are interested in 
the case of complex conjugate Floquet multipliers, only one computation is required at this stage thanks to the symmetry
\begin{align*}
\lambda_1 &= \mbox{conj}(\lambda_2) \\
A_{1,0}(t) &= \mbox{conj}\left(A_{0,1}(t) \right).
\end{align*}

Again, defining the approximate derivative and approximate inverse for Theorem \ref{thm:ContractionMapping} is as in \cite{jpHalo}. 
We omit the details, and assume that the procedure is
applied successfully for all $|\alpha|=1$ to obtain $r_\alpha>0$ such that
\begin{equation}\label{eq:RalphaSupBound}
\|A_\alpha - \bar A_\alpha \|_\infty \leq \| a_\alpha - \bar a_\alpha \|_{(\ellnu)^9} < r_\alpha.
\end{equation}

\subsection{Order $2$ through $N_t$: non-constant coefficient inhomogeneous linear equations}

We now seek validated bounds on the approximate solutions of the 
homological equators defining the coefficients $a_\alpha$ with $2\leq |\alpha| \leq N_t$. 
Since these are inhomogeneous linear equations, and since we seek 
purely periodic solutions, this requires no additional phase conditions. 
It suffices to define $F_\alpha : (\ellnu)^9 \mapsto  (\ell_{\nu'}^1)^9$, by
\begin{equation}\label{eq:Fcal_alpha}
F_\alpha^i(a_\alpha) \bydef \left\{ (-\im \omega k  -\langle \alpha, \lambda \rangle )a_{\alpha,k}^i + \mathcal{F}_{\alpha,k}^{i}(a)  : k \in \mathbb{Z} \right\}.
\end{equation}
Again, the evaluation of $\mathcal{F}_\alpha$ requires the lower order data,
but $a_\alpha$ is the only variable or unknown. All Taylor coefficients of lower 
order being already controlled by the knowledge of a numerical approximation and 
a known constant $r_\alpha$ such that \eqref{eq:RalphaSupBound} is satisfied. 
We remark that most of the $Z_1$ and $Z_2(r)$ bounds computed to apply 
Theorem \ref{thm:ContractionMapping} in the case $\alpha=0$ are the same for higher 
order cases, and can be stored to speed up the validation process.
 Moreover, the validation for all coefficients of the same order can be executed 
 in parallel as they are independent. For each step of the validation, 
 we apply Theorem \ref{thm:ContractionMapping} to obtain bounds on the truncation 
 error in Fourier space for each coefficients $a_\alpha$ up to some desired order $N_t$. 

\begin{remark}[Error spreading in the node by node validation]
The truncation error is closely related to the computation of the bound $Y$ 
in Theorem \ref{thm:ContractionMapping}. This computation requires the 
evaluation of $\mathcal{F}_\alpha(a)$ using the exact coefficients of lower orders. 
This can be done with the help of interval arithmetic. However, this will cause the error 
to grow at every step of the process. Consequently, the application of the contraction 
mapping argument might fail before to reach the target value $N_t$. In such scenario, 
it is possible to rescale the parameterized manifold, that is to pick $\gamma<1$ and do 
the substitution $B_\alpha(t) = \gamma^{|\alpha|} A_\alpha(t)$ for all $\alpha$. This will reduce the truncation error for each 
Taylor coefficients already validated, and reduce the required Taylor truncation to reach the 
desired precision. Therefore lowering the targeted dimension $N_t$.
\end{remark}

By choice of $N_t$, the terms that have yet to be validated are expected to have
norm below a chosen threshold, usually a few dozen multiples machine precision
or less. The term not yet bounded in the difference between the approximate 
parameterization and the true solution is expected to be of similar magnitude 
as the chosen threshold. More specifically, we set 
\[
\bar P (t,\sigma) = \sum_{|\alpha|=0}^{N_t} 
\bar a_\alpha(t) \sigma^\alpha \bydef \sum_{|\alpha|=0}^{N_t}
 \sum_{|k|<K} \bar a_{\alpha,k} e^{\im \omega k t} \sigma^\alpha. 
\] 
Then, thanks to the previous validations, we have
\begin{align}
\|P - \bar P\|_\infty &\leq \left\| \sum_{|\alpha|=0}^{N_t}  a_\alpha(t) \sigma^\alpha  
- \bar P(t,\sigma)\right\|_\infty 
+ \left\|\sum_{|\alpha|=N_t+1}^\infty  a_\alpha(t) \sigma^\alpha \right\|_\infty  \nonumber \\
& \leq \sum_{|\alpha|=0}^{N_t} r_\alpha  
+ \left\|\sum_{|\alpha|=N_t+1}^\infty  a_\alpha(t) \sigma^\alpha \right\|_\infty. \label{Error:finite}
\end{align}

Our objective is now to obtain an upper bound on the norm of the tail, 
which is the last term of \eqref{Error:finite}. Thanks to the previous remark, the 
approximation $\bar A_\alpha (t) = 0$ should be a good approximation around 
which to construct another fixed point argument. This is the goal of the next section.

%
%
%
%
%
%
%
%
%

\subsection{Overcoming the data dependance: rewriting for the Coefficients $|\alpha|>N_t$}\label{Sec:RewritingTail}

In this section, we rearrange \eqref{eq:Fcal_alpha} to express the Fourier-Taylor coefficient 
of a given order as the fixed point of a bounded operator defined only on coefficients of 
lower order. This is possible, as all the non-linear terms in the 
evaluation of $\mathcal{F}_\alpha$ are given by Cauchy-Convolution products:
a rearrangement of the sum will provide the desired outcome. 
This is the object of the following definition. 

\begin{definition}[Rewriting of the Cauchy-Convolution product]\label{def:hatStar}
Let $a,b \in X_2$.  For $k \in \mathbb{Z}$, and $\alpha \in \mathbb{N}^2$ with $|\alpha| > 0$, define 
\begin{equation}\label{eq:HatStar}
(a \hat \ast b)_{\alpha,k} = \sum_{\substack{\alpha_1 +\alpha_2 = \alpha \\ |\alpha_1|,|\alpha_2| > 0 }} (a_{\alpha_1} \star b_{\alpha_2})_k.
\end{equation}
Then it is possible to separate terms involving $\alpha=0$ from the Cauchy-Convolution product. 
That is
\[
(a \ast b)_{\alpha,k} =  (a_0 \star b_\alpha)_k + (a_\alpha \star b_0)_k + (a \hat \ast b)_{\alpha,k}.
\]
The indexing values $\alpha_1$ and $\alpha_2$ in \eqref{eq:HatStar} will never reach $\alpha$,
as this would require the other index to be zero, but such case is excluded by definition. 
In other words, $(a \hat \ast b)_{\alpha}$ does not depend on 
$a_\alpha$ or $b_\alpha$, as desired.
\end{definition}

In practice, this definition is useful to compute the exact reminder $R_\alpha(a)$ in the expression
\[
\mathcal{F}_{\alpha,k}(a_\alpha) = \left(D\mathcal{F}_0(a_0) a_\alpha \right)_k +R_{\alpha,k}(a).
\]
Note that this expression lets us express the homological equation so
that $R_{\alpha,k}(a)$ is the only non linear term, but in such a way that it is 
independent of $a_\alpha$. This rewriting is 
straightforward, and we refer the interested reader to \cite{Paper1} for 
an explicit step-by-step example.
In essence, the non-linear terms are broken down into two terms thanks to 
Definition \ref{def:hatStar}. In the case of the CRFBP, the reminder is
\begin{equation*}
\small
R_{\alpha,k}(a)= \begin{pmatrix}
 0 \\
 -\displaystyle \sum_{i=1}^3 m_i\left(a^{1i} \hat \ast a^{6+i} \hat \ast a^{6+i} \hat \ast a^{6+i}\right)_{\alpha,k} \\ 
0\\
 -\displaystyle \sum_{i=1}^3 m_i\left(a^{3i} \hat \ast a^{6+i} \hat \ast a^{6+i} \hat \ast a^{6+i}\right)_{\alpha,k} \\
0 \\
 -\displaystyle \sum_{i=1}^3 m_i\left(a^{5i} \hat \ast a^{6+i} \hat \ast a^{6+i} \hat \ast a^{6+i}\right)_{\alpha,k} \\
-\left(a^{11} \hat \ast a^{2} \hat \ast a^{7} \hat \ast a^{7} \hat \ast a^{7} + a^{31} \hat \ast a^{4} \hat \ast a^{7} \hat \ast a^{7} \hat \ast a^{7} + a^{51} \hat \ast a^{6} \hat \ast a^{7} \hat \ast a^{7} \hat \ast a^{7}\right)_{\alpha,k} \\
-\left(a^{12} \hat \ast a^{2} \hat \ast a^{8} \hat \ast a^{8} \hat \ast a^{8} + a^{32} \hat \ast a^{4} \hat \ast a^{8} \hat \ast a^{8} \hat \ast a^{8} + a^{52} \hat \ast a^{6} \hat \ast a^{8} \hat \ast a^{8} \hat \ast a^{8}\right)_{\alpha,k} \\
-\left(a^{13} \hat \ast a^{2} \hat \ast a^{9} \hat \ast a^{9} \hat \ast a^{9} + a^{33} \hat \ast a^{4} \hat \ast a^{9} \hat \ast a^{9} \hat \ast a^{9} + a^{53} \hat \ast a^{6} \hat \ast a^{9} \hat \ast a^{9} \hat \ast a^{9}\right)_{\alpha,k} \\
 \end{pmatrix}.
\end{equation*}
Again, no term of order $\alpha$ is required for the computation of 
$R_\alpha$. Therefore, for $|\alpha|>2$, we see that the zero of the 
operator given in \eqref{eq:Fcal_alpha} is also a solution of the equivalent 
linear (in terms of $a_\alpha$) problem 
\begin{equation}\label{eq:ReminderEquation}
(-\im \omega k  -\langle \alpha, \lambda \rangle )a_{\alpha,k} +\left( D\mathcal{F}_0(a_0)a_{\alpha} \right)_k = -R_{\alpha,k}(a), \quad \forall k \in \mathbb{Z}.
\end{equation}

We now show that the linear operator on the left-hand side is boundedly invertible, 
so that $a_\alpha$ can be expressed as the fixed point of some operator. We first 
denote the diagonal part of the operator exhibited on the left of 
equation \eqref{eq:ReminderEquation} by $\mathcal{L}_{\alpha}$. 
It is defined for any $h \in (\ellnu)^9$ as
\[
(\mathcal{L}_{\alpha}h)_k =  (-\im \omega k - \langle \alpha,\lambda \rangle)h_k.
\]
The operator $\mathcal{L}_\alpha$ is the dominant factor in equation 
\eqref{eq:ReminderEquation}, thus we will apply its inverse to both side 
of \eqref{eq:ReminderEquation}. The inverse exists as the Floquet exponents
 are assumed to be non-resonant, and it is given by 
\[
(\mathcal{L}_{\alpha}^{-1}h)_k =  \frac{h_k}{(-\im \omega k - \langle \alpha,\lambda \rangle)}.
\]
Hence, for all $\alpha$, the solution of equation \eqref{eq:ReminderEquation} 
also satisfies
\begin{equation}\label{eq:Gdefinition}
\left( \left[ \mbox{Id} +\mathcal{L}_\alpha^{-1}\circ D\mathcal{F}_0(a_0) \right] a_\alpha \right)_k  = \left(\mathcal{L}_\alpha^{-1}\circ R_\alpha(a) \right)_k, \quad \forall k \in \mathbb{Z}.
\end{equation}
We note that for all $\alpha$,
\begin{equation}\label{eq:NormLinverse}
\left\| \mathcal{L}_\alpha^{-1} \right\|_{\mathcal{B}\left((\ellnu)^9\right)} \leq \sup_{k \in \mathbb{Z}} \left| \frac{1}{(-\im \omega k - \langle \alpha,\lambda \rangle)} \right| \leq \frac{1}{|\alpha|\left|R_\lambda\right|} < \infty,
\end{equation}
where $R_\lambda = \mbox{Re}(\lambda_1) = \mbox{Re}(\lambda_2)  \neq 0$.
We see that the denominator of the last term increases to $\infty$ as $|\alpha| \to \infty$. 
Now, we use the left hand side of equation \eqref{eq:Gdefinition} to define 
$\mathcal{G}_{\alpha}: (\ellnu)^9 \to (\ellnu)^9$, whose inverse is obtained with a 
Neumann type argument. 
That is, for all $\alpha$, the existence of $\mathcal{G}_\alpha^{-1}$ will be stated with 
an explicit upper bound on its norm. This is sufficient to apply the contraction mapping
argument. The remainder of this section is devoted to the proof of  the following Theorem.

\begin{theorem}\label{Thm:InverseG}
Consider $\mathcal{G}_{\alpha}$, with $|\alpha|> N_t$, defined by
\[
 \mbox{Id} +\mathcal{L}_\alpha^{-1} D\mathcal{F}_0(a_0).
\]
Then $\mathcal{G}_{\alpha}: (\ellnu)^9 \to (\ellnu)^9$ is invertible for all $\alpha$. Moreover, 
$a_\alpha$ satisfies
\[
a_\alpha  = - \mathcal{G}_\alpha^{-1}
 \mathcal{L}_{\alpha}^{-1} R_\alpha(a),~\mbox{for all}~ |\alpha|>N_t,
\]
and the operator $\mathcal{G}_\alpha^{-1} \mathcal{L}_{\alpha}^{-1} R_\alpha : 
(\ellnu)^9 \to (\ellnu)^9 $ is Fr\'{e}chet differentiable for each $\alpha$.
\end{theorem}

Using Proposition \ref{EstimatesBound} and the triangle inequality, one verifies 
that $G_\alpha(h) \in (\ellnu)^9$ for any $h \in (\ellnu)^9$, and the first statement of the
Theorem is verified. Moreover,  Fr\'{e}chet differentiability follows from linearity of 
$\mathcal{G}_\alpha^{-1}$ and $\mathcal{L}_\alpha^{-1}$, combined with the 
fact that $R_\alpha$ is Fr\'{e}chet differentiable (it is a polynomial function in a 
Banach algebra). Finally, it follows from equation \eqref{eq:NormLinverse}
 that there exist $M$ such that for all $|\alpha|>M$, we have that
\[
\left\| \mathcal{L}_\alpha^{-1} \right\|_{\mathcal{B}\left((\ellnu)^9\right)}
 \cdot \left\|D\mathcal{F}_0(a_0)\right\|_{\mathcal{B}\left((\ellnu)^9\right)} < 1.
\]
Consequently, for $|\alpha|> M$ we have that
\begin{align}\label{eq:BoundNeumann1}
\left\|\mathcal{G}_\alpha^{-1}\right\|_{\mathcal{B}((\ellnu)^9)} \leq \frac{1}{1-\left\| \mathcal{L}_\alpha^{-1} \right\|_{\mathcal{B}((\ellnu)^9)} \cdot \left\|D\mathcal{F}_0(a_0)\right\|_{\mathcal{B}\left((\ellnu)^9\right)}}.
\end{align}
This provides the desired formulation, but it is not granted that $M\leq N_t$, the size of 
the truncation. So for $\alpha$ such that $N_t < |\alpha| \leq M$, we have to adapt the argument. 
Since there are only finitely many such cases, we use the computer to apply a 
different Neumann series argument for each of the values individually. This approach also 
leads to a method for computing an upper bound on 
$\|D\mathcal{F}_0(a_0)\|_{\mathcal{B}((\ellnu)^9)}$ with the help of the numerical 
approximation $\bar a_0$ and the following definition.

\begin{definition}[Separation of the convolution product]\label{def:ProductSplit}
Let $K$ denote the size of the finite dimensional projection which was used to apply Theorem \ref{thm:ContractionMapping} and obtain the upper bound $r_0$ on the truncation error. We 
express each component of $a_0$ as a sum of two elements, first an element containing 
all the component of order $|k|<K$, second an elements containing the higher order 
coefficients. We define $a_0^K$ by 
\[
a_0^K = \iota_K^9 \circ \pi_K^9 (a_0),
\]
and its complement $a_0^\infty = a_0 - a_0^K$. It follows by definition that $\left(a_0^\infty\right)_k = 0$ whenever $|k|<K$, and
\[
\left\| a_0^\infty \right\|_{(\ellnu)^9} < r_0.
\]
This inequality will be used to show that the operator 
$\mathcal{L}_\alpha^{-1}\circ D\mathcal{F}_0(a_0)$ 
has a small norm outside of a finite number of terms. 
To simplify the notation, set
\begin{align*}
\bar \star : &\ellnu \times \ellnu \to \ellnu \\
            &(a,b) \mapsto \iota_K^1(\pi_K(a)) \star \iota_K^1(\pi_K(b)) = (a^K \star b^K),
\end{align*}
and note that the resulting element has only finitely many non-zero entries.
In short, it is an element of $\ellnu$. By the bi-linearity of $\star$, for any 
$a,b \in \ellnu$ and for all $k \in \mathbb{Z}$, we have 
\[
(a \star b)_k = (a ~\bar\star~ b)_k + (a^K \star b^\infty)_k + (a^\infty \star b^K)_k +(a^\infty \star b^\infty)_k.
\]
This observation is extended to higher order products. In practice, quartic and quintic products
generate $16$ and $32$ terms respectively, however not all of these require careful consideration.
 Indeed, all but two include a term of the form $(a_0^\infty)^i$, whose norm is bounded by $r_0 \ll 1$. 
Repeated application of Proposition \eqref{EstimatesBound} and the triangle inequality leads to
an upper bound that we illustrate by considering a generic quartic term.  A more detailed example is given
in Section \ref{Ad:Dbounds}. Let $a,b,c,h \in \ellnu$ with $\|h\|_{1,\nu}=1$ and 
$\|a^\infty\|,\|b^\infty\|,\|c^\infty\|<r_0$.  Then there is $p(r_0)$, a polynomial of degree three with positive coefficients, such that $p(0)=0$ satisfying
\begin{align*}
\left\| a \star b \star c \star h \right\|_{1,\nu} &\leq \left\| a^K \star b^K \star c^K \star h^K \right\|_{1,\nu}
 +\left\| a^K \star b^K \star c^K \star h^\infty \right\|_{1,\nu}  +p(r_0) \\
&= \left\| a ~\bar\star~ b ~\bar\star~ c ~\bar\star~ h \right\|_{1,\nu}
 +\left\| a^K \star b^K \star c^K \star h^\infty \right\|_{1,\nu}  +p(r_0) \\
\end{align*}
We note from this inequality that any operator with non-linearities defined using the convolution products $\star$ can be broken down into a finite dimensional part using 
only $\bar \star$, and "reminder" close to the zero operator.
\end{definition}

We use Definition \ref{def:ProductSplit} to construct $B^K : (\ellnu)^9 \to (\ellnu)^9$,
 the truncation of $D\mathcal{F}_0(a_0)$.  For $h \in (\ellnu)^9$, this is given by
\begin{equation*}
\small
(B^K h)_{k}^i = \begin{cases}
  (h_k^K)^{2}, &\mbox{if}~i=1,  \\
 2(h_k^K)^{4} +(h_k^K)^{1} -\displaystyle\sum_{i=1}^3 
 m_i\left[ 3\left(a_0^{1i}\bar\star a_0^{6+i}\bar\star a_0^{6+i} 
 \bar\star  h^{6+i}\right)_k +\left(h^{1}\bar\star  
 a_0^{6+i}\bar\star a_0^{6+i}\bar\star a_0^{6+i}\right)_k\right] &\mbox{if}~i=2,  \\ 
(h_k^K)^{4} &\mbox{if}~i=3, \\
-2(h_k^K)^{2} +(h_k^K)^{3} -\displaystyle\sum_{i=1}^3 m_i
\left[ 3\left(a_0^{3i}\bar\star a_0^{6+i}\bar\star a_0^{6+i}
 \bar\star h^{6+i}\right)_k +\left(h^{3}\bar\star a_0^{6+i}
 \bar\star a_0^{6+i}\bar\star  a_0^{6+i}\right)_k\right] &\mbox{if}~i=4,  \\
(h_k^K)^{6} &\mbox{if}~i=5, \\
-\displaystyle\sum_{i=1}^3 m_i\left[ 3\left(a_0^{5i}\bar\ast a_0^{6+i}\bar\star a_0^{6+i} 
\bar\star h^{6+i}\right)_k +\left(h^{5}\bar\star a_0^{6+i}\bar\star a_0^{6+i}
\bar\star a_0^{6+i}\right)_k\right] &\mbox{if}~i=6,  \\
\end{cases}
\end{equation*}
and
\begin{align*}
(B^Kh)^i_k &= \left(h^1 \bar \star a_0^2 \bar \star a_0^i \bar \star a_0^i \bar \star a_0^i \right)_k +\left(a_0^1 \bar \star h^2 \bar \star a_0^i \bar \star a_0^i \bar \star a_0^i \right)_k + 3\left(a_0^1 \bar \star a_0^2 \bar \star h^i \bar \star a_0^i \bar \star a_0^i \right)_k \\
&+\left(h^3 \bar \star a_0^4 \bar \star a_0^i \bar \star a_0^i \bar \star a_0^i \right)_k 
+\left(a_0^3 \bar \star h^4 \bar \star a_0^i \bar \star a_0^i \bar \star a_0^i \right)_k
 + 3\left(a_0^3 \bar \star a_0^4 \bar \star h^i \bar \star a_0^i \bar \star a_0^i \right)_k \\
&+\left(h^5 \bar \star a_0^6 \bar \star a_0^i \bar \star a_0^i \bar \star a_0^i \right)_k
 +\left(a_0^5 \bar \star h^6 \bar \star a_0^i \bar \star a_0^i \bar \star a_0^i \right)_k 
 + 3\left(a_0^5 \bar \star a_0^6 \bar \star h^i \bar \star a_0^i \bar \star a_0^i \right)_k, 
 \quad \mbox{for}~ i=7,8,9. \\
\end{align*}
Note that this can be represented by a finite matrix multiplication,
as $(B^Kh)_k = 0$ for $|k|>5K$. Additionally, following the remark in 
Definition \ref{def:ProductSplit}, it is possible to define $B^\infty$ such that
\[
D\mathcal{F}_0(a_0)h = B^Kh + B^\infty h,~ \forall h\in (\ellnu)^9.
\]  
By construction, $\mathcal{L}_{\alpha}^{-1}\circ B^\infty$ has small norm, providing 
another justification that $\mathcal{G}_\alpha$ is invertible. 
Set $M_\alpha= \mbox{Id} + \mathcal{L}_\alpha^{-1} B^K$. This is eventually the 
identity for all $\alpha$, as $B^K$ is eventually zero. So that it is invertible when the finite part is, 
which can be verified with the help of a computer. Assume that this stands true, 
so that the inverse exists and is approximated numerically by 
$M_\alpha^{\dagger}$. Using an approach similar to the computation of the bound $Z_0$ 
in Theorem \ref{thm:ContractionMapping}, we have $\| \mathcal{E}_\alpha \| \ll 1$ such that
\[
M_\alpha^\dagger M_\alpha = \mbox{Id} +\mathcal{E}_\alpha.
\]
It is worth while to rewrite the problem once more, so that 
\begin{align*}
\mathcal{G}_\alpha = \mbox{Id} +\mathcal{L}_{\alpha}^{-1}D\mathcal{F}_0(a_0)  &= \mbox{Id} +\mathcal{L}_{\alpha}^{-1} B^K +\mathcal{L}_{\alpha}^{-1} B^\infty \\
&= M_\alpha +\mathcal{L}_{\alpha}^{-1} B^\infty \\
\end{align*}
and
\begin{align*}
\left[ M_\alpha +\mathcal{L}_{\alpha}^{-1} B^\infty \right]  &= -\mathcal{L}_\alpha^{-1} \mathcal{R}_\alpha \quad \Longleftrightarrow\\
M_{\alpha}^\dagger \left[ M_\alpha +\mathcal{L}_{\alpha}^{-1} B^\infty \right] &= -M_{\alpha}^\dagger \mathcal{L}_\alpha^{-1} \mathcal{R}_\alpha \quad \Longleftrightarrow \\
\left[ \mbox{Id} +\mathcal{E}_\alpha + M_\alpha^\dagger \mathcal{L}_{\alpha}^{-1} B^\infty \right] &= -M_{\alpha}^\dagger \mathcal{L}_\alpha^{-1} \mathcal{R}_\alpha
\end{align*}
The last equation suggests the use of a Neumann argument. We note that the 
left is invertible whenever
\[
\left\| \mathcal{E}_\alpha + M_\alpha^\dagger \mathcal{L}_{\alpha}^{-1} B^\infty \right\|_{\mathcal{B}\left( (\ellnu)^n \right)} \leq \left\| \mathcal{E}_\alpha \right\|_{\mathcal{B}\left( (\ellnu)^n \right)} +\left\| M_\alpha^\dagger \mathcal{L}_{\alpha}^{-1} B^\infty \right\|_{\mathcal{B}\left( (\ellnu)^n \right)} < 1.
\]  
The only term whose norm is not already known (or bounded with computer assistence) 
is $M_\alpha^\dagger\mathcal{L}_{\alpha}^{-1} B^\infty$. The first element,
 $M_\alpha^\dagger$, is an eventually diagonal operator 
 (whose norm is bounded using the computer). 
We focus on the remaining terms in a single step. To lighten the notation, 
define $B_{ij}^\infty$ for $1\leq i,j \leq 9$ so that
\[
\left( B^\infty h \right)^i = \sum_{j=1}^9 B_{i,j}^\infty h^j, ~\forall 
h \in (\ellnu)^9 ~\mbox{and}~ 1\leq i \leq 9.
\]
This represents the component-wise definition of the difference between the 
derivative $D\mathcal{F}_0$ and $B^K$. That is, for all $h \in (\ellnu)^9$, 
we have 
\[
B_{ij}^\infty h^j = D_{a_0^j}\mathcal{F}^i(a_0)h^j - \left(B^Kh^{e_j}\right)^i,
\]
where $h^{e_j} \in (\ellnu)^9$ is equal to $h$ in the $j-$th component and zero elsewhere.
 By construction, $B_{ij}^\infty$ represents the infinite part of $D_{a^j}\mathcal{F}^i(a_0)$. 
 We follow definition \ref{def:ProductSplit}, for each $B_{ij}^\infty$, to 
 compute a polynomial $p_{i,j}(r_0)$ such that for $h \in \ellnu$
\[
\left\| B_{ij}^\infty h \right\|_{1,\nu} \leq \| c_{ij} \star h^\infty \|_{1,\nu} +p_{i,j}(r_0),
\]
where $c_{ij}$ is a convolution product whose factor are all truncated to order $K$. Then, each entries
$(c_{i,j} \star h^\infty)_k$ of the convolution up to order $K$ is bounded using 
Proposition \eqref{EstimatesBound} for best accuracy. Note that the action of 
$\mathcal{L}_\alpha^{-1}$ is equal in every component and can be expressed 
using a component wise scalar multiplication. For simplicity, set
\begin{equation}\label{eq:Ltermwise}
d_{\alpha,k} =   \frac{1}{-\im \omega k - \langle \alpha,\lambda \rangle}, \quad \mbox{and} \quad d_\alpha^\infty= \sup_{|k| \geq K} \left| \frac{1}{-\im \omega k - \langle \alpha,\lambda \rangle} \right|.
\end{equation}
This leads to an upper bound on each component of $\mathcal{L}_\alpha^{-1}B^\infty$.
We will study one specific case, as the others are also convolution product and treated similarly. We take the convolution product $c_{71}= (a_0^K)^2 \star (a_0^K)^7 \star (a_0^K)^7 \star (a_0^K)^7$ to illustrate the case $B_{71}^\infty$, which involves only this product. Let $h \in (\ellnu)^9$ with $\|h\|_{(\ellnu)^9} = 1$. Then
\begin{align}
\left\| \left( \mathcal{L}_{\alpha}^{-1} B^\infty h^{e_1} )\right)^7 \right\|_{(\ellnu)^9} 
&= \sum_{k \in \mathbb{Z}} \left| d_{\alpha,k} \left( B_{71}^\infty h^1\right)_k \right| \nu^{|k|} \nonumber \\
&\leq \sum_{k \in \mathbb{Z}} \left| d_{\alpha,k} \left(c_{71}\star (h^1)^\infty \right)_k  
\right| \nu^{|k|} +\left\|\mathcal{L}_{\alpha}^{-1}\right\|_{\mathcal{B}(\ellnu)} p_{71}(r_0) \nonumber \\
&\leq \sum_{|k|<K} \left| d_{\alpha,k}\left( c_{71} \star (h^1)^\infty \right)_k \right|\nu^{|k|} 
+ d_\alpha^\infty \|c_{71}\|_{1,\nu} +\left\|\mathcal{L}_{\alpha}^{-1}
\right\|_{\mathcal{B}(\ellnu)}p_{71}(r_0). \label{def:BoundDij} 
\end{align}
The remaining sum is bounded using Proposition \ref{EstimatesBound}, so that the desired bound is fully determined. 

Note that the case just discussed includes a single convolution product. But other cases can be considered similarly after using the triangle inequality to separate each convolution. 
Moreover, some other cases such as $B_{12}^\infty,B_{34}^\infty,B_{56}^\infty$, include linear terms. For such scenario, we note that
\[
\left\| \mathcal{L}_\alpha^{-1} h^\infty \right\|_{1,\nu}= \sum_{|k|\geq K} \left|\frac{(h^\infty)_k}{-\im \omega k -\langle \alpha,\lambda \rangle} \right| \nu^{|k|} \leq d_\alpha^\infty.
\]

Now, we are able to state the final upper bound on the norm of 
$\mathcal{L}_\alpha^{-1} B^\infty$. Let $\mathcal{Z}_\alpha^{ij}$ 
denote the bound defined in equation \eqref{def:BoundDij},
 so that for any $\alpha$ we have
\[
\left\| \mathcal{L}_\alpha^{-1} B^\infty \right\|_{\mathcal{B}((\ellnu)^9)} \leq \max
\begin{pmatrix}
d_\alpha^\infty \\
3d_\alpha^\infty + \mathcal{Z}_\alpha^{21} +m_1 \mathcal{Z}_\alpha^{27} 
+m_2 \mathcal{Z}_\alpha^{28} +m_3 Z_\alpha^{29} \\
d_\alpha^\infty \\
3d_\alpha^\infty + \mathcal{Z}_\alpha^{43} +m_1 \mathcal{Z}_\alpha^{47} 
+m_2 \mathcal{Z}_\alpha^{48} +m_3 \mathcal{Z}_\alpha^{49} \\
d_\alpha^\infty \\
\mathcal{Z}_\alpha^{65} +m_1 \mathcal{Z}_\alpha^{67} 
+m_2 \mathcal{Z}_\alpha^{68} +m_3 \mathcal{Z}_\alpha^{69} \\
\mathcal{Z}_\alpha^{77} +\sum_{j=1}^6 \mathcal{Z}_\alpha^{7j} \\
\mathcal{Z}_\alpha^{88} +\sum_{j=1}^6 \mathcal{Z}_\alpha^{8j} \\
\mathcal{Z}_\alpha^{99} +\sum_{j=1}^6 \mathcal{Z}_\alpha^{9j}
\end{pmatrix} 
\]
Since an upper bound on this term can be calculated with the assistance of a computer, we can verify the criteria 
\begin{align}\label{eq:FiniteCriteria}
\left\|\mathcal{E}_\alpha\right\|_{\mathcal{B}\left( (\ellnu)^n \right)}  
+\left\|M_\alpha^\dagger\right\|_{\mathcal{B}\left( (\ellnu)^n \right)} 
\cdot \left\| \mathcal{L}_{\alpha}^{-1}\circ B^\infty 
\right\|_{\mathcal{B}\left( (\ellnu)^n \right)} < 1, 
\quad \forall \alpha ~\mbox{with}~N_t< |\alpha| \leq M,
\end{align}
This approach is computationally more costly than the first approach, and can only be done finitely 
many times, but this argument is utilized only in the cases where \eqref{eq:BoundNeumann1} is not satisfied.  

Assume now that equation \eqref{eq:FiniteCriteria} is verified (with computer assistance), so that
\begin{align}\label{eq:BoundNeumann2}
\|\mathcal{G}_\alpha^{-1}\|_{\mathcal{B}\left( (\ellnu)^n \right)}
 \leq \frac{\left\|M_\alpha^\dagger 
 \right\|_{\mathcal{B}\left( (\ellnu)^n \right)} }{1 - \left\|\mathcal{E}_\alpha 
 \right\|_{\mathcal{B}\left( (\ellnu)^n \right)}-
  \left\|M_\alpha^\dagger\right\|_{\mathcal{B}\left( (\ellnu)^n \right)} 
  \cdot \left\| \mathcal{L}_{\alpha}^{-1}\circ B^\infty \right\|_{\mathcal{B}\left( (\ellnu)^n \right)}}, 
\end{align}
for $\alpha ~\mbox{with}~N_t< |\alpha| \leq M$.  Set $B_G$ to represent the maximum value 
taken by the bounds given in \eqref{eq:BoundNeumann1} 
and \eqref{eq:BoundNeumann2}.  It follows that $B_G$ is finite and
\begin{equation}\label{eq:TboundNorm}
\|\mathcal{G}_\alpha^{-1}\|_{\mathcal{B}\left( (\ellnu)^9 \right)} \leq B_G ,\quad \forall |\alpha|>N_t.
\end{equation}
This complete the proof of Theorem \ref{Thm:InverseG}.

\subsection{Global connection: two point boundary value problem for a homoclinic orbit}\label{Sec:Chebyshev}

Recall that the CRFBP preserves the
Jacobi integral $H$ given in Equation \eqref{eq:Energy}.
So, if $\gamma \colon [0,T] \to \mathbb{R}^6$ is a periodic 
solution of $\dot u = f(u)$, then there is a $K \in \mathbb{R}$ 
so that $K= H(\gamma(t))$ for all $t$. Moreover if
\[
\mathcal{K} \bydef \left\{ u \in \R^6 : H(u)= K \right\},
\]
is the energy level set associated with $\gamma$, 
then $W_{s,u}(\gamma) \subset \mathcal{K}$. This follows from the continuity of $H$.

Note that $\mathcal{K}$ is a five dimensional smooth manifold except at 
points where $\nabla H(u)$ vanishes.
Since $f$ conserves $H$, we have that $\nabla H(\bar u) = 0$ if and only if $f(\bar u) = 0$.
Moreover, a periodic orbit, its local stable/unstable manifolds, and any orbits homoclinic
to the periodic orbit are uniformly bounded away from the equilibrium set.
It follows that $\gamma$ and its homoclinic orbits live in the five dimensional energy 
manifold $\mathcal{K}$.  
If $W^u(\gamma)$ and $W^s(\gamma)$ intersect at a point $p \in \mathcal{K}$, we say that 
they intersect transversally relative to the energy manifold $\mathcal{K}$ if the 
tangent spaces $T_p W^s(\gamma)$ and $T_p W^u(\gamma)$ span the tangent space $T_p(\mathcal{K})$.
Such an intersection is robust with respect to conservative perturbations.

Solving two point boundary value problems in energy manifolds for conservative systems
leads to a well know degeneracy in the equations. More precisely, the presence of the 
conserved quantity leads to one fewer unknown than equations.  
 This degeneracy is typically overcome by 
one of three strategies: using the energy functional to eliminate an equation,
 exploiting some symmetry of the problem and restricting to connecting orbits which preserve 
 this symmetry, or by introducing an 
``unfolding parameter''.  We follow this last course and exploit the following lemma, 
whose justification is discussed in much greater detail in 
\cite{collisionPaper}.  See also 
\cite{MR1992054}.

\begin{lemma}
Let $G_\beta : \mathbb{R}^9 \to \mathbb{R}^9$ defined by
\[
G_\beta(u) = (0,\beta u_2, 0, \beta u_4, 0, \beta u_6, 0,0,0).
\]
Then, the boundary value problem
\begin{align}\label{BVP:Chebyshev}
\begin{cases}
\dot v (t) = F(v(t)) +G_\beta(v(t)), & \forall t \in (0,2L) \\
v(0)= Q(\theta_u,\phi_u), \\
v(2L)= P(\theta_s,\phi_s),
\end{cases}
\end{align}
admits solutions if and only if $\beta = 0$.
\end{lemma}

The intuition here is that the function $G$ introduces a dissipative
perturbation which breaks the conservative structure.  Indeed, 
one can show that when $\beta \neq 0$ the system ``loses energy'' 
along orbits.  Since the images of $P$ and $Q$ are in the 
energy level of $\gamma$, an orbit segment can begin on one
and end on the other only when the perturbation is not present.  
Again, a rigorous proof is given in \cite{collisionPaper}.

We employ Chebyshev spectral approximation in our analysis of $v$.
We remark that computer assisted proofs for Chebyshev series 
solutions of two point boundary values problems between parameterized
manifolds have been discussed in a number of places including 
\cite{MR3353132,paperBridge,MR4292534}.  Indeed, our implementation makes
use of the domain decomposition techniques discussed in detail in the third reference just
cited, the main difference being that we project onto invariant manifolds 
attached to periodic rather than equilibrium solutions.
  We sketch the main ideas of these computations below.  

\begin{definition}[Chebyshev polynomials]
 Let $T_k:[-1,1] \to \R$ denotes the Chebyshev polynomials. They satisfy the recurrence relation $T_0(t)=1$, $T_1(t)=t$ and
\[
 T_{k+1}(t)=2t T_k(t) - T_{k-1}(t),~ \forall k\geq 1.
\]
An analytic function $f:[-1,1] \to \R$ can be expressed uniquely as
\[
f(t) = a_0 + 2\sum_{k=1}^\infty a_k T_k(t),
\]
where the coefficients decay exponentially as a consequence of Paley-Wiener Theorem. 
\end{definition}
After projecting into the Chebyshev basis, we truncate and solve Equation \eqref{BVP:Chebyshev}
using Newton's method.  Once we have a numerical approximate solution
 we use the a-posteriori techniques developed in 
\cite{MR3392421,MR3207723,LessardReinhardt,MR3353132,paperBridge,RayJB}
to validate the existence of the connecting orbit.  
Since these methods exploit Newton-like operators, we 
obtain transversality in the energy manifold
from the non-degeneracy of the fixed point 
of the Newton-like operator, as in 
\cite{collisionPaper,MR3906230}.
Transversality results for dissipative systems are discussed in 
\cite{MR3207723}.


\begin{definition}[$\ellnu$ norm for sequence of Chebyshev coefficients]
We can see Chebyshev polynomials as a specific case of Fourier expansion, enabling the use of the same coefficient space as for the computation of the periodic orbit. That is, for $a= \left\{ a_k \in \R : k\geq 0 \right\}$, the infinite sequence of Chebyshev coefficients of an analytic function, can be extended into a bi-infinite sequence by setting $a_k= a_{-k}$. So that $a \in \ellnu$ for some $\nu>1$, and the norm can be reduced to
\begin{equation}\label{def:NormCheby}
\left\|a \right\|_{1,\nu} = |a_0| + 2\sum_{k=1}^\infty |a_k|\nu^k.
\end{equation}
A similar simplification allows to use again the convolution product $\star: \ellnu \times \ellnu \to \ellnu$ for all $k\geq 0$ with
\begin{align*}
(a\star b)_k &= \sum_{j\in \mathbb{Z}} a_{j}b_{k-j} \\
&= a_0b_k +\sum_{j=1}^\infty a_{j}(b_{k-j}+b_{k+j}) \\
&= a_0b_k +\sum_{j=1}^\infty a_{j}(b_{|k-j|}+b_{k+j}).
\end{align*}
\end{definition}

After a translation and a rescale of time, the solution of \eqref{BVP:Chebyshev} 
is defined on $[-1,1]$ and therefore can be expressed using a Chebyshev expansion 
for each component. We now rewrite solutions of Equation 
\eqref{BVP:Chebyshev} as a zero of an operator in  Chebyshev 
coefficient space.

%
%
%
%
%
%
%
%
%
\begin{definition}
Let $x \in \R^6 \times (\ellnu)^9$. We write
\[
x= (L,\theta_u,\phi_u,\theta_s,\phi_s,\beta,a^1,\hdots,a^9),
\]
where $L$ is the half-frequency, $\beta$ is the unfolding parameter,
 the pairs $\theta_u,\phi_u$ and $\theta_s,\phi_s$ are evaluation of the 
 unstable and stable parameterization respectively, 
 and $a^i \in \ellnu$ are the coefficients of the Chebyshev expansion
  of the components of the solution.  Set
\begin{align*}
\Lambda_k: \ellnu &\to \mathbb{C}\\
h &\mapsto h_{k+1} - h_{k-1},
\end{align*}
for $k \geq 1$. The operator represents integration of a 
Chebyshev series, and allows us to simplify the 
operator.  Define $H:\mathbb{R} \times (\ellnu)^9 \to (\ell_{\nu'})^9$, 
where $\nu' < \nu$, arising from  projecting 
$\dot u = F(u) +\beta G(u)$ into Chebyshev coefficient space.  Here 
\begin{equation*}
\small
H_k(\beta,a) =\begin{pmatrix}
  a_k^{2}  \\
\beta a_k^2 +2a_k^{4} +a_k^{1} -
\displaystyle\sum_{i=1}^3 m_i\left(a^{1i}\star a^{6+i}\star a^{6+i} \star a^{6+i}\right)_k \\ 
a_k^{4} \\
\beta a_k^4 -2a_k^{2} +a_k^{3} -
\displaystyle\sum_{i=1}^3 m_i\left(a^{3i}\star a^{6+i}\star a^{6+i} \star a^{6+i}\right)_k \\ 
a_k^{6} \\
\beta a_k^6 -\displaystyle \sum_{i=1}^3 m_i\left(a^{5i}\star a^{6+i}\star 
a^{6+i} \star a^{6+i}\right)_k \\  \\
-\left(a^{11}\star a^{2}\star a^{7}\ast a^{7}\star a^{7}\right)_k 
-\left(a^{31}\star a^{4}\star a^{7}\star a^{7}\star a^{7}\right)_k 
-\left(a^{51}\star a^{6}\star a^{7}\star a^{7}\star a^{7}\right)_k \\
-\left(a^{12}\star a^{2}\star a^{8}\ast a^{8}\star a^{8}\right)_k 
-\left(a^{32}\star a^{4}\star a^{8}\star a^{8}\star a^{8}\right)_k
 -\left(a^{52}\star a^{6}\star a^{8}\star a^{8}\star a^{8}\right)_k \\
-\left(a^{13}\star a^{2}\star a^{9}\ast a^{9}\star a^{9}\right)_k
 -\left(a^{33}\star a^{4}\star a^{9}\star a^{9}\star a^{9}\right)_k
  -\left(a^{53}\star a^{6}\star a^{9}\star a^{9}\star a^{9}\right)_k \\
 \end{pmatrix}.
\end{equation*}
A function $u:[0,2L] \to \mathbb{R}^9$ solves $\dot u = F(u) +\beta G(u)$ with initial
 condition $u(0)= Q(\theta_u,\phi_u)$ if the Chebyshev coefficients of the function
\begin{align*}
v: [-1,1] &\to \mathbb{R}^9 \\
	t &\mapsto  u\left( (t+1)L \right)
\end{align*}
are a zero of the operator with components 
\[
F_k^i(x) =
\begin{cases}
 \left(a_0^i +2\displaystyle \sum_{l=1}^\infty (-1)^la_l^i \right) - Q^i(\theta_u,\phi_u), & k=0, \\
 2ka_k^i -L\Lambda_k(H^i(\beta,a)), & k\geq 1.
\end{cases}
\]
The boundary conditions in 
Equation \eqref{BVP:Chebyshev} are satisfied whenever $x$ is a zero of 
\[
\eta^i(x)= \left(a_0^i +2\sum_{l=1}^\infty a_l^i \right) - P^i(\theta_s,\phi_s), 
\quad \forall i=1,2,\hdots, 6.
\]
Gathering both operator, we define 
\begin{align*}
\mathcal{F}: \R^6 \times (\ellnu)^9  &\to \mathbb{R}^6 \times \left( \ell_{\nu'} \right)^9 \\
\mathcal{F}(x) &\mapsto \left( \eta^1(x),\hdots,\eta^6(x),F^1(x),\hdots,F^9(x) \right).
\end{align*}
\end{definition}
%
%
%
%
%
%
%
%
%
%

\begin{definition}[Chebyshev domain decomposition]
For large values of $L$, it is inconvenient to use a single Chebyshev series
to represent the connecting orbit segment, as
the coefficients will delay at a slow rate. 
This can be offset using domain decomposition. Fix $M$ and set
\[
L = \sum_{j=1}^M L_j,
\]
where each $L_j$ is a positive constant. For $j=1,\hdots,M$, define
\begin{align*}
v_j:[-1,1] &\to \mathbb{R}^9 \\
 t &\mapsto u\left( \sum_{m=1}^{j-1}L_m +(t+1)L_j \right),
\end{align*} 
so that the family $\{ v_j : j=1,\hdots,M \}$ represents a piece-wise representaiton of $u$,
the solution of Equation \eqref{BVP:Chebyshev}. 
By construction, each $v_j$ is expressed using a Chebyshev series
whose coefficients $a^j \in (\ellnu)^9$ are a zero, for all $1\leq i \leq 9$, of
\[
F_k^{i,j}(\beta,a^{j}) =
\begin{cases}
 \left(a_0^j +2\displaystyle \sum_{m=1}^\infty (-1)^m a_m^j \right) - C_0^{i,j}, & k=0, \\
 2ka_k^i -L_j\Lambda_k(H^i(\beta,a^j)), & k\geq 1.
\end{cases}
\]
The initial condition $C_0^{i,j}$ provides the initial value of the Chebyshev expansion. 
For $j=1$ and $1\leq i \leq 9$, it is given by $C_0^{i,1} = Q^i(\theta_u,\phi_u)$, while 
the cases of $2\leq j \leq M$ and 
$1\leq i \leq 9$, are given by
\[
C_0^{i,j} = a_0^{i,j-1} +2\displaystyle \sum_{m=1}^\infty a_m^{i,j-1}.
\]
The choice of initial condition arises from the continuity of $u$.
\end{definition}

\begin{remark}[Bounding derivatives of the local invariant manifold parameterizations] \label{rem:derivatives}
Computation of the $Y$ and $Z$ bounds associated with Equation  
\eqref{BVP:Chebyshev} (projected into Chebyshev coefficient space) requires evaluation 
of first and second derivatives of $P = P^N + R$ and $Q = Q^N + S$.  Derivatives of the 
polynomial parts are computed formally, while derivatives of the remainders 
(with respect to both Fourier and Taylor variables)
are bound using the Estimates discussed in Appendix \ref{sec:CauchyBound}.  
\end{remark}

\section{Bounds on the tail of the Fourier-Taylor parameterization}\label{sec:TailArgument}
To simplify notation, we introduce notation related to the Banach space of coefficients 
in which the validation is performed. 
\begin{definition}
Fix $N_t>1$, and let
\begin{align*}
X_2^{N_t} &\bydef \left\{ x \in X_2 : x_{\alpha} = 0, ~\forall |\alpha|>N_t \right\},  \\
\quad \mbox{and}\\
X_2^{\infty} &\bydef \left\{ x \in X_2 : x_{\alpha} = 0,~\forall |\alpha|\leq N_t \right\}. 
\end{align*}
\end{definition}
One verifies that $X_2$ is a direct sum of these subspaces, that they inherit the norm 
on $X_2$, and that $X_2^\infty$ is closed under the Cauchy-Convolution 
product.  $X_2^\infty$ is a Banach space where we formulate the fixed point argument 
for the Fourier-Taylor truncation error argument.

As previously discussed,  $0 \in X_2^\infty$ is a good approximation of the true 
solution $a^\infty$ by construction of the truncation. 
To prove this claim, we show that the disk centered at $0 \in X_2^\infty$ contains a 
unique fixed point of $\mathcal{T}:X_2^\infty \to X_2^\infty$ 
defined as 
\[
\mathcal{T}(x) = \left\{ - \mathcal{G}_\alpha^{-1} 
\mathcal{L}_{\alpha}^{-1} R_\alpha(a^{N_t} + x) : |\alpha| > N_t \right\},
\]
where $a^{N_t} \in X_2^{N_t}$ are the exact Fourier-Taylor coefficients of order up to $N_t$. The operator is well defined thanks to Theorem \ref{Thm:InverseG}, and its norm is
 computed with the help of Equation \eqref{eq:TboundNorm}. %
%
%
%
%
%

\subsubsection{The $Y$ bound}

We seek a constant $Y$ so that $\left\|\mathcal{T}(a^{N_t})\right\|_{X_2^\infty} \leq Y$. 
Note that $a^{N_t}$ has non-zero components up to Taylor order $N_t$,
 and that the CRFBP vector field is fifth degree polynomial so that non-zero entries do not exceed order $5N_t$. Hence 
 $\|T(a^{N_t})\|_{X_2^\infty}$ is a finite sum in the Taylor direction, and is evaluated with the help of a computer. 
 More precisely, we first note that
\begin{equation*}
\left\| \mathcal{T}(a^{N_t}) \right\|_{X_2^\infty} \leq 
\max_{i=1,\hdots,9} \sum_{|\alpha|= N_t+1}^{5N_t} \| 
\mathcal{T}_\alpha^i (a^{N_t}) \|_{1,\nu} \leq  B_G 
\cdot \max_{i=1,\hdots,9} \sum_{|\alpha|= N_t+1}^{5N_t} 
\left\| \mathcal{L}_\alpha^{-1}R_\alpha^i (a^{N_t}) \right\|_{1,\nu}.  
\end{equation*}
Since the non zero terms are given by Cauchy-Convolution products, 
we compute positive constants $b_{\alpha,k}^2,b_{\alpha,k}^4,b_{\alpha,k}^6$ having
\begin{align*}
\left|\frac{R_{\alpha,k}^2(a^{N_t})}{\im \omega k -\langle 
\alpha,\lambda \rangle}\right| &\leq \sum_{i=1}^3 m_i\left| \frac{\left(a^{1i} \hat 
\ast a^{6+i} \hat \ast a^{6+i} \hat \ast x^{6+i}\right)_{\alpha,k}}{\im \omega k 
-\langle \alpha,\lambda \rangle}\right| \leq b_{\alpha,k}^2,\\
\left|\frac{R_{\alpha,k}^4(a^{N_t})}{\im \omega k -\langle 
\alpha,\lambda \rangle}\right| &\leq \sum_{i=1}^3 m_i\left| \frac{\left(a^{3i} \hat
 \ast a^{6+i} \hat \ast a^{6+i} \hat \ast x^{6+i}\right)_{\alpha,k}}{\im \omega k 
 -\langle \alpha,\lambda \rangle}\right| \leq b_{\alpha,k}^4,\\
\left|\frac{R_{\alpha,k}^6(a^{N_t})}{\im \omega k -\langle \alpha,\lambda \rangle}\right| &\leq \sum_{i=1}^3 m_i\left| \frac{\left(a^{5i} \hat \ast a^{6+i} \hat \ast a^{6+i} \hat \ast x^{6+i}\right)_{\alpha,k}}{\im \omega k -\langle \alpha,\lambda \rangle}\right| \leq b_{\alpha,k}^6.\\
\end{align*}
The last three components 
$b_{\alpha,k}^7,b_{\alpha,k}^8,b_{\alpha,k}^9$ are defined similarly, 
so that we discuss only the first case.  That is 
\begin{align*}
 \left| \frac{\left(a^{11}\hat\ast a^2 \hat \ast a^7 \hat\ast 
 a^7 \hat \ast a^7\right)_{\alpha,k}}{\im \omega k -\langle \alpha,\lambda \rangle}\right| 
+\left| \frac{\left(a^{31}\hat\ast a^4 \hat \ast a^7 \hat\ast a^7 \hat
 \ast a^7\right)_{\alpha,k}}{\im \omega k -\langle \alpha,\lambda \rangle}\right| 
+\left| \frac{\left(a^{51}\hat\ast a^6 \hat \ast a^7 \hat\ast a^7 \hat \ast a^7\right)_{\alpha,k}}{\im \omega k -\langle \alpha,\lambda \rangle}\right| \leq b_{\alpha,k}^7. \\
\end{align*}
The remainder is zero in all components where the original vector field is linear. So, 
it suffices to set $b_{\alpha,k}^1 = b_{\alpha,k}^3 = b_{\alpha,k}^5 =0$ for all $k$ and $\alpha$. 
Using this notation, we obtain the bound 
\begin{align*}
Y \bydef B_G \cdot \max_{1\leq i \leq 9} \displaystyle\sum_{|\alpha|=N_t+1}^{5N_t} \|b_{\alpha}^i\|_{1,\nu}.
\end{align*}
Each $b^i$ is computed via the approach discussed in
Section \ref{Sec:RewritingTail}, which allows us to bound the convolution products.
We illustrate the procedure for one example term. 
The other cases are similar.

\begin{example}[Bound on convolution product]
We separate each component of $a^{N_t}$ into a piece containing all Fourier 
coefficients of order lower than $K$, and one piece containing the remaining ones. 
That is, for $1\leq i \leq 9$ define $x^i \in X_2^{N_t}$
\begin{align*}
x_{\alpha,k}^i = \begin{cases}
(a_{\alpha,k}^i)^{N_t}, &\mbox{if}~ |k|<K,~ 0<|\alpha|\leq N_t, \\
0, & \mbox{otherwise},
\end{cases}
\end{align*} 
and $R^i \in X_2^{N_t}$ such that $\left(a_{\alpha}^{N_t}\right)^i = x_{\alpha}^i 
+ R_{\alpha}^i$ for all $0<|\alpha|\leq N_t$. It follows that
\[
\left\| R^i \right\|_{X_d} \leq \sum_{|\alpha|=1}^{N_t} r_\alpha \bydef E^{N_t},
\]
where each $r_\alpha$ is given by the successive applications of Theorem 
\ref{thm:ContractionMapping}  described in Section \ref{sec:lowOrderTerms}. 
The convolution term $\left(a^{11} \hat \ast a^{7} \hat \ast a^{7} \hat \ast a^{7}\right)_{\alpha,k}$ 
satisfies 
\begin{align}
\left(a^{11} \hat \ast a^{7} \hat \ast a^{7} \hat \ast a^{7}\right)_{\alpha,k} &= 
\left(x^{11} \hat \ast x^{7} \hat \ast x^{7} \hat \ast x^{7}\right)_{\alpha,k}
 + \left(R^{11} \hat \ast x^{7} \hat \ast x^{7} \hat \ast x^{7}\right)_{\alpha,k} 
 +3\left(x^{11} \hat \ast x^{7} \hat \ast x^{7} \hat \ast R^{7}\right)_{\alpha,k} \label{eq:line1} \\
&+3\left(R^{11} \hat \ast R^{7} \hat \ast x^{7} \hat \ast x^{7}\right)_{\alpha,k} 
+3 \left(x^{11} \hat \ast x^{7} \hat \ast R^{7} \hat \ast R^{7}\right)_{\alpha,k} 
+\left(x^{11} \hat \ast R^{7} \hat \ast R^{7} \hat \ast R^{7}\right)_{\alpha,k} \label{eq:line2}\\
&+3\left(R^{11} \hat \ast R^{7} \hat \ast R^{7} \hat \ast x^{7}\right)_{\alpha,k} 
+ \left(R^{11} \hat \ast R^{7} \hat \ast R^{7} \hat \ast R^{7}\right)_{\alpha,k} \label{eq:line3}
\end{align}

The first term, at line \eqref{eq:line1}, is computed explicitly using interval arithmetic as each $x^i$
 has only finitely many non zero coefficients. The remaining terms are bound using
 Proposition \ref{EstimatesBound}. The computation uses the dual bound estimates of 
\ref{EstimatesBound} for terms of order $1$ in $R$ in equation \eqref{eq:line1}. 
The other cases, in equations \eqref{eq:line2} and \eqref{eq:line3}, are expected to be small and can be bounded
using the second norm estimates \ref{EstimatesBound}. The resulting bound defines one term of 
the desired  $b_{\alpha,k}^{2}$ term.
\end{example}

This example completes the presentation for the computation of the $Y$ bound. 
It follows from this definition that $\|T(a^N)\|_{X_2^\infty} \leq Y$, as desired. 
%
%
%
%
%
%
%
%
%
\subsubsection{The $Z$ bound}

The next matter of interest is to obtain the polynomial bound $Z(r)$. 
Note that it is enough to compute an upper bound on
\[
 \sup_{b,c \in \overline{B_r(0)}}\| D\mathcal{T}(a^{N_t} +b)c\|_{X_2^\infty},
\]
and note that --from linearity of the operators involved in the definition of $\mathcal{T}$ --
we have that
\begin{align}\label{eq:Zdefinition}
\left( D\mathcal{T}(a^{N_t} +b)c \right)_\alpha = \mathcal{G}_\alpha^{-1} \mathcal{L}_\alpha^{-1} \left( DR(a^{N_t} +b)c\right)_\alpha,~\forall~ |\alpha|>N_t.
\end{align}
After the substitution $b= ur$, $c= vr$ 
where $u,v \in X_2^\infty$ are both elements of norm one, equation \eqref{eq:Zdefinition} becomes a degree five polynomial in $r$. The coefficients for all terms 
of degree one are treated with more caution, while the terms of higher degree are bounded 
using the Banach algebra. Thus, we focus on bounding the linear terms in $r$. 

Every such terms is a convolution product of the form
\[
(a^{N_t})^{i_1}\hat\ast(a^{N_t})^{i_2}\hat\ast(a^{N_t})^{i_3}\hat\ast v^{i_4} \quad \mbox{or} 
\quad (a^{N_t})^{i_1}\hat\ast(a^{N_t})^{i_2}\hat\ast(a^{N_t})^{i_3}\hat\ast (a^{N_t})^{i_4} \hat\ast v^{i_5}, 
\]
and each convolution is handled separately. We show explicitly the computation of the bound 
in the case of a fifth degree convolution, and note that the fourth degree case is similar.  Set
\begin{equation}\label{eq:aquadracticDef}
c^{ij} = (a^{N_t})^{i_1}\hat\ast(a^{N_t})^{i_2}\hat\ast(a^{N_t})^{i_3}\hat\ast (a^{N_t})^{i_4},
\end{equation}
where the index $i,j$ denotes that the convolution product arise as a term from 
$D_{a_j}R^i(a^{N_t}+b)c$. For any pair $i,j$, the convolution product $c^{ij}$ has only finitely 
many non-zero coefficients in the Taylor direction.  Again, this is a consequence of the fact 
that $a^{N_t} \in X_2^{N_t}$. This lets us reduce the 
expansion of each norm computation, for example 
\begin{align*}
\left\|\sum_{|\alpha|>N_t}\mathcal{L}_\alpha^{-1}(c^{ij}\hat\ast v^j) \right\|_{X_d^\infty} &= \sum_{|\alpha|>N_t}
\sum_{k\in\mathbb{Z}} \left| 
\frac{(c^{ij}\hat\ast v^{j})_{\alpha,k}}{\im \omega k +\langle \alpha,\lambda \rangle} \right|\nu^{|k|} \\
&\leq \sum_{|\alpha|>N_t} \sum_{|k|<K} \left| d_{\alpha,k} (c^{ij}\hat\ast v^{j})_{\alpha,k}\right|\nu^{|k|}
+ \sum_{|\alpha|>N_t} \sum_{|k|\geq K} \left| 
d_\alpha^\infty (c^{ij}\hat\ast v^{j})_{\alpha,k}\right|\nu^{|k|} \\
&\leq \sum_{|\alpha|>N_t} \sum_{|k|<K} \left| d_{\alpha,k}(c^{ij}\hat\ast v^{j})_{\alpha,k}\right|\nu^{|k|}
+ \left\|c^{ij}\right\|_{X_d} \cdot \sup_{|\alpha|>N_t} d_\alpha^\infty  \\
\end{align*}
Where $d_{\alpha,k}$ and $d_\alpha^\infty$ are as defined in \eqref{eq:Ltermwise}. For the first sum, representing the finite part in Fourier direction, we set
\[
d_k = \sup_{|\alpha|>N_t}\left| d_{\alpha,k} \right|.
\]
Note that the supremum exists for the same reason as in the case of equation \eqref{eq:NormLinverse}. Then 
\begin{align*}
\sum_{|\alpha|>N_t} \sum_{|k|<K} \left| 
d_{\alpha,k}(c^{ij}\hat\ast v^{j})_{\alpha,k}\right|\nu^{|k|}
&= \sum_{|\alpha|>N_t} \sum_{|k|<K} \left|  
d_{\alpha,k}
\sum_{\substack{ 1\leq |\beta| \leq 4N_t \\ \alpha-\beta \in \mathbb{Z}_d^+ }} 
\left( c_\beta^{ij}\star v_{\alpha-\beta}^{j}\right)_k\right|\nu^{|k|} \\
&\leq \sum_{|k|<K} d_k \sum_{|\alpha|>N_t} \left| 
\sum_{\substack{ 1\leq |\beta| \leq 4N_t \\ \alpha-\beta \in \mathbb{Z}_d^+ }}
 \left( c_{\beta}^{ij}\star v_{\alpha-\beta}^{j}\right)_k\right|\nu^{|k|} \\
&\leq \sum_{|k|<K} d_k \sum_{|\alpha|>N_t} 
\sum_{\substack{ 1\leq |\beta| \leq 4N_t \\ \alpha-\beta \in \mathbb{Z}_d^+ }} 
\left| \left( c_{\beta}^{ij}\star v_{\alpha-\beta}^{j}\right)_k\right|\nu^{|k|} \\
&\leq \sum_{|k|<K} \sum_{ 1\leq |\beta| \leq 4N_t } \sup_{\substack{m\in \mathbb{Z} \\ 
 1\leq |b| \leq 4N_t  }} \frac{d_k c_{b,k-m}^{ij}}{\nu^{|m|}}
  \sum_{|\alpha|>N_t}\sum_{m \in \mathbb{Z}}  \left| 
  \left(v_{\alpha-\beta,m}^{j}\right)_k\right|\nu^{|m|}\nu^{|k|} \\
&\leq \sum_{|k|<K} \sum_{ 1\leq |\beta| \leq 4N_t } d_k
 \left(\sup_{\substack{m\in \mathbb{Z} \\  
 1\leq |b| \leq 4N_t  }} \frac{c_{b,k-m}^{ij}}{\nu^{|m|}}\right) \nu^{|k|} \\
&\leq  (4N_t(4N_t+1))\sum_{|k|<K} d_k
\left(\sup_{\substack{m\in \mathbb{Z} \\  
1\leq |b|\leq 4N_t }} \frac{c_{b,k-m}^{ij}}{\nu^{|m|}} \right) \nu^{|k|} \bydef Z_K^{ij}. \\
\end{align*}
Here the last two inequalities follow from the application of Proposition \eqref{EstimatesBound},
the fact that $\left\| v \right\|_{X_d^\infty}= 1$, and the fact that the norm is unchanged 
from shifting the Taylor index. Finally, the last expression is independent of
$\beta$ and the sum is evaluated directly to obtain $4N_t(4N_t+1)$. Combining this estimate with the previous computation gives that
\[
\left\| \sum_{|\alpha|>N_t} \mathcal{L}_\alpha^{-1}\left( c^{ij} \hat\ast v^{j} \right) \right\|_{X_d^\infty} \leq Z_K^{ij} + \|c^{ij}\|_{X_d}
 \cdot \sup_{|\alpha|>N_t} d_\alpha^\infty \bydef Z_1^{ij}.
\]

This bound is presented with sufficient generality to treat all terms of order one in $r$, using the triangle inequality for cases with several distinct convolutions. 
It follows that 
\begin{equation}
Z_1 = B_G \cdot \max \begin{pmatrix}
\displaystyle\sum_{j=1}^3 m_j\bigg(Z_1^{21} +3Z_1^{2j}\bigg) \\
\displaystyle\sum_{j=1}^3 m_j\bigg(Z_1^{43} +3Z_1^{4j}\bigg) \\
\displaystyle\sum_{j=1}^3 m_j\bigg(Z_1^{65} +3Z_1^{6j}\bigg) \\
3Z_1^{77} +\displaystyle\sum_{j=1}^6 Z_1^{7j} \\
3Z_1^{88} +\displaystyle\sum_{j=1}^6 Z_1^{8j} \\
3Z_1^{99} +\displaystyle\sum_{j=1}^6 Z_1^{9j} \\
\end{pmatrix},
\end{equation}
so that
\[
 \sup_{b,c \in \overline{B_r(0)}}\| DT(x^N +b)c\|_{X^\infty} \leq Z_1 r + O(r^2).
\]
Again, terms of higher order do not need small coefficients for the 
argument to succeed.  Their explicit calculation is omitted in the present work, 
and we refer the interested reader to \cite{Ransford} for an example of the 
development of the bounds in the case of the CRTBP. This completes the 
computation of the polynomial required to verify Theorem \ref{thm:ContractionMappingFixedPt}, 
and allows us to construct the Radii polynomial associated to the estimates.
If the polynomial is negative for some positive radius, it follows that there exists 
an $0<r^\infty<\infty$ such that
 \[
\left\|\sum_{|\alpha|=N_t+1}^\infty  A_\alpha(t) \sigma^\alpha \right\|_\infty \leq r^\infty.
\]
Recall that this term arise from regrouping the terms not yet bounded in \eqref{Error:finite}. 
Hence, it is now possible to bound the truncation error associated to $\bar P$. That is, we have 
\begin{align}
\|P - \bar P\|_\infty & \leq \sum_{|\alpha|=0}^{N_t} r_\alpha  + 
\left\|\sum_{|\alpha|=N_t+1}^\infty  A_\alpha(t) \sigma^\alpha \right\|_\infty \nonumber \\
&\leq \sum_{|\alpha|=0}^{N_t} r_\alpha  + r^\infty \bydef r_{\bar P}. \label{Error:Full}
\end{align}
Note that the error bound depends on the approximation itself, and we express the 
dependency using the subscript $\bar P$. This specification is crucial when a pair of 
distinct validated manifolds are computed and used to define the two point boundary 
value problem of interest to compute cycle-to-cycle connecting orbits.

\section{Applications to the spatial CRFBP}\label{Sec:Results}

In this section we describe computer assisted existence proofs for 
a number of transverse homoclinic connections to spatial periodic orbits
in the CRFBP.  More precisely, we present the results of twelve 
validations: six of them involving the 
vertical Lyapunov family at $\mathcal{L}_5$ in the 
Triple Copenhagen Problem (equal masses $m_1=m_2=m_3 = 1/3$),
and six more validations for the CRFBP with non-equal masses 
 $m_1= 0.4$, $m_2= 0.33$, and $m_3= 0.27$.  We remark that in the 
 Triple Copenhagen problem, each proof of a connecting orbit at 
 $L_0$ leads to the existence of two connections more at $L_0$, 
 and each proof of a connecting orbit at $L_5$ leads to one at $L_4$ and 
 $L_6$ --  both facts by $120$ degree rotational symmetry.
 Then our results actually imply the existence of 18 distinct connections
 in the Triple Copenhagen problem.  
For the case of non-equal masses on the other hand, every connection 
must be proven singly.  We note also that, because of the transversality, 
each of the 24 total orbits proven here 
implies the existence of chaotic dynamics nearby, and 
of homoclinic orbits for nearby energies, 
though we do not obtain bounds on the range of energies here.  

The next two Theorems summarize our results, which 
constitute the main results of the paper.  We remark however that, using the 
technology developed here, many other similar results could be proven.

%

\begin{figure}[!t]
	\subfigure{{\includegraphics[width=.48\textwidth]{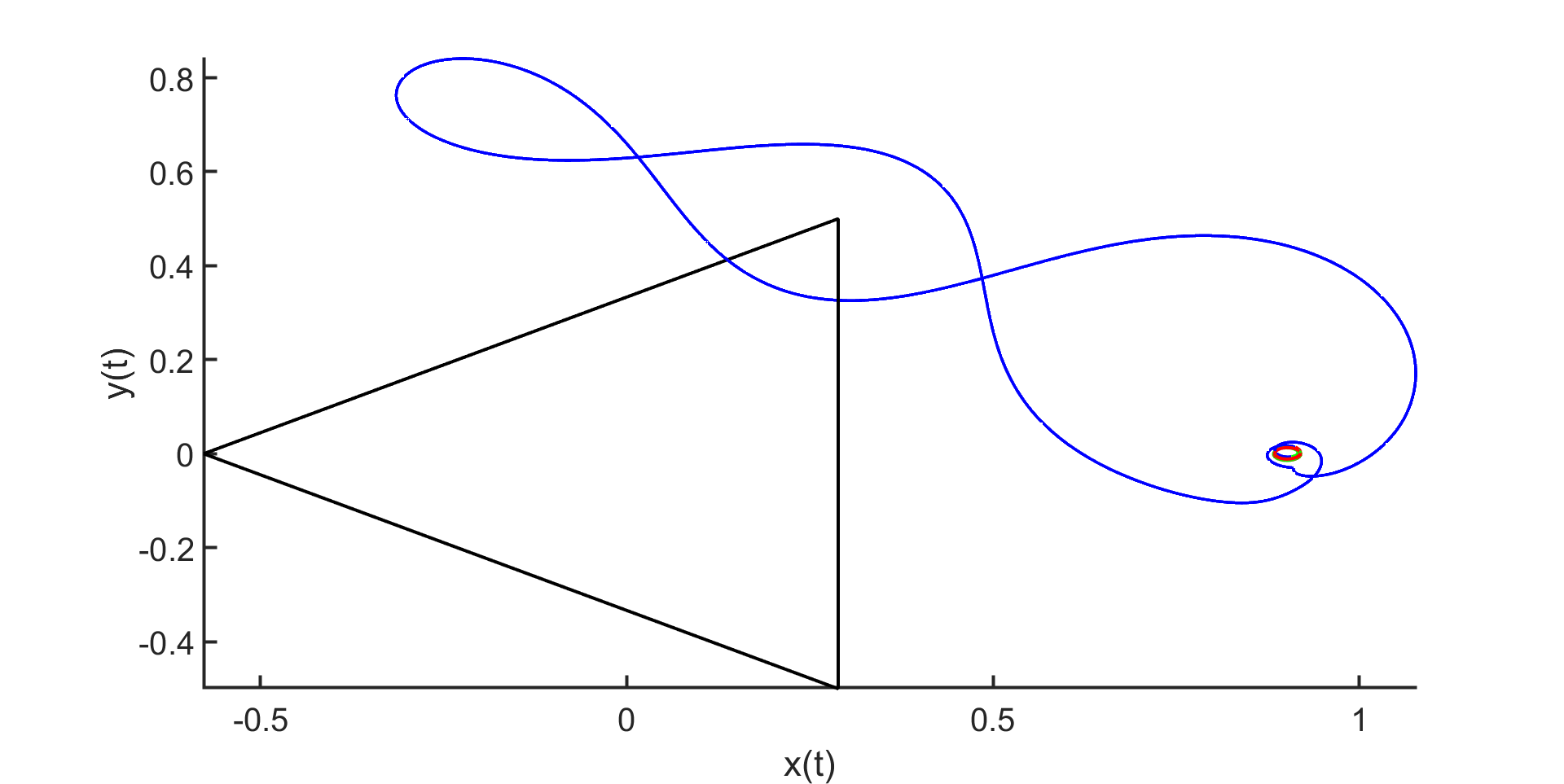}}} 
	\subfigure{{\includegraphics[width=.48\textwidth]{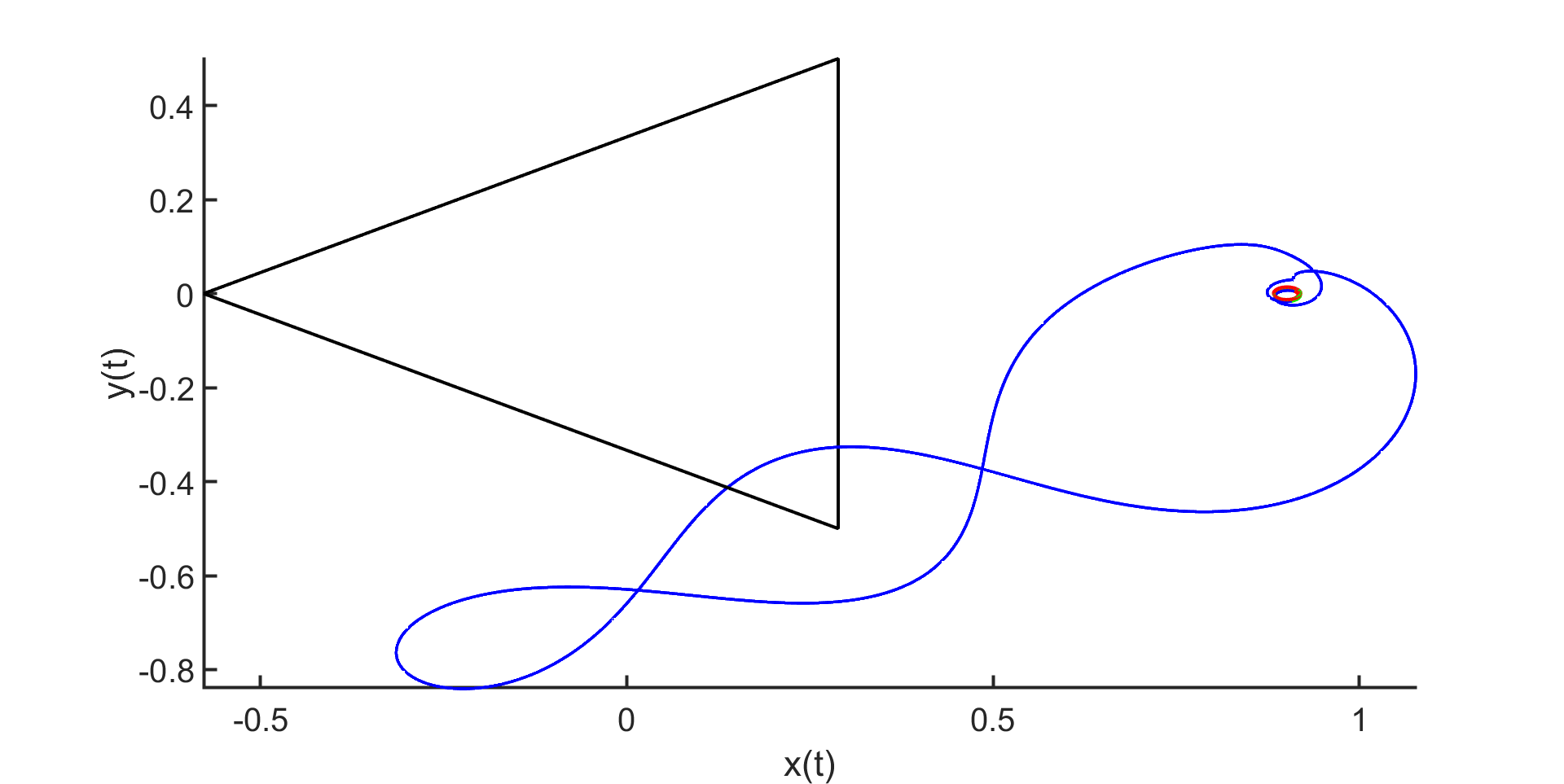}}} \\    
    \subfigure{{\includegraphics[width=.48\textwidth]{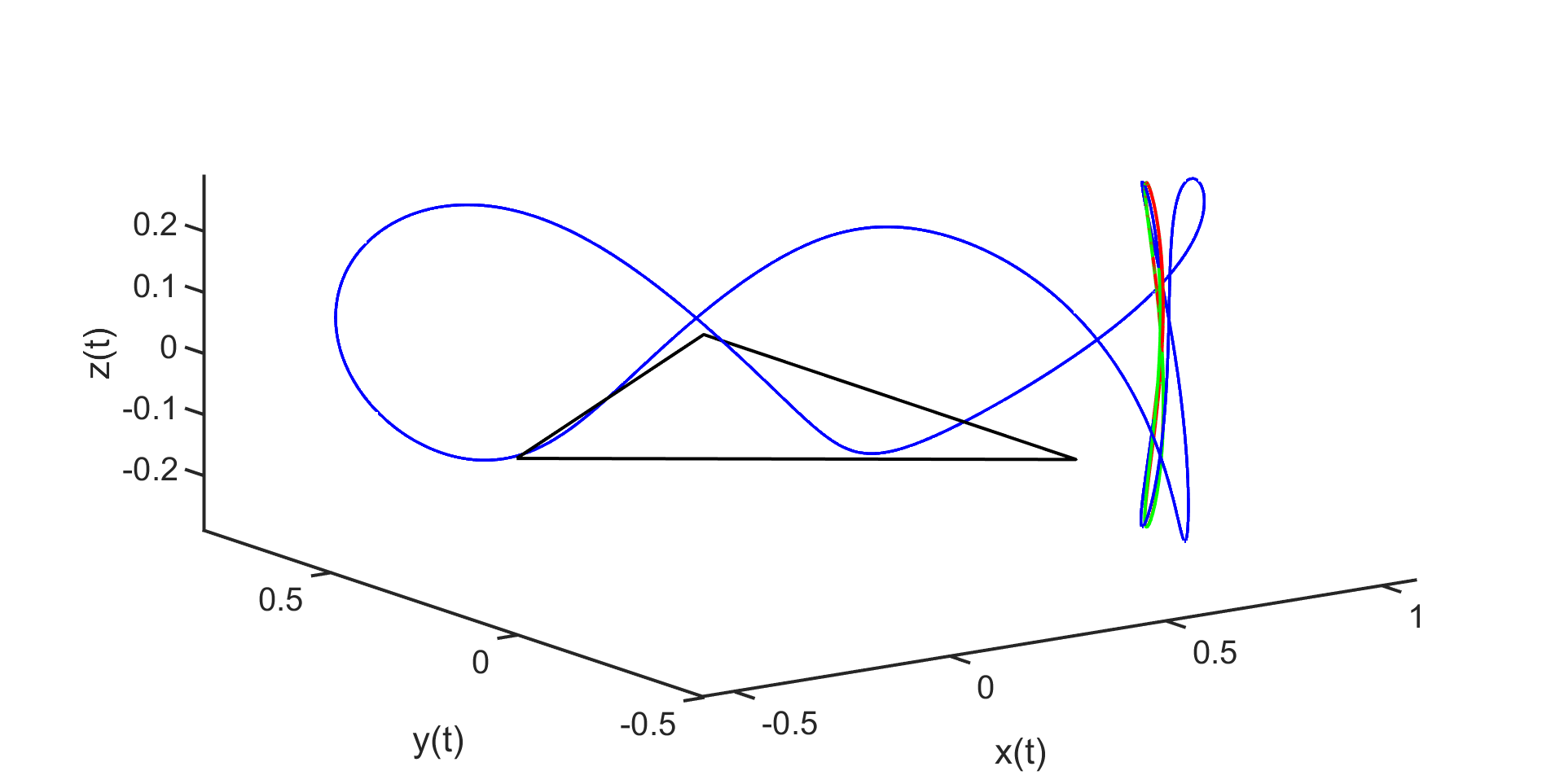}}} 
    \subfigure{{\includegraphics[width=.48\textwidth]{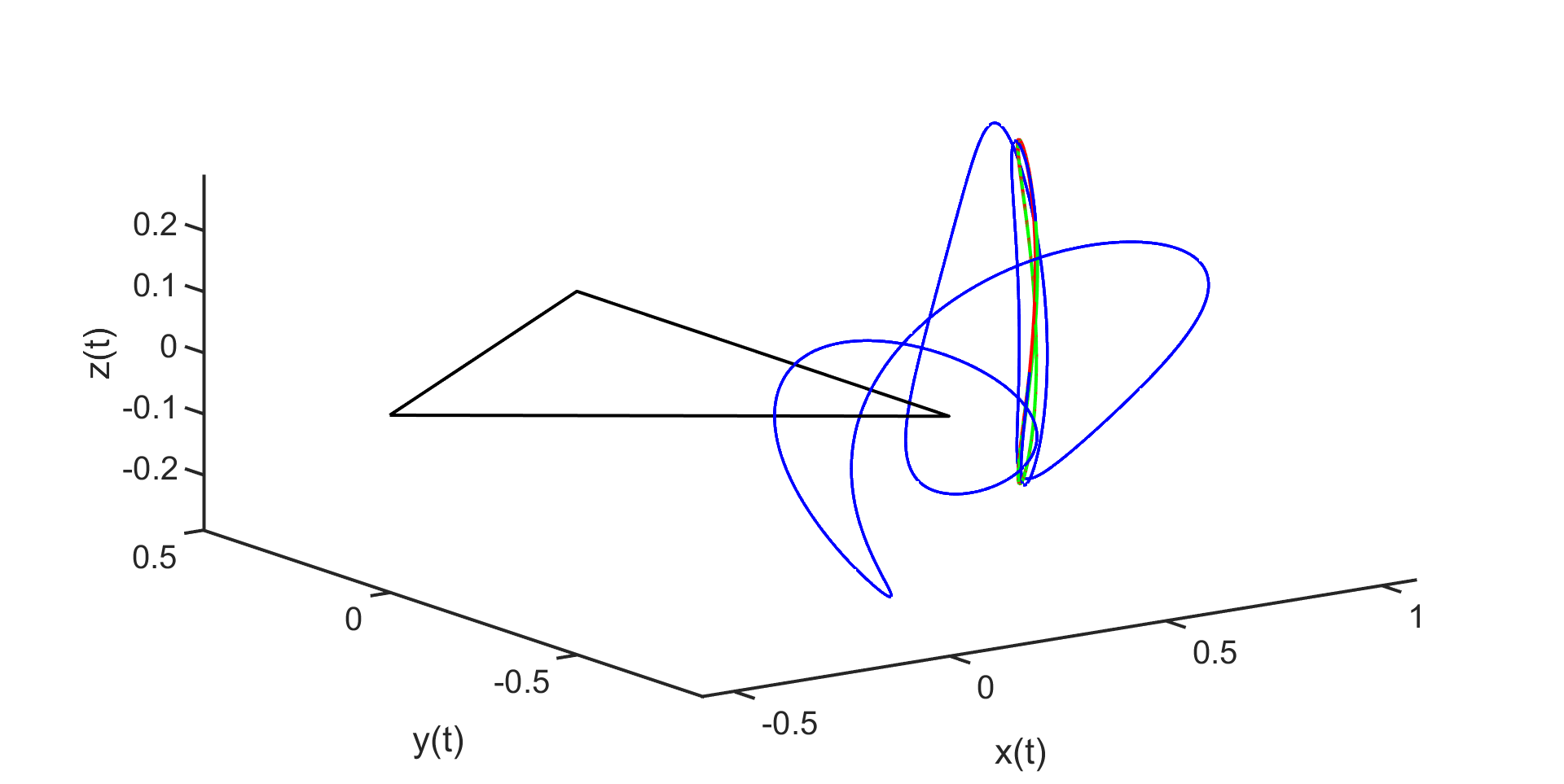}}}\\    
\caption{Two homoclinic connections with non-trivial winding about a primary, 
for a spatial periodic orbit in the vertical Lyapunov 
family at $\mathcal{L}_5$ in the triple Copenhagen problem. The top row displays the 
projection in the $x-y$ plane, while the bottom row illustrates the 
$x-y-z$ projection. Notice from the second angle that both connections are 
indeed out of plane. 
Each of the connections is validated using Theorem \ref{thm:ContractionMapping}, 
so that a true solution lays in a neighborhood of the approximation displayed. The size 
of the neighborhood in the $C^0$ norm is specified for each case in 
Theorem \ref{thm:ValidationL5}.  The Jacobi constant for these cycle-to-cycle
connections is approximately $2.84$.        
}\label{fig:homoclinicL5A-B}
\end{figure}

\begin{figure}[!t]

 	\subfigure{{\includegraphics[width=.48\textwidth]{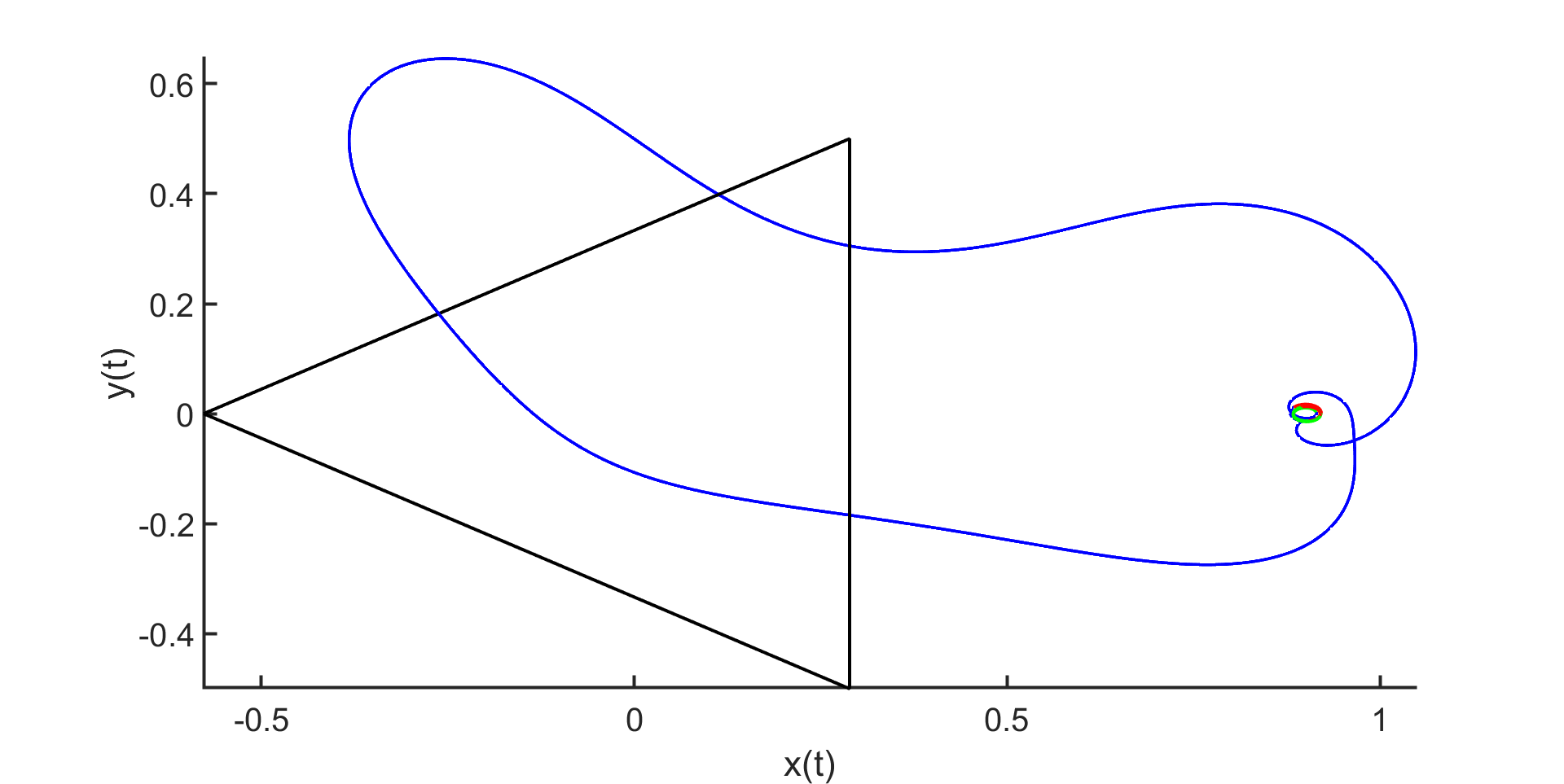}}}    
    \subfigure{{\includegraphics[width=.48\textwidth]{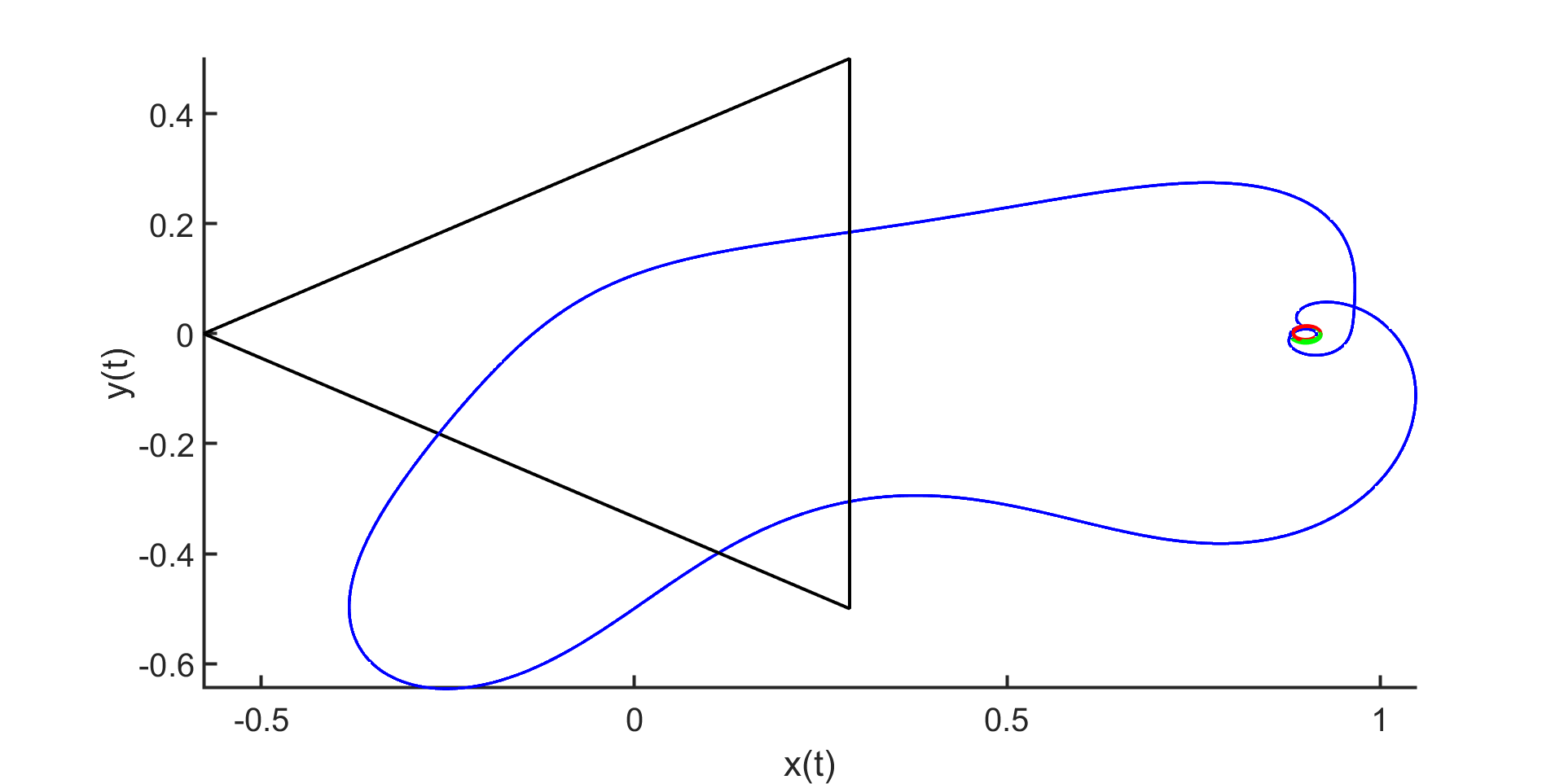}}} \\ 
    \subfigure{{\includegraphics[width=.48\textwidth]{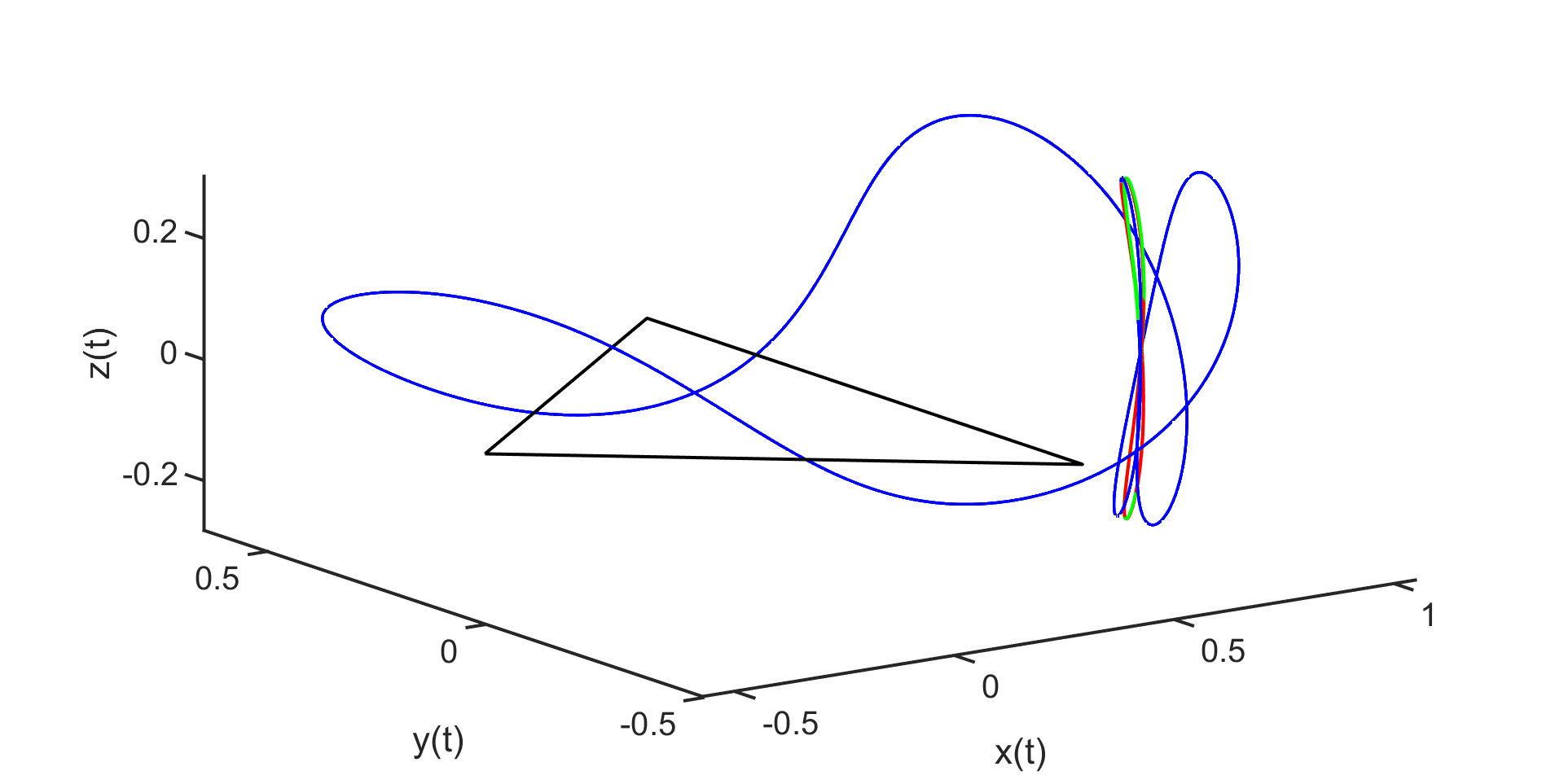}}}     
    \subfigure{{\includegraphics[width=.48\textwidth]{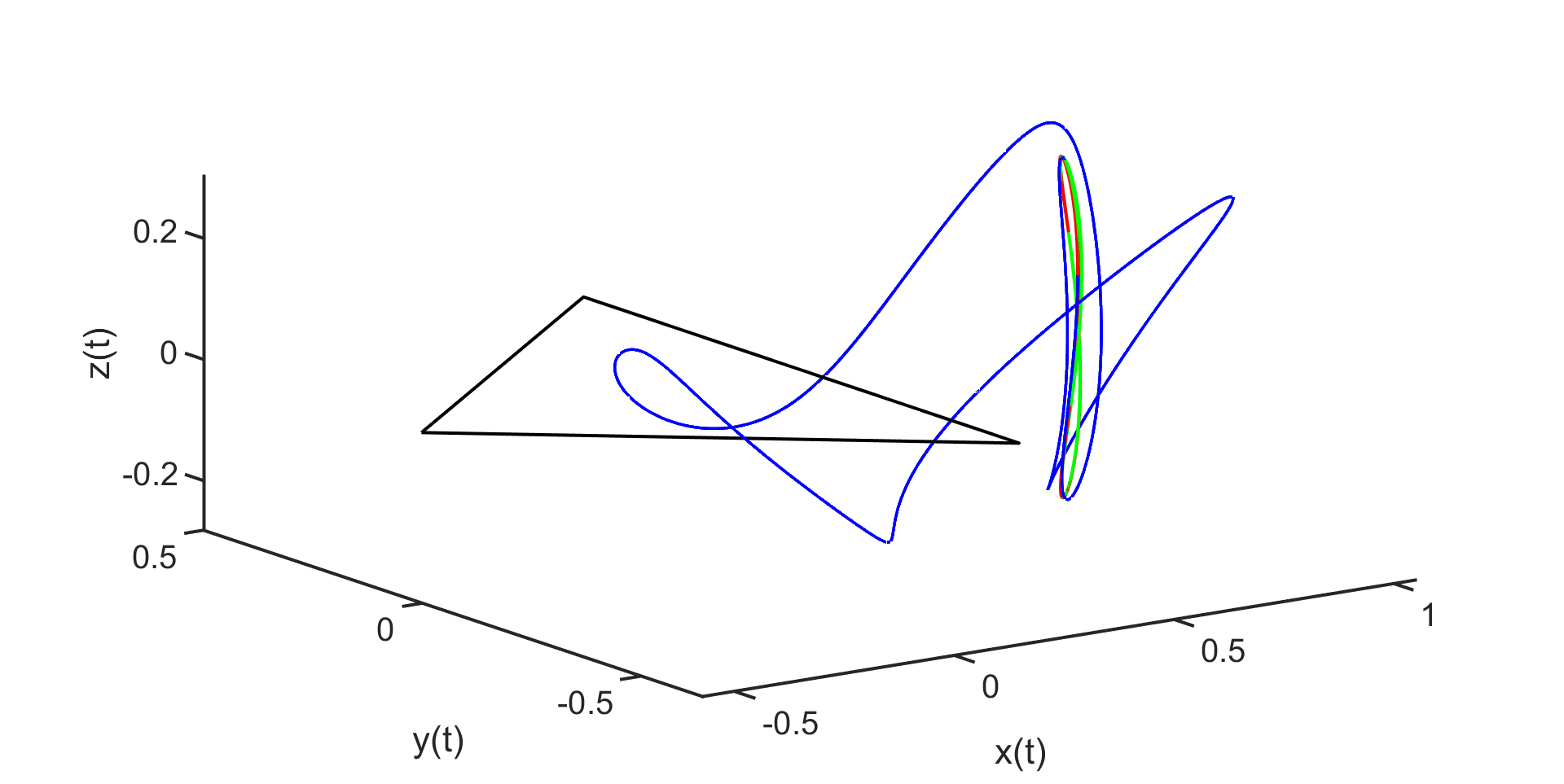}}} \\
\caption{ Homoclinic connections with trivial winding about all three primary bodies, 
for a spatial periodic orbit in the vertical Lyapunov family 
at $\mathcal{L}_5$ in the triple Copenhagen problem. 
Note that the connections pass through the ``inner system'', i.e. through the 
equilateral triangle formed by the primaries.
Error bounds are as given in Theorem \ref{thm:ValidationL5}, and 
projections displayed and energies are as described in the caption of 
Figure \ref{fig:homoclinicL5A-B}       
}\label{fig:homoclinicL5C-D}
\end{figure}

\begin{theorem}[Transverse spatial homoclinics for a Lyapunov orbit at $\mathcal{L}_5$
in the Triple Copenhagen problem] \label{thm:ValidationL5}
There exists a spatial periodic orbit $\gamma$ in the
Triple Copenhagen problem with 2D stable and 
2D unstable attached local invariant manifolds.  The Jacobi constant
for $\gamma$ is approximately $2.84$.  Moreover, there exist
six distinct transverse homoclinic connections for $\gamma$ in the 
same energy level.  
The connections are illustrated in Figures \ref{fig:homoclinicL5A-B},
\ref{fig:homoclinicL5C-D} and \ref{fig:homoclinicL5E-F}.
\end{theorem} 

The computation are performed with $K=30$ Fourier modes, and $N_t=6$ 
Taylor modes for the parameterized manifolds. The scale of the eigenfunctions are
controlled by taking phase conditions as in Equation \eqref{eq:PhaseBundle} 
with $k_0=10$ and $\xi_0 = 10^{-5}$. The parameterized local stable manifold 
$P$, and local unstable manifold $Q$ satisfy
\begin{align*}
\left\| P - \bar P \right\|_{\infty} &\leq 1.076 \cdot 10^{-12}, ~\mbox{and} \\
\left\| Q - \bar Q \right\|_{\infty} &\leq 1.597 \cdot 10^{-12}.
\end{align*}
These parameterizations are used to set up BVPs \eqref{BVP:Chebyshev}
for all six connecting orbits.  We numerically compute approximate 
solutions $\bar{v}_a,\bar{v}_b,\bar{v}_c,\bar{v}_d,\bar{v}_e,\bar{v}_f$ displayed in 
 Figures \ref{fig:homoclinicL5A-B},
\ref{fig:homoclinicL5C-D} and \ref{fig:homoclinicL5E-F}. 
Applying Theorem \ref{thm:ContractionMapping}, we obtain 
the existence of true solutions having 
\begin{align*}
\left\| v_a - \bar{v}_a \right\|_{\infty} &\leq 5.155 \cdot 10^{-9}, \\
\left\| v_b - \bar{v}_b \right\|_{\infty} &\leq 5.108 \cdot 10^{-9}, \\
\left\| v_c - \bar{v}_c \right\|_{\infty} &\leq 3.074 \cdot 10^{-9}, \\
\left\| v_d - \bar{v}_d \right\|_{\infty} &\leq 2.636 \cdot 10^{-9}, \\
\left\| v_e - \bar{v}_e \right\|_{\infty} &\leq 1.626 \cdot 10^{-8}, ~\mbox{and} \\
\left\| v_f - \bar{v}_f \right\|_{\infty} &\leq 1.936 \cdot 10^{-8},
\end{align*}
in the $C^0$ norm.  
The proofs for $v_a$ and $v_b$ are formulated with
$M=9$ Chebyshev domains, each with $50$ 
non-zero coefficients. The remaining proofs are formulated
with $M=8$ domains and $45$ Chebyshev coefficients.

\begin{theorem}[Error bound of solutions at $\mathcal{L}_0$]\label{thm:ValidationL0}
There exists a spatial periodic orbit $\gamma$ in the
CRFBP with  $m_1=0.4$, $m_2=0.33$, and $m_3= 0.27$.  The cycle $\gamma$
has  2D stable and 2D unstable attached local invariant manifolds.  
The Jacobi constant
for $\gamma$ is approximately $3.26$.  Moreover, there exist
six distinct transverse homoclinic connections for $\gamma$.   
The connections are illustrated in Figure \ref{fig:homoclinicL0}.
\end{theorem} 

These computation where performed with 
$K=30$ Fourier and $N_t=6$ Taylor modes for the parameterization of each manifold,
and scalings set (as in Equation \eqref{eq:PhaseBundle}) 
with $k_0=10$ and $\xi_0 = 10^{-4}$. The parameterized stable manifold $P$, 
and unstable manifold $Q$ satisfy
\begin{align*}
\left\| P - \bar P \right\|_{\infty} &\leq 2.154 \cdot 10^{-12}, ~\mbox{and} \\
\left\| Q - \bar Q \right\|_{\infty} &\leq 1.679 \cdot 10^{-12}.
\end{align*}
These parameterization are used to set up BVPs \eqref{BVP:Chebyshev}
for all six connections.  We use Netwon's method to compute
approximated solutions $\bar{v}_a,\bar{v}_b,\bar{v}_c,\bar{v}_d,\bar{v}_e,\bar{v}_f$ 
displayed in Figure \ref{fig:homoclinicL0}. Again, using 
Theorem \ref{thm:ContractionMapping} we obtain that
\begin{align*}
\left\| v_a - \bar{v}_a \right\|_{\infty} &\leq 1.999 \cdot 10^{-9}, \\
\left\| v_b - \bar{v}_b \right\|_{\infty} &\leq 2.504 \cdot 10^{-9}, \\
\left\| v_c - \bar{v}_c \right\|_{\infty} &\leq 1.885 \cdot 10^{-9}, \\
\left\| v_d - \bar{v}_d \right\|_{\infty} &\leq 4.763 \cdot 10^{-9}, \\
\left\| v_e - \bar{v}_e \right\|_{\infty} &\leq 3.027 \cdot 10^{-9}, ~\mbox{and} \\
\left\| v_f - \bar{v}_f \right\|_{\infty} &\leq 3.591 \cdot 10^{-9}.
\end{align*}
The short connections (displayed on the left of Figure \ref{fig:homoclinicL0}) are 
computed with $M=4$ Chebyshev domains, each with $40$ non-zero coefficients. 
The longer connections (displayed on the right of Figure \ref{fig:homoclinicL0}) are
computed with $M=8$ Chebyshev domains, each with $50$ non-zero coefficients.

\begin{figure}[!t]
	\subfigure{{\includegraphics[width=.48\textwidth]{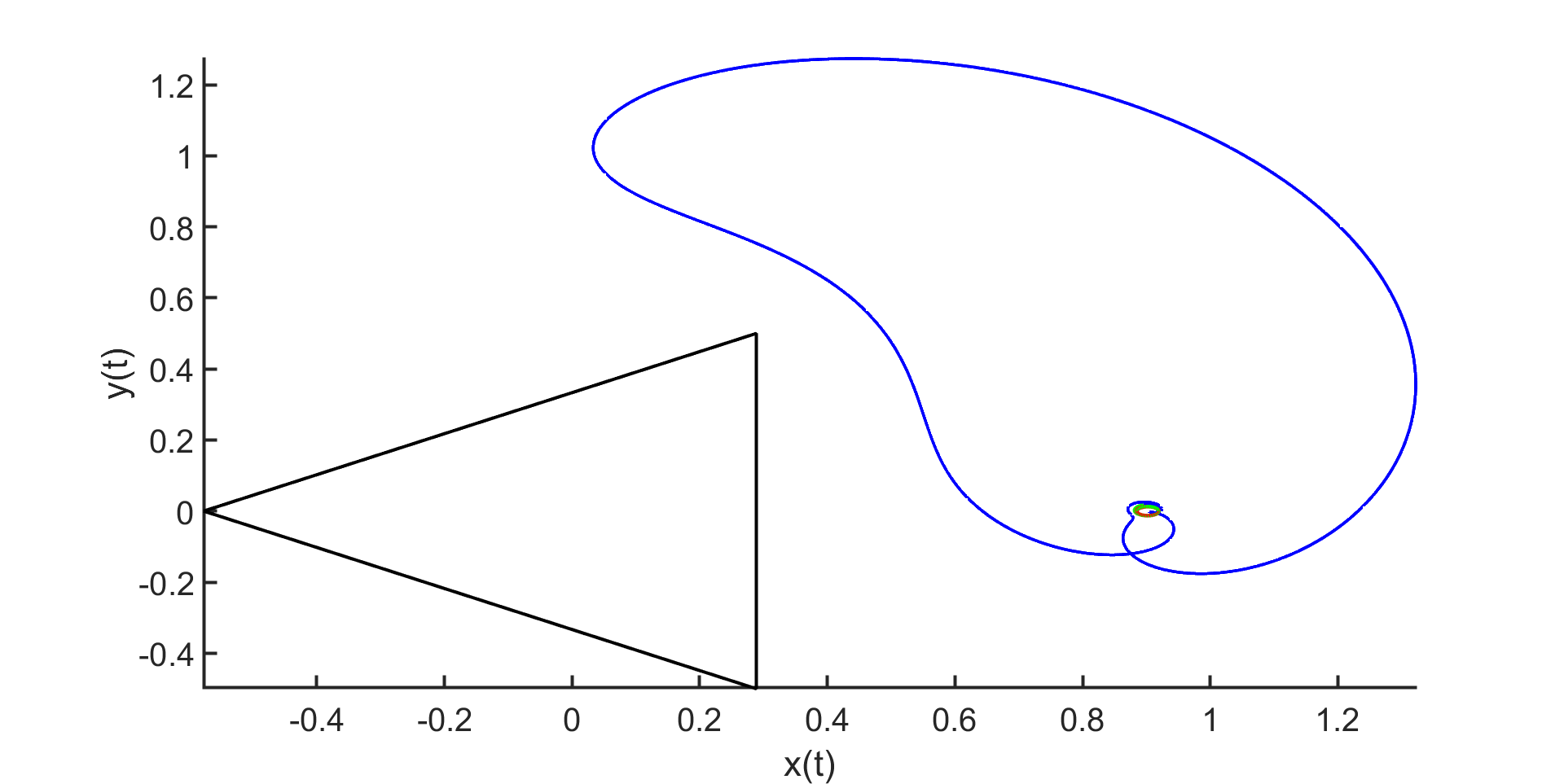}}}
 	\subfigure{{\includegraphics[width=.48\textwidth]{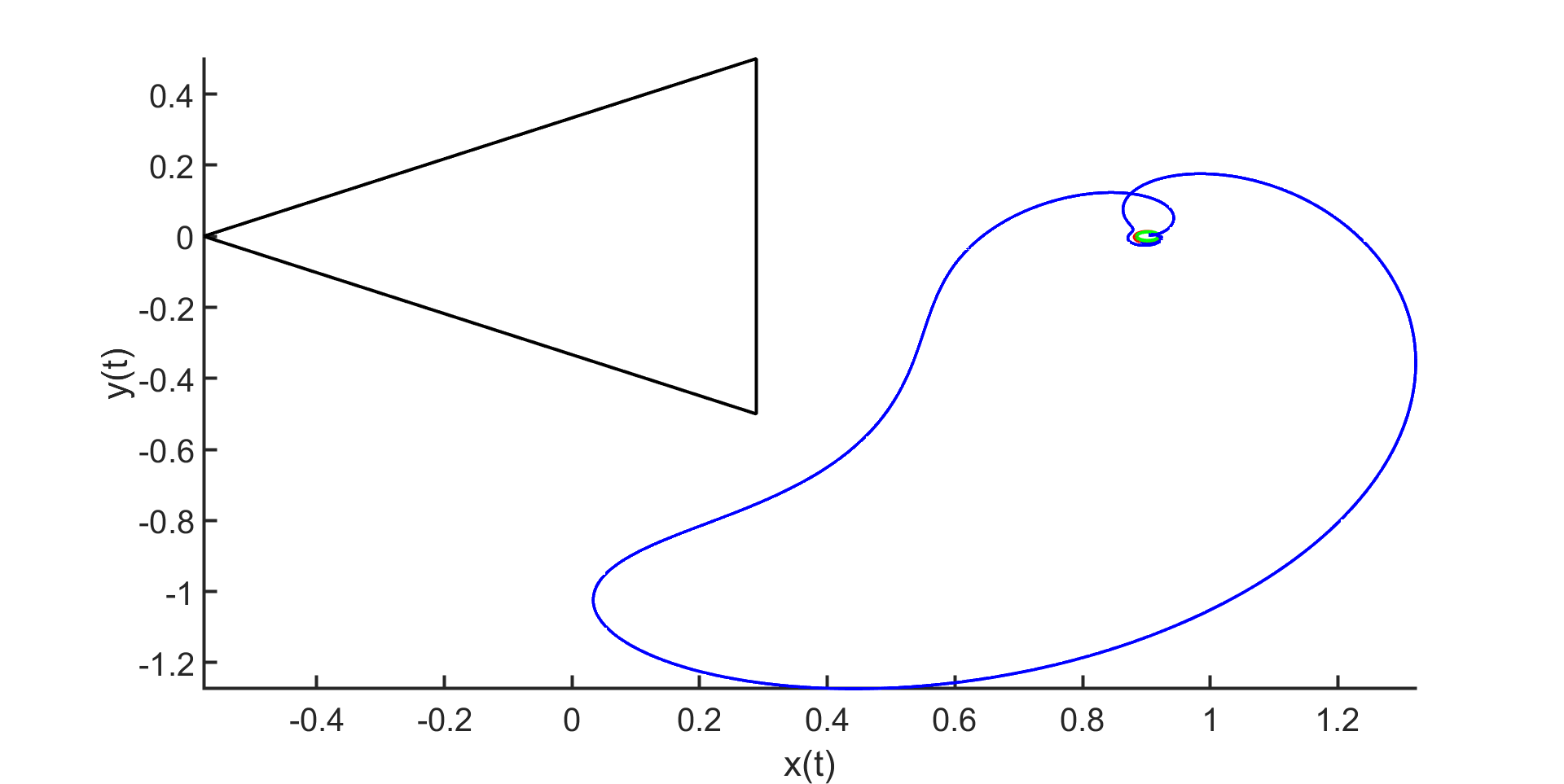}}}\\    
    \subfigure{{\includegraphics[width=.48\textwidth]{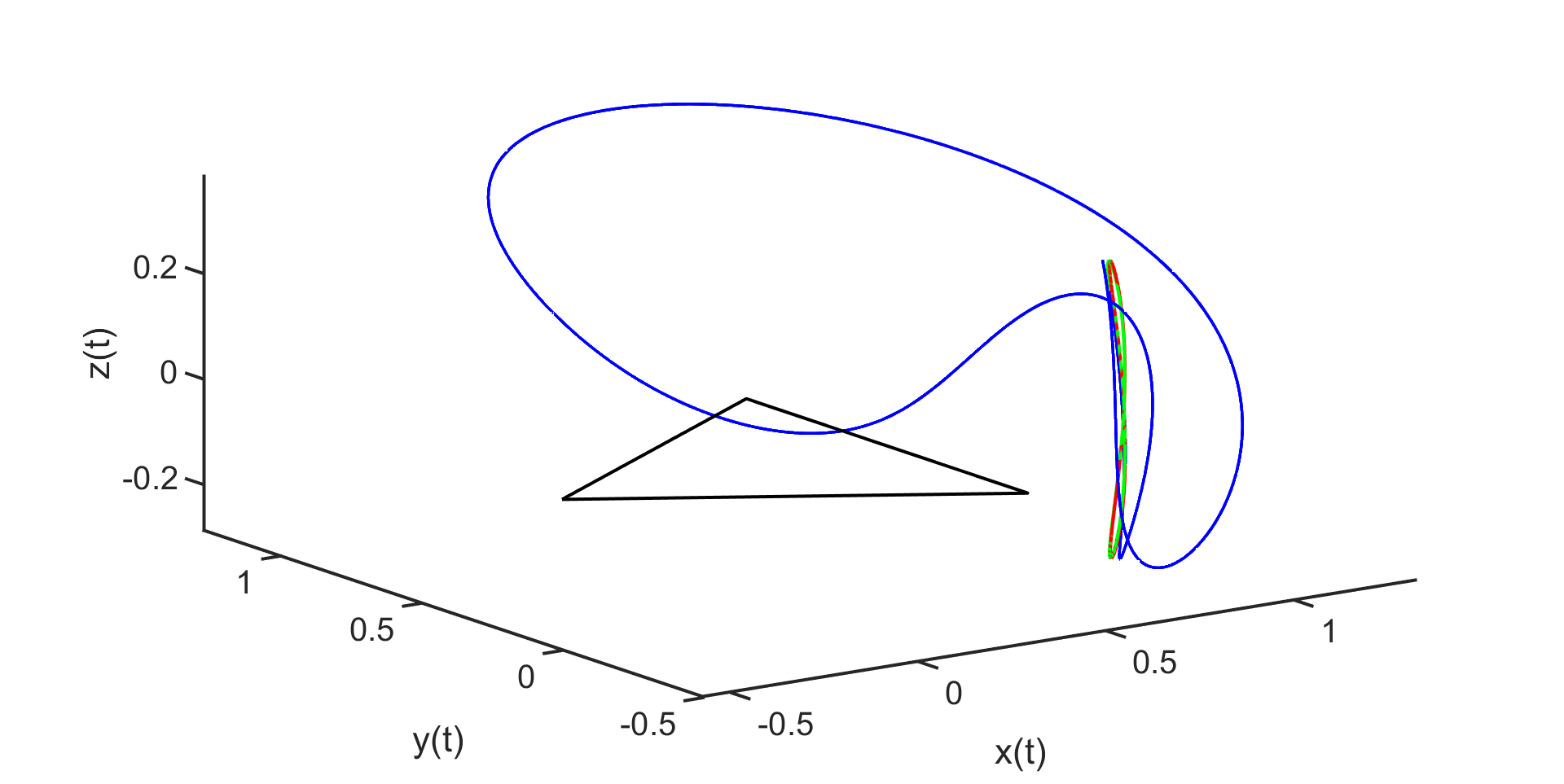}}}
    \subfigure{{\includegraphics[width=.48\textwidth]{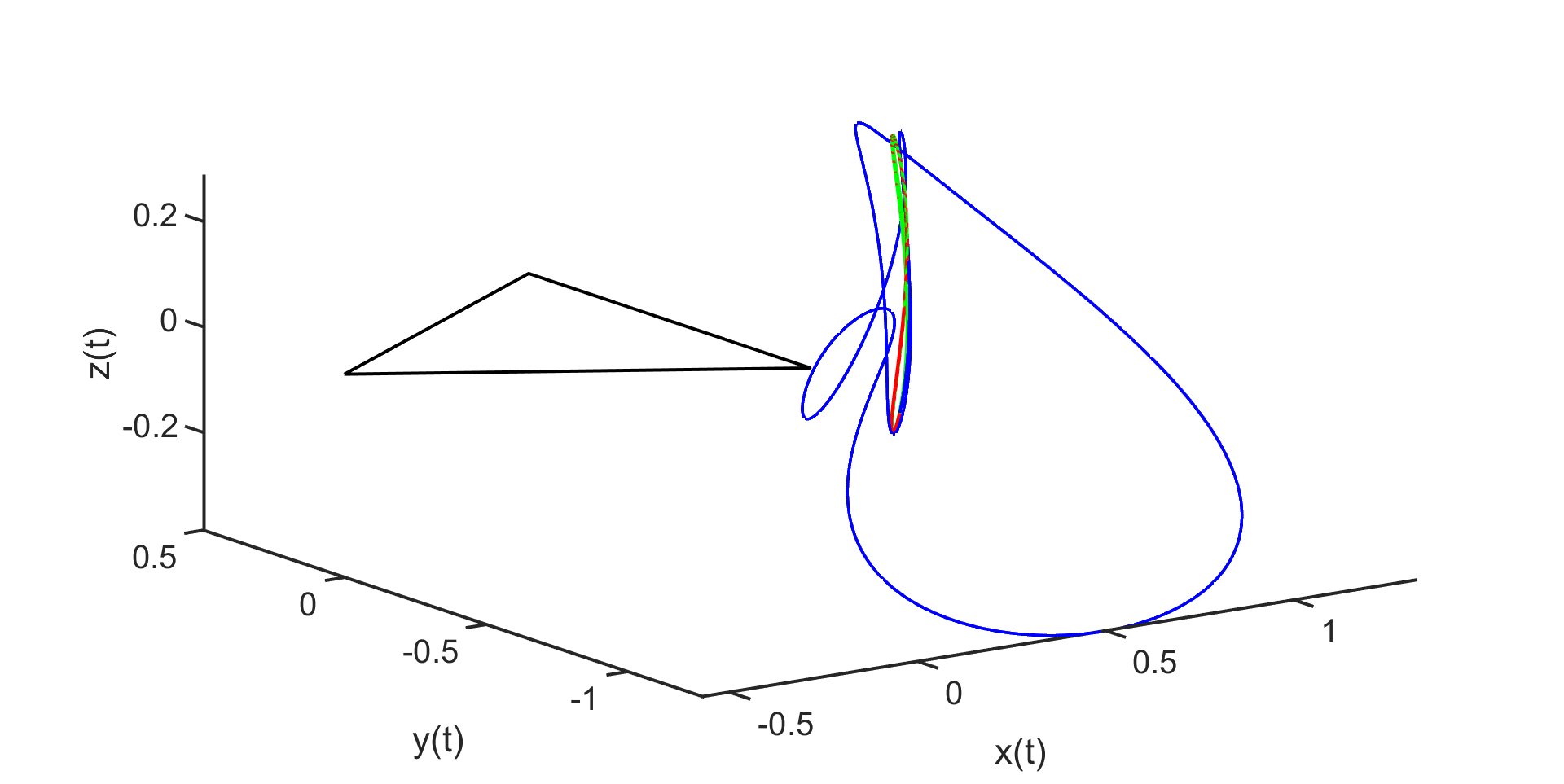}}}
\caption{ Homoclinic connections with trivial winding about all three primary bodies, 
for a spatial periodic orbit in the vertical Lyapunov family 
at $\mathcal{L}_5$ in the triple Copenhagen problem. 
Note that the connections are entirely in the ``outer system'', i.e. 
the orbits do not intersect the equilateral triangle formed by the primaries.
Error bounds are as given in Theorem \ref{thm:ValidationL5}, and 
projections displayed and energies are as described in the caption of 
Figure \ref{fig:homoclinicL5A-B}
}\label{fig:homoclinicL5E-F}
\end{figure}

\begin{figure}[!t]
	\subfigure{{\includegraphics[width=.48\textwidth]{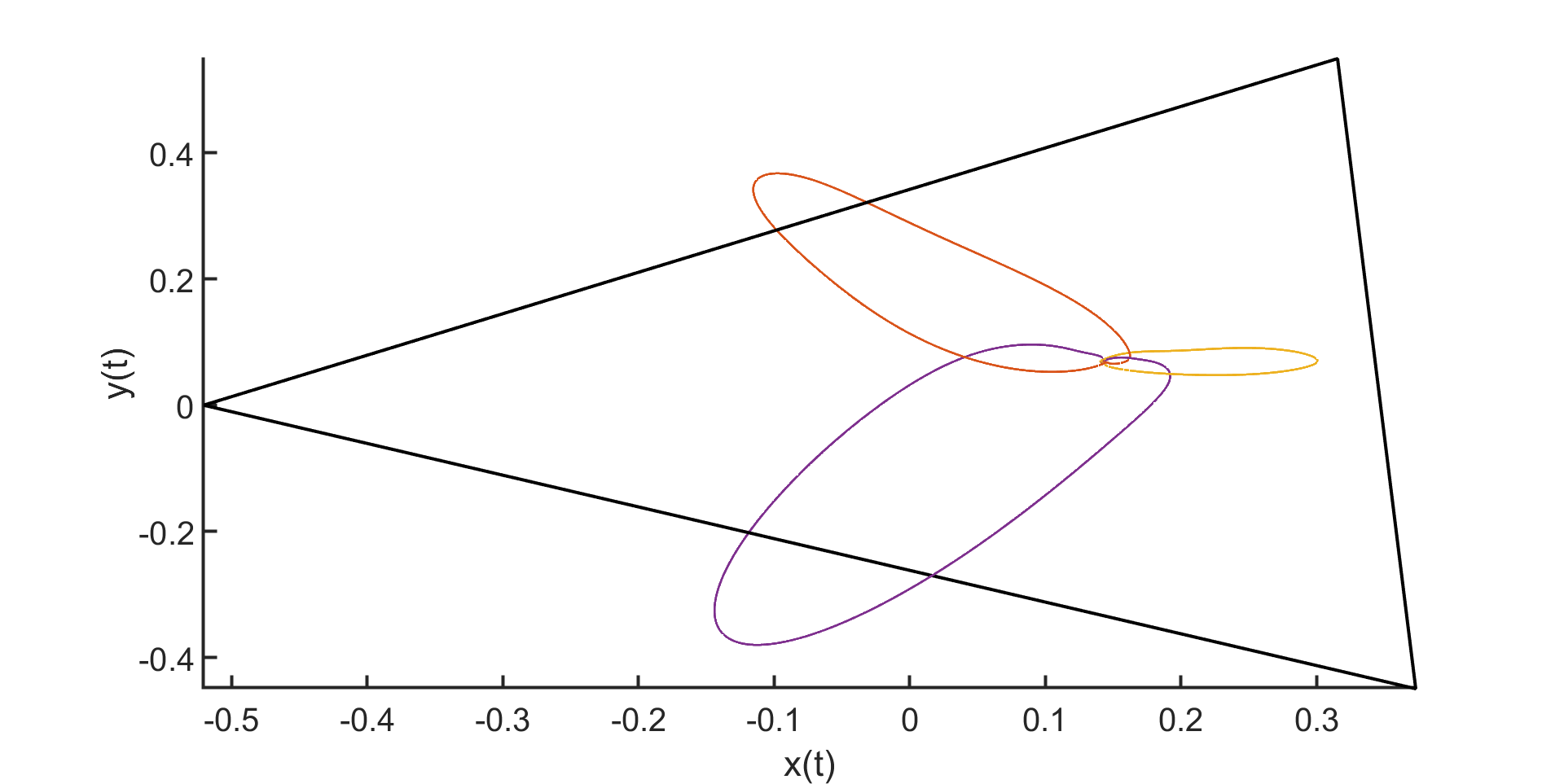}}} 
 	\subfigure{{\includegraphics[width=.48\textwidth]{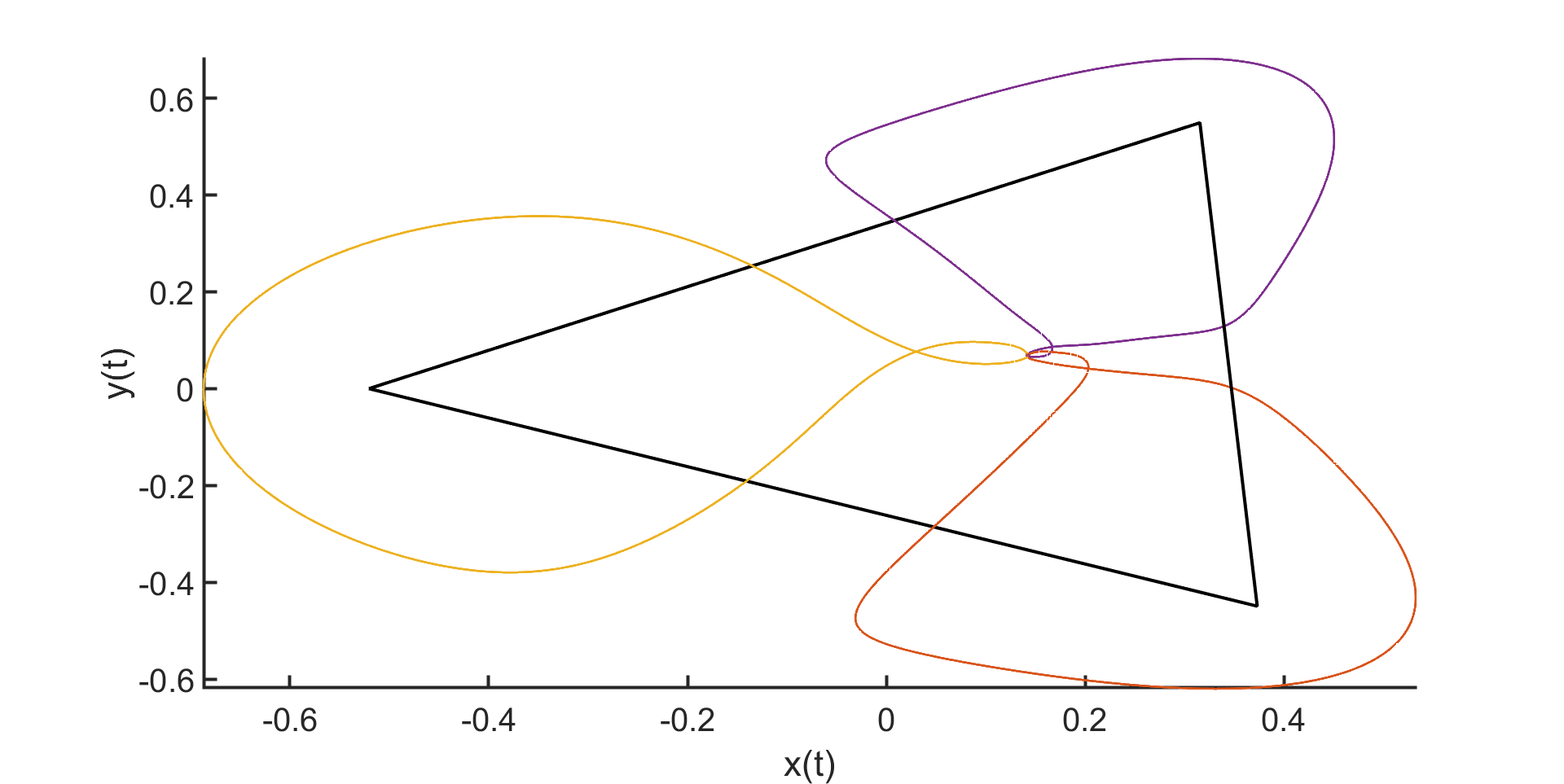}}}\\    
    \subfigure{{\includegraphics[width=.48\textwidth]{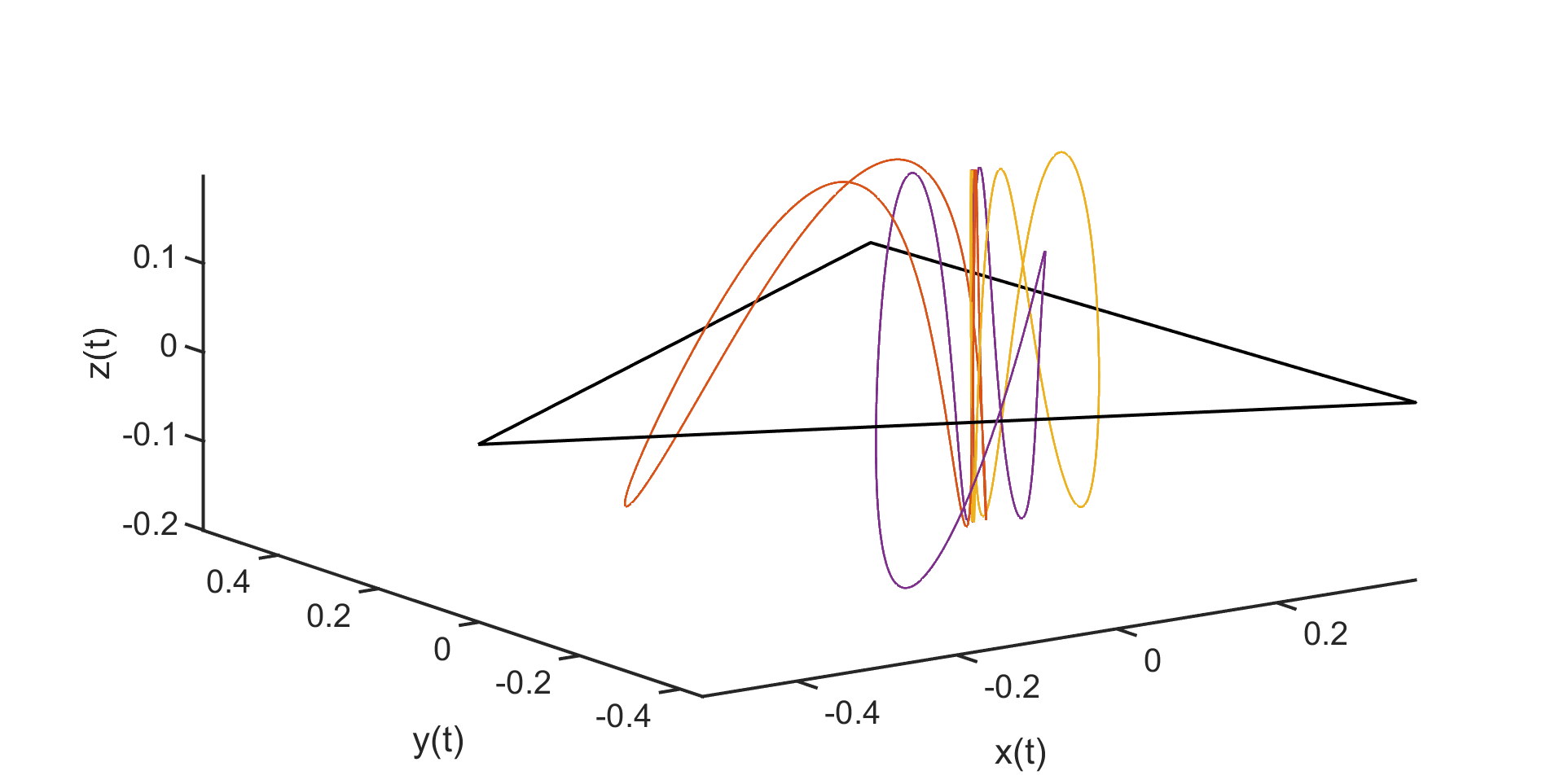}}} 
    \subfigure{{\includegraphics[width=.48\textwidth]{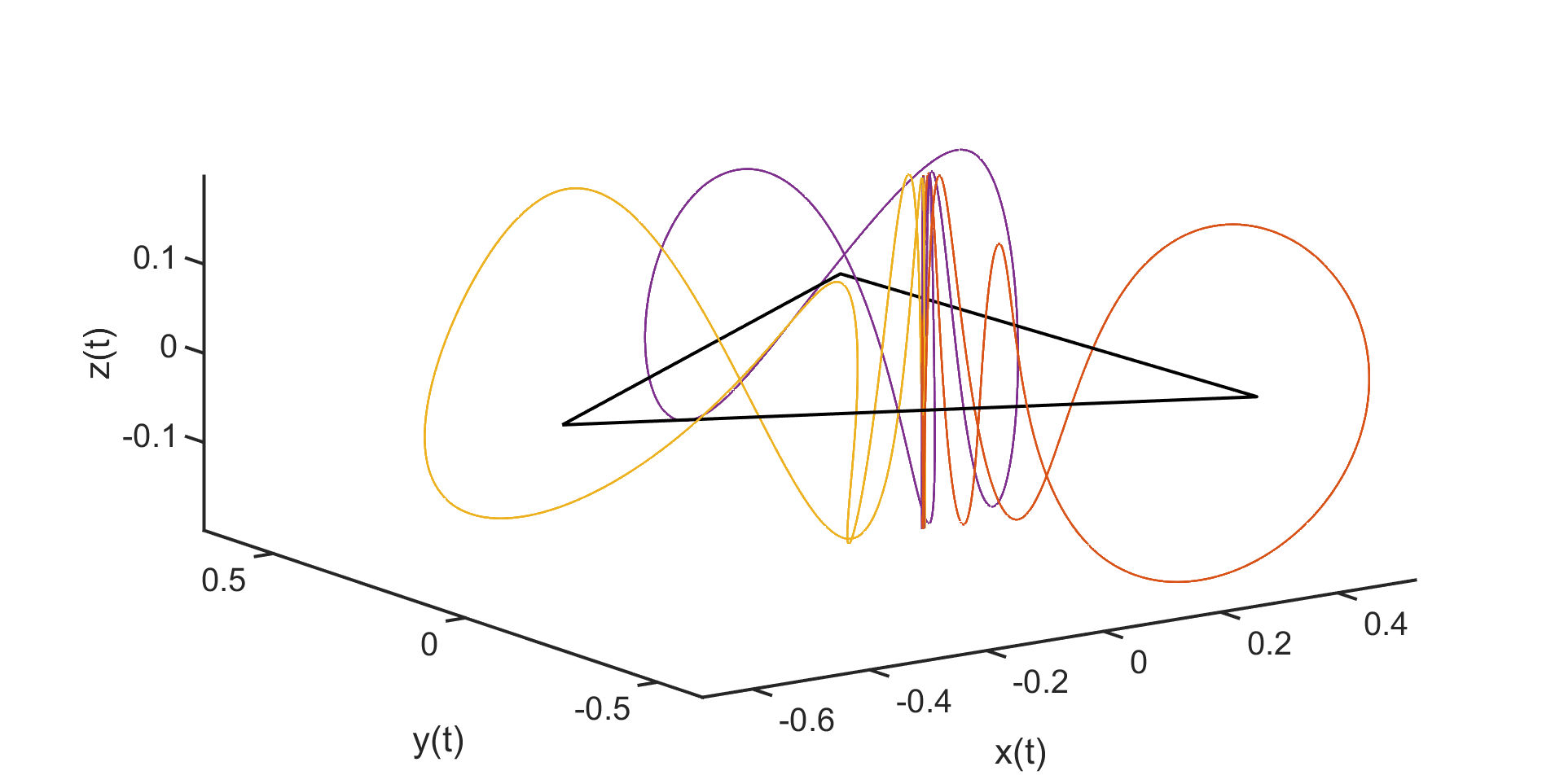}}}
\caption{ Homoclinic connections to a spatial periodic orbit from the vertical Lyapunov family 
at $\mathcal{L}_0$ with $m_1=0.4$, $m_2=0.33$, and $m_3= 0.27$. The first row displays the connection seen from above, while the bottom row exhibits that the connections are indeed out of plane. Each of the connections has been validated with the help of Theorem \ref{thm:ContractionMapping}, so that a true solution lays in a neighborhood of the approximation displayed. The size of said neighborhood, thus a bound on the $C^0$ norm, is specified for each case in Theorem \ref{thm:ValidationL0}. The computation were executed with a fixed frequency for the periodic orbit, hence the exact energy isn't provided by the validation algorithm. We note however that the energy is of approximately $3.26$.      
}\label{fig:homoclinicL0}
\end{figure}

%
%
%
%
%

{\tiny

\appendix

\section{Cauchy Bounds for analytic functions: Fourier and Taylor cases}\label{sec:CauchyBound}
%
The following lemmas facilitate control of all necessary partial derivatives 
the the Fourier-Taylor series appearing in our computer assisted proofs.
The quite standard proofs are included for the sake of completeness.  

\begin{lemma}[Bounds for Analytic Functions given by Fourier Series] 
\label{lem:cauchyBounds_Fourier}
Suppose that $\{a_n \}_{n \in \mathbb{Z}}$ is a two sided sequence of 
complex numbers, that $\omega > 0$, and that $\| a\| _{1,\nu} < \infty$, 
for some $\nu>1$. Let 
\[
f(z) := \sum_{n \in \mathbb{Z}} a_n e^{i \omega k z}.
\]
For $r > 0$ let 
\[
A_r := \{z \in \mathbb{C} \, | \, |\mbox{{\em imag}}(z)| < r \},
\]
denote the open complex strip of width $r$.
\begin{itemize}
\item[(1)] \textbf{$\ell_\nu^1$ bounds imply $C^0$ bounds:} 
The function $f$ is analytic on the strip $A_r$ with $r = \ln(\nu)/\omega$,
continuous on the closure of $A_r$, and satisfies
\[
\sup_{z \in A_r} |f(z)| \leq \| a\|_{1,\nu}.
\]
Moreover $f$ is $T$-periodic with $T = 2 \pi/\omega$. 
\item[(II)] \textbf{Cauchy Bounds:} Let $0 < \sigma < r$ and 
\[
\tilde \nu = e^{\omega (r - \sigma)}.
\]
Define the sequence $b = \{b_n \}_{n \in \mathbb{Z}}$ by 
\[
b_n = i \omega n a_n.
\]
Then 
\[
\| b \|_{\tilde \nu}^1 \leq \frac{1}{e \sigma} \|a\|_{1,\nu},
\]
and
\[
\sup_{z \in A_{r - \sigma}} |f'(z)| \leq \frac{1}{e \sigma} \| a\|_{1,\nu}.
\]
\end{itemize}
\end{lemma}

\begin{proof}
For $\nu > 1$, let $r = \ln(\nu)/\omega$ (i.e. $\nu  = e^{\omega r}$) and consider $z \in A_r$.
We have that 
\begin{align*}
|f(z)| &\leq \sum_{n \in \mathbb{Z}} |a_n| \left| \left(e^{i \omega z}\right)^{n} \right| \\
&\leq \sum_{n \in \mathbb{Z}} |a_n| \left(e^{\omega \left|\mbox{imag}(z) \right|}\right)^{|n|} \\
&\leq \sum_{n \in \mathbb{Z}} |a_n| \nu^{|n|} \\
&= \| a \|_{1,\nu}.
\end{align*}
It follows that $f$ is analytic as the Fourier series converges absolutely and uniformly in $A_r$.
Continuity at the boundary of the strip also follows from the absolute summability of the series. 
From the fact that $f$ is analytic in $A_r$ it follows that the derivative $f'(z)$ exists 
for any $z$ in the interior of $A_r$, and that 
\[
f'(z) = \sum_{n \in \mathbb{Z}} i \omega n a_n e^{i \omega n z} = \sum_{n \in \mathbb{Z}} b_n e^{i \omega n z}.
\]
However the fact that $f$ is uniformly bounded on $A_r$ does not imply uniform 
bounds on $f'$.  Indeed it may be that $f'$ has singularities at the boundary.  

In order to obtain uniform bounds on the derivative we give up a portion of the width of the 
domain. More precisely we 
consider the supremum of $f'$ on the strip $A_{r - \sigma}$ with $0 < \sigma < r$.
Letting $\nu = e^{\omega r}$, so that $\tilde \nu = e^{\omega(r - \sigma)}$, 
define the function $s \colon \mathbb{R}^+ \to \mathbb{R}$ by 
\[
s(x) := x \alpha^x,
\]
with $\alpha = \tilde{\nu} / \nu = e^{-\omega \sigma} < 1$.
Note that $s(x) \geq 0$ for all $x \geq 0$,  $s(0) = 0$, and that $s(x) \to 0$ as $x \to \infty$.
Moreover $s$ is bounded and attains its maximum at 
\[
\hat x = \frac{1}{\omega \sigma}, 
\]
as can be seen by computing the critical point of $s$.  
It follows that 
\[
s(x) \leq s(\hat x) = \frac{1}{e \omega \sigma}.
\]
From this we have 
\begin{eqnarray*}
\| b\|_{1,\tilde \nu} &=& \sum_{n \in \mathbb{Z}} |b_n| \tilde{\nu}^{|n|} \\
&=& \sum_{n \in \mathbb{Z}} \omega |n| |a_n| \tilde{\nu}^{|n|} \frac{\nu^{|n|}}{\nu^{|n|}} \\
&=& \sum_{n \in \mathbb{Z}} \omega |n| \left(\frac{\tilde \nu}{\nu}  \right)^{|n|} |a_n| \nu^{|n|}  \\
&=& \sum_{n \in \mathbb{Z}}  \omega s(|n|) |a_n| \nu^{|n|} \\
&\leq& \sum_{n \in \mathbb{Z}}  \omega \frac{1}{e \omega \sigma} |a_n| \nu^{|n|} \\ 
&=& \frac{1}{e \sigma} \| a\|_{1,\nu}.
\end{eqnarray*} 
It follows by $(I)$ that if $z \in A_{r - \sigma}$ then
\[
|f'(z)| \leq  \|b\|_{1,\tilde \nu} \leq \frac{1}{e \sigma} \| a\|_{1,\nu},
\]
as desired.
\end{proof}

\begin{lemma}[Bounds for Analytic Functions given by Taylor Series]
 \label{lem:cauchyBounds_Taylor}
Let $\nu > 0$ and suppose that $\{a_n \}_{n \in \mathbb{N}}$ is a one sided 
sequence of complex numbers 
with  
\[
\| a\| _{1,\nu} = \sum_{n =0}^\infty |a_n| \nu^n  < \infty.
\]
Define 
\[
f(z) := \sum_{n=0}^\infty a_n z^n,
\]
and let
\[
D_\nu := \{z \in \mathbb{C} \, | \, |z| < \nu \},
\]
denote the complex disk of radius $\nu > 0$ centered at the origin.
\begin{itemize}
\item[(1)] \textbf{$\ell_\nu^1$ bounds imply $C^0$ bounds:} 
The function $f$ is analytic on the disk $D_\nu$,
continuous on the closure of $D_\nu$, and satisfies
\[
\sup_{z \in D_\nu} |f(z)| \leq \| a\|_{1,\nu}.
\] 
\item[(II)] \textbf{Cauchy Bounds:} Let $0 < \sigma \leq 1$ and 
\[
\tilde \nu = \nu e^{- \sigma}.
\]
Then 
\[
\sup_{z \in D_{\tilde \nu}} |f'(z)| \leq \frac{1}{\nu \sigma} \|a\|_{1,\nu}.
\]
\end{itemize}
\end{lemma}

\begin{proof}
The proof is similar to the proof of Lemma \ref{lem:cauchyBounds_Fourier}, 
the difference being the 
estimate of the derivative.  Since $f$ is analytic in $D_\nu$ we have
\[
f'(z) = \sum_{n=1}^\infty n a_n z^{n-1},
\]
for all $z \in B_\nu$, and again we will trade some domain for  
uniform bounds on derivatives.  So, 
choose $0 < \sigma \leq 1$ and define $\tilde \nu = \nu e^{-\sigma}$.  As in the 
Fourier case, define the function $s \colon [0, \infty) \to \mathbb{R}$ by 
\[
s(x) := x e^{- \sigma x}.
\]
We have that $s$ is positive and
\[
s(x) \leq \frac{1}{e \sigma}.
\]
Then for any $z \in D_{\tilde \nu}$ we have
\begin{eqnarray*}
|f'(z)| &=& \sum_{n=1}^\infty n |a_n| |z|^{n-1} \\
&\leq& \sum_{n=1}^\infty n |a_n|  \frac{\nu}{\nu}  | \nu e^{-\sigma} |^{n-1}  \\
&\leq& \sum_{n=1}^\infty \left( \frac{e^{\sigma}}{\nu} n e^{-\sigma n}  \right) |a_n| \nu^n \\
&\leq& \sum_{n=1}^\infty \left( \frac{e}{\nu} n e^{-\sigma n}  \right) |a_n| \nu^n \\
&\leq& \sum_{n=0}^{\infty} \frac{e}{\nu} s(n) |a_n|\nu^n \\
&\leq& \sum_{n=0}^\infty \frac{e}{\nu} \frac{1}{e \sigma} |a_n| \nu^n \\
&=& \frac{1}{\nu \sigma} \|a\|_{1,\nu}^1.
\end{eqnarray*}
\end{proof}

We now to consider these lemmas to the context of parameterized periodic manifolds 
for the CRFBP. Let $\omega_m$ denote the Frequency of the periodic orbit of interest, 
and $\nu_m$ denote the weight taken for the $\ellnu$ norms from each step of the validation 
discussed in Section  \ref{sec:TailArgument}. The weight in the Taylor direction is $1$ by construction of the norm in $X_2^\infty$. 
Lemma is \ref{lem:cauchyBounds_Taylor} is applied with $\nu=1$, for fixed $r^\star$, 
to the endpoint of the domain of the polynomial bound in Theorem 
\ref{thm:ContractionMapping}.  Choose $\sigma_m$ such that
\[
r^\star < \frac{\ln(\nu_m)}{\omega_m} - \sigma_m.
\]
It follows that for all radius $r \leq r^\star$ and for $b \in \mathbb{R}$ with $|b| = r$, we have that
\begin{equation}\label{eq:CauchyFourier}
\left\| \frac{\partial}{\partial\theta} E(\bar \theta,\bar\phi)b\right\|_\infty = \max_{1\leq i \leq 9} \left| \frac{\partial}{\partial\theta} E^i(\bar \theta,\bar\phi)b\right|\leq \left(\frac{r_m}{e\sigma_m}\right) r.
\end{equation}
The other partial derivative is bound similarly, and it follows from the chain rule that
\begin{align*}
\frac{\partial}{\partial \phi} E(\bar\theta,\bar\phi) &= \frac{\partial}{\partial \sigma_1} E(\bar\theta,\bar \sigma_1,\bar \sigma_2) \frac{\partial \sigma_1}{\partial \phi} +\frac{\partial}{\partial \sigma_2} E(\bar\theta,\bar \sigma_1,\bar \sigma_2) \frac{\partial \sigma_2}{\partial \phi}, \\
\end{align*}
where $\frac{\partial \sigma_1}{\partial \phi}$ and $\frac{\partial \sigma_2}{\partial \phi}$ depend 
on the choice of the function $\sigma(\rho,\phi) = (\sigma_1,\sigma_2)$. Note however that in this work, the 
partial derivatives of interest are bounded by a constant $\rho$ that can be fixed in the construction of the BVP. Again, Lemma \ref{lem:cauchyBounds_Taylor} is 
applied with $\nu=1$ as the weight in the Taylor direction. Hence, whenever $\rho$ 
satisfies 
\[
\rho < \tilde \nu = e^{\ln(0.99)} \bydef e^{-s}
\]
it follows that
\begin{equation}\label{eq:CauchyTaylor}
\left\| \frac{\partial}{\partial \phi} E(\bar\theta,\bar\phi)b \right\|_{\infty} \leq
 \left( 2\rho \frac{r_m}{s} \right) r,
\end{equation}
for any $\|b\|_\infty = r$. Finally, note that whenever $\rho_u$ is a variable of the problem
it is possible to use Lemma \ref{lem:cauchyBounds_Taylor} to obtain 
\begin{equation}\label{eq:CauchyRho}
\left\| \frac{\partial}{\partial \rho_u} E(\bar\theta,\bar\phi)b \right\|_{\infty} \leq \left( 2\frac{r_m}{s} \right) r.
\end{equation}
This follows from the chain rule, and we note that this estimates require that $L$ is fixed large 
enough so that $\bar \rho_u + r < \tilde \nu$. This verification ensures the validity of the estimates 
in a neighborhood of the approximation. The variable $\rho_s$ is fixed and the stable 
parameterization will not require this last estimate. 

%
%
%
%
%
%
%

\section{Numerical bounds on the norms of each $D_{ij}^\infty$} \label{Ad:Dbounds}

It was already shown in Section \ref{Sec:RewritingTail} that
\begin{align*}
\left\| \mathcal{L}_{\alpha}^{-1}\circ D^\infty(h) \right\|_{(\ellnu)^9} 
&\leq \max_{1\leq i \leq 9} \sum_{|k|<K} 
\left|\frac{\left(D^\infty h\right)_k^i}{-\im \omega k -\langle \alpha,\lambda \rangle}\right|\nu^{|k|} +d_\alpha^\infty \sum_{j=1}^9\left\|D_{ij}^\infty \right\|_{\mathcal{B}\left(\ellnu \right)} \\
\end{align*}
The cases $i=1,3,5$ include only a single linear term, and is non-zero only in the 
case $|k|\geq K$, so that the first sum vanishes in these cases and we set
\[
\|D_{12}^\infty \|_{\mathcal{B}(\ellnu)} = \|D_{34}^\infty \|_{\mathcal{B}(\ellnu)} = \|D_{35}^\infty \|_{\mathcal{B}(\ellnu)} = 1.
\]
The other cases contain convolution products.  For example, for all $k \in \mathbb{Z}$
 we have that
\begin{align*}
(D_{21}^\infty h^1)_k &=  \sum_{i=1}^3 m_i\left[ (h^\infty)^1 \star (a_0^K)^{6+i}
 \star (a_0^K)^{6+i}\star (a_0^K)^{6+i} \right]_k + 3m_i\left[ h^1 
 \star (a_0^\infty)^{6+i} \star (a_0^K)^{6+i}\star (a_0^K)^{6+i} \right]_k \\
& +\sum_{i=1}^3 3m_i\left[ h^1 \star (a_0^\infty)^{6+i} 
\star (a_0^\infty)^{6+i}\star (a_0^K)^{6+i} \right]_k +m_i\left[ h^1 
\star (a_0^\infty)^{6+i} \star (a_0^\infty)^{6+i}\star (a_0^\infty)^{6+i} \right]_k \\
\end{align*}
To compute a bound on this term, we note that 
\[
\| (a_0^K)^i \|_{1,\nu} \leq \| \bar a_0^i \|_{1,\nu} + r_0 \bydef n^i, \quad \forall i=1,2,\hdots,9,
\]
so that
\[
\left\| D_{21}^\infty h^1 \right\|_{1,\nu} \leq \sum_{i=1}^3 \left[ \sum_{|k|<K}\left|\left[ (h^\infty)^1 \star (a_0^K)^{6+i} \star (a_0^K)^{6+i}\star (a_0^K)^{6+i} \right]_k\right|\nu^{|k|} +3m_i r_0 n^{6+i} n^{6+i} + 3m_i r_0r_0 n^{6+i} +m_i r_0 r_0 r_0 \right]. 
\]


\section{The $Z_1$ bound for Chebyshev expansion} \label{Add:Z1bound}
In this section we present a table containing bounds for all non-zero 
terms in the $Z_1$ estimates. 
Throughout the discussion, $h^{i,j}$ denotes an element of $\ellnu$ 
with norm $r$ 
and $h^{i,j}_\infty = h^{i,j} - \pi_m{h^{i,j}}$, 
so that $(h^{i,j}_\infty)_k = 0$ for $k< m-1$.   Recall also that $\bar a^{i,j}$ is the $i-$th component of the $j-$th sequence of Chebyshev coefficients 
of the solution to \eqref{BVP:Chebyshev}. 
Some $d^i$ are simplified due to the fact that that we always choose  
$\bar \beta =0$.  Then any terms which are constant or higher order 
with respect to $\beta$ do not appear in the derivatives, and hence are 
not present in the bounds (only first order terms with respect to 
$\beta$ are present).

For the sake of completeness, one case is illustrated. Consider 
\begin{align*}
d_k^{2,j}= &\Lambda_k \left( h_l\left( \bar\beta \bar a^{2,j} +2 \bar a^{4,j} +\bar a^{1,j} \right) + \bar L\left( h_\beta \bar a^{2,j} +\bar\beta h^{2,j} +2 h^{4,j} +h^{1,j} \right) \right) \\
&+ h_L \Lambda_k \left( \sum_{j=1}^3 m_j (\bar a^{1,j} -x_1)\ast \bar a^{6+i,j}\ast \bar a^{6+i,j}\ast \bar a^{6+i,j} \right) \\
&+ \bar L \Lambda_k\left( \sum_{l=1}^3 m_j \left( h^{1,j} \ast \bar a^{6+l,j}\ast \bar a^{6+l,j}\ast \bar a^{6+l,j} +3(\bar a^{1,j} -x_l) \ast h^{6+l,j}\ast \bar a^{6+l,j}\ast \bar a^{6+l,j} \right) \right)
\end{align*}
Finally, note that all other cases are zero. So that the six scalar components are zero, 
and $d_0^{i,j}=0$ for all $1 \leq i \leq 9$, $1\leq j \leq M$. The other cases are in the 
two tables.
   
\begin{table}[H]
\caption{The sequence $d \in X$ for all $1\leq j \leq M$, cases $1\leq i \leq 6$}
\begin{tabular}{|l |l |}
\hline
$i$ and $k$  & $d_k^{6 +9(j-1) +i}$ \\ 
\hline
$1$, $1\leq k < m$ & $0$ \\
$1$, $k \geq m$ & $ \bar L \Lambda_k (h^{2,j})$ \\
\hline
$2$, $1\leq k < m$ & $\bar L \displaystyle \sum_{l=1}^3 m_l\Lambda_k
\left(h^{1,j}_{\infty} \ast \bar a^{6+l,j}\ast \bar a^{6+l,j}\ast \bar a^{6+l,j} 
+3(\bar a^{1,j} -x_l) \ast  h^{6+l,j}_{\infty}\ast \bar a^{6+l,j}\ast \bar a^{6+l,j} \right)$ \\
$2$, $k\geq m$ & $ \bar L \Lambda_k(2h^{4,j} +h^{1,j}) 
+h_L \Lambda_k \left( \displaystyle \sum_{j=1}^3 m_j (\bar a^{1,j} -x_1)\ast
 \bar a^{6+i,j}\ast \bar a^{6+i,j}\ast \bar a^{6+i,j} \right)$ \\
 & $+\bar L \Lambda_k\left( \displaystyle \sum_{l=1}^3 m_j \left( h^{1,j} \ast 
 \bar a^{6+l,j}\ast \bar a^{6+l,j}\ast \bar a^{6+l,j} +3(\bar a^{1,j} -x_l) 
 \ast h^{6+l,j}\ast \bar a^{6+l,j}\ast \bar a^{6+l,j} \right) \right)$ \\
\hline
$3$, $1\leq k < m$ & $0$ \\
$3$, $k \geq m$ & $ \bar L \Lambda_k (h^{4,j})$ \\
\hline
$4$, $1\leq k < m $ & $\bar L \displaystyle
 \sum_{l=1}^3 m_l\Lambda_k\left(h^{3,j}_{\infty} \ast \bar a^{6+l,j}\ast
  \bar a^{6+l,j}\ast \bar a^{6+l,j} +3(\bar a^{3,j} -y_l) \ast h^{6+l,j}_\infty
   \ast \bar a^{6+l,j}\ast \bar a^{6+l,j} \right)$ \\
$4$, $k\geq m$ & $ \bar L \Lambda_k(-2h^{2,j} +h^{3,j}) + 
h_L \Lambda_k \left( \displaystyle \sum_{j=1}^3 m_j (\bar a^{3,j} -y_1)\ast
 \bar a^{6+i,j}\ast \bar a^{6+i,j}\ast \bar a^{6+i,j} \right)$ \\
 & $+\bar L \Lambda_k\left( \displaystyle \sum_{l=1}^3 m_j \left( h^{3,j} 
 \ast \bar a^{6+l,j}\ast \bar a^{6+l,j}\ast \bar a^{6+l,j} +3(\bar a^{3,j} -y_l)
  \ast h^{6+l,j}\ast \bar a^{6+l,j}\ast \bar a^{6+l,j} \right) \right)$ \\
\hline
$5$, $1\leq k < m$ & $0$ \\
$5$, $k \geq m$ & $ \bar L \Lambda_k (h^{6,j})$ \\
\hline
$6$, $1\leq k < m $ & $\bar L \displaystyle \sum_{l=1}^3 m_l\Lambda_k\left(h^{5,j}_\infty 
\ast \bar a^{6+l,j}\ast \bar a^{6+l,j}\ast \bar a^{6+l,j} +3(\bar a^{5,j} -z_l) \ast h^{6+l,j}_\infty
 \ast \bar a^{6+l,j}\ast \bar a^{6+l,j} \right)$ \\
$6$, $k\geq m$ & $ h_L \Lambda_k \left( \displaystyle \sum_{j=1}^3 m_j (\bar a^{5,j} -z_1)
\ast \bar a^{6+i,j}\ast \bar a^{6+i,j}\ast \bar a^{6+i,j} \right)$ \\
 & $+\bar L \Lambda_k\left( \displaystyle \sum_{l=1}^3 m_j \left( h^{5,j} \ast \bar a^{6+l,j}
 \ast \bar a^{6+l,j}\ast \bar a^{6+l,j} +3(\bar a^{5,j} -z_l) \ast h^{6+l,j}\ast \bar a^{6+l,j}
 \ast \bar a^{6+l,j} \right) \right)$ \\
\hline
\end{tabular}
\end{table}

\begin{table}[H]
\caption{The sequence $d \in X$ for all $1\leq j \leq M$, case $i=7,8,9$}
\begin{tabular}{|l |l |}
\hline
$1\leq k < m $ & $~\bar L \Lambda_k\left( \displaystyle \sum_{l=1}^6 h^{l,j}_\infty 
\ast \bar a^{i,j}\ast \bar a^{i,j}\ast \bar a^{i,j} +3(\bar a^{1,j} -x_1) \ast \bar a^{2,j} \ast h^{i,j}_\infty 
\ast \bar a^{i,j}\ast \bar a^{i,j} \right)$ \\
& $ +\bar L \Lambda_k\bigg( 3(\bar a^{3,j} -y_l) \ast \bar a^{4,j} \ast h^{i,j}_\infty \ast \bar a^{i,j}\ast 
\bar a^{i,j} + 3(\bar a^{5,j} -z_l) \ast \bar a^{6,j} \ast h^{i,j}_\infty \ast \bar a^{i,j}\ast \bar a^{i,j} \bigg) $ \\
\hline
$k \geq  m $ & $~\bar L \Lambda_k\left( \displaystyle \sum_{l=1}^6 h^{l,j} \ast 
\bar a^{i,j}\ast \bar a^{i,j}\ast \bar a^{i,j} +3(\bar a^{1,j} -x_1) \ast \bar a^{2,j} 
\ast h^{i,j} \ast \bar a^{i,j}\ast \bar a^{i,j} \right)$ \\
& $ +\bar L \Lambda_k\bigg( 3(\bar a^{3,j} -y_l) \ast \bar a^{4,j} \ast h^{i,j} \ast 
\bar a^{i,j}\ast \bar a^{i,j} + 3(\bar a^{5,j} -z_l) \ast \bar a^{i,j} \ast h^{i,j}\ast \bar a^{i,j}\ast \bar a^{i,j} \bigg) $ \\
& $ +h_L \Lambda_k\bigg( \left((\bar a^{1,j} -x_l) \ast \bar a^{2,j} +(\bar a^{3,j} -y_l)
 \ast \bar a^{4,j} +(\bar a^{5,j} -z_l) \ast \bar a^{6,j} \right) \ast \bar a^{i,j}\ast \bar a^{i,j}\ast \bar a^{i,j} \bigg) $ \\
\hline
\end{tabular}
\end{table}

}

\bibliographystyle{plain}
\bibliography{references} 
\end{document}